\documentclass[a4paper,11pt,reqno, english]{amsart}

\usepackage{anysize,color}
\usepackage[utf8]{inputenc}
\usepackage[dvipsnames]{xcolor}
\usepackage[colorlinks=true,urlcolor=blue,linkcolor={magenta},citecolor={orange}]{hyperref}  
\usepackage{float}
\usepackage{amsfonts,amssymb,amsmath}
\usepackage{amsthm}    
\usepackage{url}
\usepackage{paralist}
\usepackage{tikz}
    \usetikzlibrary{decorations.markings,arrows,shapes.callouts}
\usepackage{rotating}
\usepackage{multirow} 
\usepackage{makecell}
\usepackage[mathscr]{euscript}
\usepackage{charter}
\usepackage{verbatim}
\usepackage{tikz-cd}
\usepackage{tikz}
\usepackage{mathtools}
\usepackage{tabularx}
\usepackage{soul}
\usepackage{subcaption}
\usepackage[arrow,curve,matrix,tips,2cell]{xy}  
  \SelectTips{eu}{10} \UseTips
  \UseAllTwocells
 
\usepackage{bbold}
\usepackage{enumitem}

\usetikzlibrary{shapes}

\newtheorem*{theorem*}{Theorem}
\newtheorem*{claim*}{Claim}
\newtheorem{theorem}{Theorem}[section]
\newtheorem{proposition}[theorem]{Proposition}
\newtheorem{lemma}[theorem]{Lemma}

\newtheorem{corollary}[theorem]{Corollary} 

\theoremstyle{definition}
\newtheorem{definition}[theorem]{Definition}
\theoremstyle{remark}
\newtheorem{example}[theorem]{Example} 
 
\newtheorem{remark}[theorem]{Remark}
\newtheorem{observation}[theorem]{Observation}

\newcommand\R{\mathbb R}

\newcommand\Z{\mathbb Z}

\newcommand\sgn{{\rm sgn\,}}

\newcommand\U{\mathscr U}

\DeclareMathOperator{\vertex}{\mathrm{vert}}
\DeclareMathOperator{\link}{\mathrm{link}}

\DeclareMathOperator{\cone}{\mathrm{cone}}

\DeclareMathOperator{\interior}{\mathrm{int}}
\DeclareMathOperator{\relint}{\mathrm{relint}}
\DeclareMathOperator{\id}{\mathrm{id}}

\DeclareMathOperator{\conv}{\mathrm{conv}}

\begin{document}

\title[Cone and constrained colorful Carath\'eodory Theorems]{Cone and constrained colorful Carath\'eodory Theorems}

\author[Blagojevi\'c]{Pavle V. M. Blagojevi\'{c}}
\thanks{The research by Pavle V. M. Blagojevi\'{c} leading to these results has
received funding from the Serbian Ministry of Science, Technological development and Innovations. It was also supported by G\"unter Ziegler's group at the Freie Universit\"at Berlin ``Arbeitsgruppe Diskrete Geometrie und Topologische Kombinatorik''.}
\address{Inst. Math., FU Berlin, Arnimallee 2, 14195 Berlin, Germany\hfill\break
	\mbox{\hspace{4mm}}Mat. Institut SANU, Knez Mihailova 36, 11001 Beograd, Serbia}
\email{blagojevic@math.fu-berlin.de}

\author[Karasev]{Roman N. Karasev}
\thanks{}
\address{Institute for Information Transmission Problems RAS, Bolshoy Karetny 19, 127994 Moscow, Russia\hfill\break \mbox{\hspace{4mm}} Moscow Institute of Physics and Technology, Institutskiy 9, 141700 Dolgoprudny, Russia}
\email{r n karasev@mail.ru}

\author[Zsigri]{B\'alint Zsigri}
\thanks{The research by B\'alint Zsigri was funded by the Deutsche Forschungsgemeinschaft (DFG, German Research Foundation) under Germany's Excellence Strategy -- The Berlin Mathematics Research Center MATH+ (EXC-2046/1, EXC-2046/2, project ID: 390685689).}
\address{Inst. Math., FU Berlin, Arnimallee 2, 14195 Berlin, Germany}
\email{balint.zsigri@fu-berlin.de}

\begin{abstract}
Holmsen proved in 2016 a generalization of the classical colorful Carathéodory theorem in which a matroid imposes additional constraints on the desired colorful transversal. His approach also works in the more general setting of oriented matroids, rather than relying directly on convex hulls.

In this paper, we extend these ideas in several directions. First, we study which colorful Carath\'eodory-type results remain valid when convex cones replace convex hulls, as well as analogous modifications in the oriented matroid setting. Second, we consider variants in which the additional constraint on the transversal is not encoded by a matroid. This leads to new extensions of the classical Tverberg theorem.

Our approach is topological, following the methods of Holmsen, and Kalai and Meshulam, on which it builds. The key idea is to analyze homology groups of simplicial complexes that encode colorful Carathéodory-type phenomena, such as the support complex of an oriented matroid. 
In particular, one shows that these complexes are (near-)$d$-Leray. We extend this analysis by carrying out more detailed homology computations for these complexes, with the aim of enabling further and more refined applications of the method.
\end{abstract}

\maketitle

\bigskip
\section{Introduction and statement of central results}\label{sec:intro-and-main-results}
\bigskip

Since its introduction, the colorful Carath\'eodory theorem of B\'ar\'any \cite{Barany1982} has become one of the central results in combinatorial convexity. 
Its significance was firmly established through its role in the celebrated proof of the classical Tverberg theorem \cite{Tverberg1966} via Sarkaria’s tensor trick \cite{Sarkaria1992}, as well as through numerous applications, including colorful analogues of Fenchel’s and Kirchberger’s theorems \cite{Barany2021}. 
Efforts to further generalize this result have continued ever since, yielding extensions by Arocha, B\'ar\'any, Bracho, Fabila, and Montejano \cite{ArochaBaranyBrachoFabilaMontejano2009}, Holmsen, Pach, and Tverberg \cite{HolmsenPachTverberg2008}, Kalai and Meshulam \cite{KalaiMeshulam2005}, Holmsen \cite{Holmsen2016}, and, most recently, \cite{Blagojevic2025}. 
We review these developments in Section \ref{sec:intro-and-main-results:existing-generalizations} before presenting our new contributions in Section \ref{sec:intro-and-main-results:new-generalizations}.

\medskip
\subsection{Generalizations of the colorful Carath\'eodory theorem so far}\label{sec:intro-and-main-results:existing-generalizations}
\medskip

In this section, we review all known versions of the colorful Carath\'eodory theorem; a summary of these results is provided in Theorem \ref{thm:caratheodory:all} and Table \ref{tbl:caratheodory-thms}.

\medskip
A collection of vectors in $\R^d$, indexed by the elements of a finite set $V$, can be formally represented as a function $A\colon V\to\R^d$. In this language, the original colorful Carath\'eodory theorem of B\'ar\'any \cite[Thm.\,2.1]{Barany1982} claims the following.

\begin{theorem}\label{thm:caratheodory:original}
Let $V$ be a finite set, $d\ge 0$, and let $A\colon V\to\R^d$ be a function. 
Suppose that $r\ge d+1$ and that $U_1,\dots,U_r\subseteq V$ are pairwise disjoint subsets satisfying
\[0\in\conv(A(U_1))\cap\cdots\cap\conv(A(U_r)).\]
Then there exist elements $u_1\in U_1,\ldots,u_r\in U_r$ such that 
\[0\in\conv(A(\{u_1,\ldots,u_r\})).\]
\end{theorem}

Alongside Theorem \ref{thm:caratheodory:original}, the same paper also introduced two further generalizations of the colorful Carath\'eodory theorem \cite[Thm.\,2.2 and Thm.\,2.3]{Barany1982}. 
For the first of these, recall that the \emph{cone} $\cone(S)$ of a set $S\subseteq\R^d$ is the smallest convex cone containing $S$, where a convex cone is a set containing $0$ and is closed under addition and multiplication by nonnegative scalars.

\begin{theorem}\label{thm:caratheodory:cone}
Let $V$ be a finite set, $v_0\notin V$, $d\geq 0$, and let $A\colon {V\sqcup\{v_0\}}\to\R^d$ be a function. 
Suppose that $r\geq d$ and that $U_1,\ldots,U_r\subseteq V$ are pairwise disjoint subsets of $V$ satisfying 
\[A(v_0)\in\cone(A(U_1))\cap\cdots\cap\cone(A(U_r)).\] 
Then there exist elements $u_1\in U_1,\ldots,u_r\in U_r$ such that 
\[A(v_0)\in\cone(A(\{u_1,\ldots,u_r\})).\]
\end{theorem}

Any function $A_{\operatorname{aff}}\colon V\to\R^{d-1}$ can be transformed into a function $A\colon V\sqcup{v_0}\to\R^d$ by setting $A(v_0):=(0,\ldots,0,1)$ and composing $A_{\operatorname{aff}}$ with the embedding $\R^{d-1}\hookrightarrow\R^d$ given by $x\mapsto(x,1)$. 
Then, for every $U\subseteq V$, one has
\[
A(v_0)\in\cone(A(U)) \ \Longleftrightarrow \ 0\in\conv(A_{\operatorname{aff}}(U)).
\]
Consequently, the theorem above indeed generalizes the original colorful Carath\'eodory theorem (Theorem \ref{thm:caratheodory:original}).

\medskip
The second variation of Theorem \ref{thm:caratheodory:original} presented in B\'ar\'any’s 1982 paper \cite[Thm.\, 2.3]{Barany1982} generalizes the original theorem in a different, but equally natural, direction.

\begin{theorem}\label{thm:caratheodory:fixed}
Let $V$ be a finite set, $d\geq 0$, and let $A\colon V\to\R^d$ be a function.
Suppose that $r\geq d+1$ and that $U_1,\ldots,U_{r-1}\subseteq V$ are  pairwise disjoint subsets of $V$ satisfying 
\[0\in\conv(A(U_1))\cap\cdots\cap\conv(A(U_{r-1})).\] 
Then for any element $u\in V-\bigcup_{i=1}^{r-1}U_i$ there exist elements $u_1\in U_1,\ldots,u_{r-1}\in U_{r-1}$ such that 
\[0\in\conv(A(\{u_1,\ldots,u_{r-1},u\})).\]
\end{theorem}

Since B\'ar\'any’s original results, the colorful Carath\'eodory theorem has inspired a continuing search for new variations.
In particular, Arocha, B\'ar\'any, Bracho, Fabila, and Montejano \cite[Thm.\,1]{ArochaBaranyBrachoFabilaMontejano2009}, as well as Holmsen, Pach, and Tverberg \cite[Thm.\,8]{HolmsenPachTverberg2008}, independently established the first genuinely new colorful Carath\'eodory theorem, stated below.

\begin{theorem}\label{thm:caratheodory:pair}
Let $V$ be a finite set, $d\geq 0$, and let $A\colon V\to\R^d$ be a function. 
Suppose that $r\geq d+1$ and that $U_1,\ldots,U_r\subseteq V$ are pairwise disjoint subsets of $V$  satisfying
\[0\in\bigcap_{1\leq i<j\leq r}\conv(A(U_i\cup U_j)).\] 
Then there exist elements $u_1\in U_1,\ldots,u_r\in U_r$ such that 
\[0\in\conv(A(\{u_1,\ldots,u_r\})).\]
\end{theorem}

Both Theorem \ref{thm:caratheodory:fixed} and Theorem \ref{thm:caratheodory:pair} modify the original Theorem \ref{thm:caratheodory:original} by altering which subsets $U\subseteq V$ are assumed to satisfy $0\in\conv(A(U))$ and what type of subset $U'\subseteq V$ is sought satisfying $0\in\conv(A(U'))$. 
Kalai and Meshulam \cite[Cor.\,1.4]{KalaiMeshulam2005} discovered another colorful Carath\'eodory theorem of this kind: their formulation uses a \emph{matroid} to encode both the subsets $U\subseteq V$ for which $0\in\conv(A(U))$ is assumed and the class of subsets $U'\subseteq V$ among  which we look for one with  $0\in\conv(A(U'))$.

\medskip
Recall that a matroid on a finite set $V$ is a nonvoid abstract simplicial complex $\mathscr{M}$ on the ground set $V$ satisfying the following exchange property: whenever $\sigma,\sigma'\in\mathscr{M}$ with $|\sigma|<|\sigma'|$, there exists an element $v\in\sigma'-\sigma$ such that $\sigma\cup{v}\in\mathscr{M}$.  The faces of $\mathscr{M}$ are called its \emph{independent sets}. The \emph{rank} $\rho_{\mathscr{M}}(U)$ of a subset $U\subseteq V$ is the cardinality of a largest independent set contained in $U$, and the rank of the matroid $\mathscr{M}$ is defined by $\operatorname{rank}(\mathscr{M}):=\rho_\mathscr{M}(V)$.
Terminology and notation concerning simplicial complexes will be reviewed in Section \ref{sec:the-method:sc-tools}. 

\medskip
For example, let $U_1,\ldots,U_r\subseteq V$ be pairwise disjoint subsets, each regarded as $0$-dimensional simplicial complex with vertex set $U_i$. 
Their join $U_1*\cdots*U_r$ is then a matroid on the ground set $V$, namely the \emph{partition matroid}, whose vertex set is $U_1\cup\cdots\cup U_r$. 
The original colorful Carath\'eodory theorem of B\'ar\'any (Theorem \ref{thm:caratheodory:original}) is recovered from the colorful Carath\'eodory theorem of Kalai and Meshulam stated below by taking $\mathscr{M}:=U_1*\cdots *U_r$.

\begin{theorem}\label{thm:caratheodory:matroid}
Let $V$ be a finite set, $d\geq 0$, and let $A\colon V\to\R^d$ be a function. 
Suppose that $\mathscr{M}$ is a matroid on the ground set $V$ satisfying $0\in\conv(A(U))$ for every $U\subseteq V$ with $\rho_\mathscr{M}(U)=\rho_\mathscr{M}(V)$ and $\rho_\mathscr{M}(V-U)\leq d$. 
Then there exists an independent set $U'$ of $\mathscr{M}$ such that $0\in\conv(A(U'))$.
\end{theorem}

Another way to generalize B\'ar\'any’s colorful Carath\'eodory theorem is to replace the function $A\colon V\to\R^d$ with its associated oriented matroid, a combinatorial abstraction of $A$, as was done by Holmsen \cite[Thm.\,1.2]{Holmsen2016}. 
The necessary theory of oriented matroids is reviewed in detail in Section \ref{sec:prelim-oms}, but for the purposes of the result below it suffices to note that, for any subset $U\subseteq V$, one has $0\in\conv(A(U))$ if and only if $U$ contains a positive circuit of the oriented matroid associated with $A$. 
Oriented matroids capture enough of the affine geometry of $A$ for the theorem below to remain valid, while at the same time providing a discrete and more general framework.

\begin{theorem}\label{thm:caratheodory:oriented-matroid}
Let $V$ be a finite set, $d\geq 0$, and let $\mathscr{O}$ be an oriented matroid on set of elements $V$ of rank at most $d$. 
Suppose that $\mathscr{M}$ is a matroid on ground set $V$ of rank at least $d+1$ such that $U$ contains a positive circuit of $\mathscr{O}$ for every subset $U\subseteq V$ with $\rho_\mathscr{M}(V-U)<d$.
Then there exists an independent subset $U'$ of $\mathscr{M}$ containing a positive circuit of $\mathscr{O}$.
\end{theorem}

In the original phrasing of this result \cite[Thm.\,1.2]{Holmsen2016}, it is implicitly requires that $1 \leq \operatorname{rank}(\mathscr{O})$.
Indeed, all oriented matroids in that paper are assumed to be loop-free, and $\operatorname{rank}(\mathscr{M}) \geq d+1 \geq 1$ implies that $V$ is nonempty. 
Nevertheless, it is straightforward to verify --- as we do in Section \ref{sec:cc-proofs} --- that the theorem remains valid when $\operatorname{rank}(\mathscr{O}) = 0$.

\medskip
The final variant of the colorful Carath\'eodory theorem we wish to mention is the so-called constrained colorful Carath\'odory theorem \cite[Thm.\,1.6]{Blagojevic2025}. However, it is instructive to first provide some context, so we postpone its statement to Section \ref{sec:intro-and-main-results:new-generalizations}.

\medskip
\subsection{New colorful Carath\'eodory theorems}\label{sec:intro-and-main-results:new-generalizations}
\medskip

In all colorful Carath\'eodory theorems considered so far, we fixed a finite set $V$ and a logical formula $\mathcal{F}(-)$ with a single free set-valued variable. 
Assuming that $\mathcal{F}(U_i)=\mathrm{true}$ for a given collection $\mathfrak{U}:=(U_i)_{i\in I}$ of subsets of $V$, we concluded that there exists some $j\in J$ such that $\mathcal{F}(U'_j)=\mathrm{true}$, where $\mathfrak{U}':=(U'_j)_{j\in J}$ is another prescribed collection of subsets of $V$. 
In Theorems \ref{thm:caratheodory:original}, \ref{thm:caratheodory:fixed}, \ref{thm:caratheodory:pair}, and \ref{thm:caratheodory:matroid} the formula $\mathcal{F}(U)$ denotes the statement ``$0\in\conv(A(U))$'', in Theorem \ref{thm:caratheodory:cone} stands for ``$A(v_0)\in\cone(A(U))$'', while in Theorem \ref{thm:caratheodory:oriented-matroid}, $\mathcal{F}(U)$ expands to ``$U$ contains a positive circuit of $\mathscr{O}$''.
The corresponding choices of $\mathfrak{U}=(U_i)_{i\in I}$ and $\mathfrak{U}'=(U'_j)_{j\in J}$ can be read off directly from the respective statements; see Theorem \ref{thm:caratheodory:all}.
Now, observe that in each case, whenever $W\subseteq U$ and $\mathcal{F}(U)=\mathrm{false}$, one also has $\mathcal{F}(W)=\mathrm{false}$.
Consequently,
\[\mathscr{C}:=\{U\subseteq V:\mathcal{F}(U)=\mathrm{false}\}\]
forms an abstract simplicial complex, and the conclusion of each colorful Carath\'eodory theorem can be expressed simply as $\mathfrak{U}'\nsubseteq\mathscr{C}$.
If $\mathscr{K}$ denotes the smallest simplicial complex containing $\mathfrak{U}'$ (with respect to inclusion), then 
\[\mathfrak{U}'\nsubseteq\mathscr{C}\iff\mathscr{K}\nsubseteq\mathscr{C}.\]

\medskip
In all theorems presented so far, the complex $\mathscr{K}$ turns out to be a matroid (see the discussion preceding Theorem \ref{thm:caratheodory:matroid}). 
In \cite[Theorem 1.6]{KalaiMeshulam2005}, Kalai and Meshulam established a homological criterion guaranteeing that a matroid cannot be a subcomplex of a so-called $(d-1)$-Leray complex. 
Although, in most cases, the complex $\mathscr{C}$ is \emph{not} $(d-1)$-Leray, and hence their theorem does not apply directly, Holmsen proved in \cite[Prop.\,1.3]{Holmsen2016} that $\mathscr{C}$ is what he calls near-$(d-1)$-Leray whenever $\mathcal{F}(U)$ is given by $0\in\conv(A(U))$'' or $U$ contains a positive circuit of $\mathscr{O}$''. 
In the same paper, the assumptions of \cite[Thm.\,1.6]{KalaiMeshulam2005} were adapted to accommodate near-$(d-1)$-Leray complexes (see \cite[Thm.\,1.4]{Holmsen2016}), thereby yielding colorful Carath\'eodory Theorems \ref{thm:caratheodory:original} and \ref{thm:caratheodory:oriented-matroid}.

\medskip
The combined work of Kalai, Meshulam, and Holmsen was further developed in \cite{Blagojevic2025} in two directions, both of which we pursue here. 
First, observe that the previous approach relied on only a single essential property of matroids. 
This permits a generalization of the method to more general simplicial complexes, leading to colorful Carath\'eodory theorems in which $\mathscr{K}$ is no longer required to be a matroid. 
Second, \cite{Blagojevic2025} carried out extensive homological computations for the complex $\mathscr{C}$ arising from $\mathcal{F}(U)=$``$0\in\conv(A(U))$'', going substantially beyond the mere verification that $\mathscr{C}$ is near-$(d-1)$-Leray. 
We restructure and extend these computations, and perform a comparable analysis for several additional variants of $\mathscr{C}$ beyond those introduced at the beginning of this section. 
Besides enabling --- and indeed yielding --- new colorful Carath\'eodory theorems, these calculations are intended to serve as groundwork for future developments. The central tool in the method of Kalai and Meshulam is the Homological Nerve Theorem, which may, for instance, be derived from the more powerful Leray spectral sequence. 
Combined with the computations presented here, this perspective may permit extensions of colorful Carath\'eodory theorems even in settings where the original formulation of the method, as developed by the previous authors, no longer applies.

\medskip
The work of Holmsen, as well as that of Kalai and Meshulam, will be discussed in detail in Section \ref{sec:the-method}. 
For the remainder of this section, however, we shift our focus to the new colorful Carath\'eodory theorems arising from our results. 
Like the colorful Carath'eodory theorems discussed above, these generalize the original Theorem \ref{thm:caratheodory:original} in several distinct directions:
\begin{enumerate}[label=\normalfont{(\Alph*)}]
    \item\label{en:generalization:complex} New collections of subsets $\mathfrak{U}=(U_i)_{i\in I}$ and $\mathfrak{U}'=(U'_j)_{j\in J}$ of $V$ are identified such that, whenever $0\in\conv(A(U_i))$ for all $i\in I$, one has $0\in\conv(A(U'_j))$ for some $j\in J$.
    
    \item\label{en:generalization:oms} The function $A\colon V\to\R^d$ is replaced by its associated oriented matroid, and both the assumptions and the conclusions of the colorful Carath\'eodory theorem are formulated in the language of oriented matroids.
    
    \item\label{en:generalization:hulls} Convex hulls are replaced by other operators; for example, instead of considering ``$0\in\conv(A(U))$'', we study conditions such as ``$A(v_0)\in\cone(A(U))$'' or ``$A(v_0)\in\operatorname{aff}(A(U))$''.
\end{enumerate}
Recall that Theorem \ref{thm:caratheodory:cone} follows the description in \ref{en:generalization:hulls}, Theorems \ref{thm:caratheodory:fixed}, \ref{thm:caratheodory:pair}, and \ref{thm:caratheodory:matroid} are generalizations falling under scheme \ref{en:generalization:complex}, while Theorem \ref{thm:caratheodory:oriented-matroid} constitutes a generalization in the sense of both \ref{en:generalization:complex} and \ref{en:generalization:oms}. 

\medskip
The first theorem we present, the \emph{constrained colorful Carath\'eodory theorem}, is a generalization of type \ref{en:generalization:complex}, originally introduced in \cite[Thm.\,1.6]{Blagojevic2025}. 
It demonstrates that meaningful versions of the colorful Carath'eodory theorem exist in which the smallest simplicial complex $\mathscr{K}$ containing the collection $\mathfrak{U}'=(U'_j)_{j\in J}$ is not necessarily the independence complex of a matroid $\mathscr{M}$. 
Observation \ref{obs:constrained-caratheodory-new} will later show that the constrained colorful Carath\'eodory theorem presented below is genuinely not a special case of either Theorem \ref{thm:caratheodory:matroid} or Theorem \ref{thm:caratheodory:oriented-matroid}. 
Moreover, it is, of course, an extension of the original colorful Carath\'eodory theorem, Theorem \ref{thm:caratheodory:original}, as will be explained in Observation \ref{obs:implications-between-caratheodorys}.

\begin{theorem}\label{thm:caratheodory:constrained}
    Let $V$ be a finite set, $d\geq 0$, and let $A\colon V\to\R^d$ be a function. 
    Suppose that $r\geq\max\{d+1,2\}$, $U_1,\ldots,U_r\subseteq V$ are pairwise disjoint subsets of $V$ with the property that $0\in \conv(A(U_1))\cap\cdots\cap\conv(A(U_r))$, and  that $\mathscr{K}$ is a simplicial complex on the ground set $V$ of the form $\mathscr{K}_{12}*U_3*\cdots*U_r$, where $\mathscr{K}_{12}$ is a connected simplicial subcomplex of $U_1*U_2$ with vertex set $U_1\cup U_2$ where the sets $U_i$ are viewed as $0$-dimensional simplicial complexes with vertex set $U_i$.
    Then there exists a face $\sigma\in\mathscr{K}$ such that $0\in\conv(A(\sigma))$.
\end{theorem}

This result admits further generalizations, albeit in a not entirely obvious way; see Theorem \ref{thm:caratheodory:3-joined} and Observation \ref{obs:no-further-grouping-in-constrained}. 
Moreover, as observed in \cite[Cor.\,1.7]{Blagojevic2025}, the theorem above yields a new strengthening of Tverberg's theorem, especially in the context of \emph{ordered} Tverberg partitions. 
For further details, see Section \ref{sec:applications-and-limits}, in particular Corollary \ref{crl:new-tverberg}.
   
\medskip
While the constrained colorful Carath\'eodory theorem stated above is a generalization in the sense of \ref{en:generalization:complex}, we will in fact prove a slightly stronger version, incorporating the additional generalization described in \ref{en:generalization:oms}. 
By contrast, the second colorful Carath\'eodory theorem presented below provides a generalization in all three senses, combining the ideas underlying Theorems \ref{thm:caratheodory:cone} and \ref{thm:caratheodory:oriented-matroid}. 
To formulate it, we first recall some notation.

\medbreak
In the setting of the oriented matroid $\mathscr{O}$ associated to the function $A\colon V\sqcup{v_0}\to\R^d$, there is an operator $\conv\colon 2^{V\sqcup{v_0}}\to 2^{V\sqcup{v_0}}$ with the property that
 \[v_0\in\conv(U)\iff A(v_0)\in\cone(A(U))\] 
 for every $U\subseteq V$.
This operator $\conv$ admits a definition for arbitrary oriented matroids $\mathscr{O}$, and it gives rise to the following generalization of Theorem \ref{thm:caratheodory:cone}, analogous to the way in which Theorem \ref{thm:caratheodory:oriented-matroid} generalizes Theorem \ref{thm:caratheodory:original}:

\begin{theorem}\label{thm:caratheodory:om-cone-matroid}
    Let $V$ be a finite set, $v_0\notin V$, $d\geq 0$, and let $\mathscr{O}$ be an oriented matroid on set of elements $V\sqcup\{v_0\}$ of rank at most $d$. 
    Suppose that $\mathscr{M}$ is a matroid on $V$ such that $v_0\in\conv(U)$ for every $U\subseteq V$ with $\rho_{\mathscr{M}}(V-U)<d$.
    Then there exists an independent set $U'$ of $\mathscr{M}$ such that  $v_0\in\conv(U')$.
\end{theorem}

Since the assumption is that $v_0\in\conv(U)$ for all $U\subseteq V$ satisfying $\rho_{\mathscr{M}}(V-U)<d$, rather than for all $U\subseteq V$ with $\rho_{\mathscr{M}}(V-U)<d-1$, this is not a direct generalization of Theorem \ref{thm:caratheodory:oriented-matroid}, even in the affine setting. In fact, Theorem \ref{thm:caratheodory:om-cone-matroid} fails if the condition $\rho_{\mathscr{M}}(V-U)<d$ is replaced by $\rho_{\mathscr{M}}(V-U)<d-1$. Thus, it constitutes a genuinely new colorful Carath'eodory theorem, independent of Theorem \ref{thm:caratheodory:oriented-matroid}. See Observations \ref{obs:conv-vs-cone-in-oms} and \ref{obs:no-cone-caratheodory} for further discussion.
Nevertheless, Theorem \ref{thm:caratheodory:om-cone-matroid} still generalizes Theorem \ref{thm:caratheodory:cone}, as applying it to the partition matroid $U_1*\cdots*U_r$ yields the following result.

\begin{theorem}\label{thm:caratheodory:om-cone}
    Let $V$ be a finite set, $v_0\notin V$, $d\geq 0$, and let $\mathscr{O}$ be an oriented matroid on set of elements $V\sqcup\{v_0\}$ of rank at most $d$. 
    Suppose that $r\geq d$ and that $U_1,\ldots,U_r\subseteq V$ are pairwise disjoint subsets of $V$  with the property that $v_0\in\conv(U_1)\cap\cdots\cap\conv(U_r)$.
    Then there exist $u_1\in U_1,\ldots,u_r\in U_r$ such that $v_0\in\conv(\{u_1,\ldots,u_r\})$.
\end{theorem}

However, it turns out that Theorems \ref{thm:caratheodory:fixed}, \ref{thm:caratheodory:pair}, and \ref{thm:caratheodory:constrained} do not extend well to the setting in which $0\in\conv(A(-))$'' is replaced by $A(v_0)\in\cone(A(-))$'' or $v_0\in\conv(-)$''. In fact, only Theorem \ref{thm:caratheodory:pair} admits a meaningful extension, as discussed in Observation \ref{obs:no-cone-caratheodory}; however, even this extension does not imply the original version either (see Observation \ref{obs:conv-vs-cone-in-oms} for details). In this sense, this extension is also a new and independent result. For brevity, we state only the variant of Theorem \ref{thm:caratheodory:pair} in which $0\in\conv(A(U))$'' is replaced by $v_0\in\conv(U)$'', since this formulation implies the version in which we use $A(v_0)\in\cone(A(U))$''.

\begin{theorem}\label{thm:caratheodory:om-cone-pair}
    Let $V$ be a finite set, $d\geq 0$, and let $\mathscr{O}$ be an oriented matroid on set of elements $V\sqcup\{v_0\}$ of rank at most $d$. 
    Suppose that $r\geq d+1$ and that $U_1,\ldots,U_r\subseteq V$ are pairwise disjoint subsets of $V$ satisfying that $v_0\in\bigcap_{1\leq i<j\leq r}\conv(U_i\cup U_j)$.
    Then there exist $u_1\in U_1,\ldots,u_r\in U_r$ such that $v_0\in\conv(\{u_1,\ldots,u_r\})$.
\end{theorem}

Finally, we note that Theorems \ref{thm:caratheodory:fixed}, \ref{thm:caratheodory:pair}, and \ref{thm:caratheodory:constrained} all admit generalizations of type \ref{en:generalization:oms}, that is, if in these statements the map $A$ is replaced by an oriented matroid $\mathscr{O}$ of rank at most $d$, and ``$0\in\conv(A(U))$'' is replaced by ``$U$ contains a positive circuit of $\mathscr{O}$'', then the resulting statements remain valid.

\medskip
The proofs of the colorful Carath\'eodory theorems will be presented in Section \ref{sec:cc-proofs} and will follow the method described in Section \ref{sec:the-method}. In between, we review the relevant aspects of the theory of oriented matroids in Section \ref{sec:prelim-oms}, and study the different versions of $\mathscr{C}$ introduced above in Section \ref{sec:avoiding-complexes}, establishing the key properties on which the various colorful Carath\'eodory theorems rely. Before that, however, in Section \ref{sec:applications-and-limits} we explore counterexamples and applications related to our colorful Carath\'eodory theorems.

Since we will frequently refer to and compare Theorems \ref{thm:caratheodory:original} through \ref{thm:caratheodory:om-cone-pair} in what follows, we conclude this section by presenting a unifying theorem (Theorem \ref{thm:caratheodory:all}) together with a summary table (Table \ref{tbl:caratheodory-thms}). 
To keep these concise, we slightly reformulate the various colorful Carath\'eodory theorems above, without changing their essential content.

\begin{theorem}\label{thm:caratheodory:all}
Given a finite set $V$ with $v_0\notin V$, integers $d\geq 0$ and $r\geq d+1$, and either a function $A\colon V\sqcup{v_0}\to\R^d$ or an oriented matroid $\mathscr{O}$ of rank at most $d$ on the ground set $V\sqcup{v_0}$, we have that the implication 
    \begin{quotation}\centering
        ``\,\emph{If $\mathcal{F}(U)$ for all $U\in\mathfrak{U}$ then $\mathcal{F}(U')$ for some $\U'\in\mathfrak{U}'$}\,''
    \end{quotation}
    holds if any of the points below are true (labelled by the numbers of Theorems \ref{thm:caratheodory:original} through \ref{thm:caratheodory:om-cone-pair}):
    \begin{itemize}
        \setlength{\itemsep}{5pt}
        \item[\ref{thm:caratheodory:original}.] 
        $\mathcal{F}(W)=$ ``$0\in\conv(A(W))$'', $|\mathfrak{U}|=r$, $\bigsqcup\mathfrak{U}\subseteq V$, \\ 
        $\mathfrak{U}'=\{U'\subseteq V: |U'\cap U|=1\text{~for all~}U\in\mathfrak{U}\}$.
        
        \item[\ref{thm:caratheodory:cone}.] 
        $\mathcal{F}(W)=$ ``$A(v_0)\in\cone(A(W))$'', $|\mathfrak{U}|=r-1$, $\bigsqcup\mathfrak{U}\subseteq V$, $\emptyset\notin\mathfrak{U}$, \\
        $\mathfrak{U}'=\{U'\subseteq V: |U'\cap U|=1\text{~for all~}U\in\mathfrak{U}\}$. 
        
        \item[\ref{thm:caratheodory:fixed}.]
        $\mathcal{F}(W)=$ ``$0\in\conv(A(W))$'', $|\mathfrak{U}|=r-1$, $\bigsqcup\mathfrak{U}\subseteq V$, \\
        $\mathfrak{U}'=\{U'\cup\{v_0\}:U'\subseteq V,|U'\cap U|=1\text{~for all~}U\in\mathfrak{U}\}$.
        
        \item[\ref{thm:caratheodory:pair}.]
        $\mathcal{F}(W)=$ ``$0\in\conv(A(W))$'', $\mathfrak{U}=\{U_1\cup U_2:U_1\neq U_2\in\mathfrak{U}_0\}$, $|\mathfrak{U}_0|=r$, $\bigsqcup\mathfrak{U}_0\subseteq V$, $\emptyset\notin\mathfrak{U}_0$, \\
        $\mathfrak{U}'=\{U'\subseteq V, |U'\cap U|=1\text{~for all~}U\in\mathfrak{U}_0\}$.
        
        \item[\ref{thm:caratheodory:matroid}.] 
        $\mathcal{F}(W)=$ ``$0\in\conv(A(W))$'', $\mathfrak{U}=\{U\subseteq V:\rho_{\mathfrak{U}'}(U)=\rho_{\mathfrak{U}'}(V),\rho_{\mathfrak{U}'}(V-U)\leq d\}$, \\
        $\mathfrak{U}'$ is a matroid on $V$.
        
        \item[\ref{thm:caratheodory:oriented-matroid}.]
        $\mathcal{F}(W)=$ ``$U$ contains a positive circuit of $\mathscr{O}$'', $\mathfrak{U}=\{U\subseteq V:\rho_{\mathfrak{U}'}(V-U)<d\}$, \\
        $\mathfrak{U}'$ is a matroid on $V$ of rank $r$.

        \item[\ref{thm:caratheodory:constrained}.]
        $\mathcal{F}(W)=$``$0\in\conv(A(W))$'', $|\mathfrak{U}|=r\geq 2$, $\bigsqcup\mathfrak{U}\subseteq V$, \\
        $\mathfrak{U}'=\mathscr{K}_{12}*U_3*\cdots*U_r$, $\mathfrak{U}=\{U_1,\ldots,U_r\}$, $\mathscr{K}_{12}\subseteq U_1*U_2$ connected on vertex set $U_1\cup U_2$.

        \item[\ref{thm:caratheodory:om-cone-matroid}.]
        $\mathcal{F}(W)=$``$v_0\in\conv(U)$'', $\mathfrak{U}=\{U\subseteq V:\rho_{\mathfrak{U}'}(V-U)< d\}$, \\
        $\mathfrak{U}'$ is a matroid on $V$.
        
        \item[\ref{thm:caratheodory:om-cone}.]
        $\mathcal{F}(W)=$``$v_0\in\conv(U)$'', $|\mathfrak{U}|=r-1$, $\bigsqcup\mathfrak{U}\subseteq V$,$\emptyset\notin\mathfrak{U}$, \\
        $\mathfrak{U}'=\{U'\subseteq V:|U'\cap U|=1\text{~for all~}U\in\mathfrak{U}\}$.

        \item[\ref{thm:caratheodory:om-cone-pair}.]
        $\mathcal{F}(W)=$``$v_0\in\conv(U)$'', $\mathfrak{U}=\{U_1\cup U_2:U_1\neq U_2\in\mathfrak{U}_0\}$, $|\mathfrak{U}_0|=r$, $\bigsqcup\mathfrak{U}_0\subseteq V$, $\emptyset\notin\mathfrak{U}_0$, \\
        $\mathfrak{U}'=\{U'\subseteq V:|U'\cap U|=1\text{~for all~}U\in\mathfrak{U}\}$.
    \end{itemize}
\end{theorem}

\begin{table}[ht]
    \centering
    \renewcommand{\arraystretch}{1.5}
    \begin{tabular}{c|r|ll|l|l|}
        \cline{2-6}
        $\mathcal{F}$ &
        Thm & 
        \multicolumn{2}{l|}{If $\mathfrak{U}:=\{U_1,\ldots\}\subseteq 2^V-\emptyset$} & 
        $\mathcal{F}(U)$ for $U$ & 
        Then $\exists U'\in\mathscr{K}$: $\mathcal{F}(U')$ for \\\cline{2-6}\cline{2-6}
        $0$ &
        \ref{thm:caratheodory:original}. & 
        $\bigsqcup\mathfrak{U}\subseteq V$, & 
        $|\mathfrak{U}|=r$, & 
        in $\mathfrak{U}$. & 
        $\mathscr{K}=U_1*\cdots*U_r$. \\\cline{2-6}
        $A(v_0)$ &
        \ref{thm:caratheodory:cone}. &
        $\bigsqcup\mathfrak{U}\subseteq V$, &
        $|\mathfrak{U}|=r-1$, &
        in $\mathfrak{U}$. &
        $\mathscr{K}=U_1*\cdots*U_{r-1}$.\\\cline{2-6}
        $0$ &
        \ref{thm:caratheodory:fixed}. & 
        $\bigsqcup\mathfrak{U}\subseteq V$, & 
        $|\mathfrak{U}|=r-1$, & 
        in $\mathfrak{U}$. & 
        $\mathscr{K}=U_1*\cdots*U_{r-1}*\{v_0\}$. \\\cline{2-6}
        $0$ &
        \ref{thm:caratheodory:pair}. & 
        $\bigsqcup\mathfrak{U}\subseteq V$, &
        $|\mathfrak{U}|=r$, &
        \makecell[l]{$=U_1\cup U_2$,\\$U_1\neq U_2\in\mathfrak{U}$.} &
        $\mathscr{K}=U_1*\cdots*U_r$. \\\cline{2-6}
        $0$ &
        \ref{thm:caratheodory:matroid}. &
        is a matroid, &
        &
        \makecell[l]{s.t. $\rho_\mathfrak{U}(U)=\rho_\mathfrak{U}(V)$,\\ $\rho_\mathfrak{U}(V-U)\leq d$.} &
        $\mathscr{K}=\mathfrak{U}$. \\\cline{2-6}
        $\mathscr{C}^+$ &
        \ref{thm:caratheodory:oriented-matroid}. &
        is a matroid &
        of rank $r$, & 
        s.t. $\rho_\mathfrak{U}(V-U)<d$. &
        $\mathscr{K}=\mathfrak{U}$. \\\cline{2-6}
        $0$ &
        \ref{thm:caratheodory:constrained}. & 
        $\bigsqcup\mathfrak{U}\subseteq V$, &
        $|\mathfrak{U}|=r$, & 
        in $\mathfrak{U}$. &
        \makecell[l]{$\mathscr{K}=\mathscr{K}_{12}*U_3*\cdots*U_r$,\\ $\mathscr{K}_{12}\subseteq U_1*U_2$ connected,\\$\vertex(\mathscr{K}_{12})=U_1\cup U_2$.} \\\cline{2-6}
        $v_0$ &
        \ref{thm:caratheodory:om-cone-matroid}. &
        is a matroid, &
        &
        s.t. $\rho_\mathfrak{U}(V-U)<d$. &
        $\mathscr{K}=\mathfrak{U}$.\\\cline{2-6}
        $v_0$ &
        \ref{thm:caratheodory:om-cone}. &
        $\bigsqcup\mathfrak{U}\subseteq V$, &
        $|\mathfrak{U}|=r-1$, &
        in $\mathfrak{U}$. &
        $\mathscr{K}=U_1*\cdots*U_{r-1}$. \\\cline{2-6}
        $v_0$ &
        \ref{thm:caratheodory:om-cone-pair} &
        $\bigsqcup\mathfrak{U}\subseteq V$, &
        $|\mathfrak{U}|=r$, &
        \makecell[l]{$=U_1\cup U_2$,\\$U_1\neq U_2\in\mathfrak{U}$.} &
        $\mathscr{K}=U_1*\cdots*U_r$. \\\cline{2-6}        
    \end{tabular}
    \caption{For each $n$, reading only rows $1$ and $n$ (up to down, left to right) gives a colorful Carath\'eodory theorem. It is assumed that $r>d\geq1$. In rows marked with ``$0$'', $\mathscr{F}(U)$ stands for ``$0\in\conv(A(U))$'', while in rows marked with ``$A(v_0)$'' it expands to ``$A(v_0)\in\cone(A(U))$'', where $A\colon {V\sqcup\{v_0\}}\to\R^d$ is a function. In rows marked with ``$\mathscr{C}^+$'' $\mathscr{F}(U)$ means ``$U$ contains a positive circuit of $\mathscr{O}$'', while in those marked with ``$v_0$'' it is shorthand for ``$v_0\in\conv(U)$'', where $\mathscr{O}$ is an oriented matroid on $V\sqcup\{v_0\}$ of rank at most $d$.}
    \label{tbl:caratheodory-thms}
\end{table}

\bigskip
\section{(Non-)implications between, counterexamples for, and applications of colorful Carath\'eodory theorems}\label{sec:applications-and-limits}
\bigskip

Here we review in a more detailed way why some colorful Carath\'eodory theorems imply or do not imply each other, why generalizing them in certain directions is not feasible, and how the new variation of Tverberg's theorem (Corollary \ref{crl:new-tverberg}) follows from the constrained colorful Carath\'eodory theorem.

First, we wish to highlight how many of the colorful Carath\'eodory theorems follow from other ones using simple substitutions.
\begin{observation}\label{obs:implications-between-caratheodorys}
    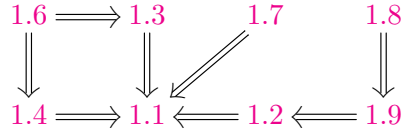
\begin{figure}
        \centering
        \[\xymatrix{
                \ref{thm:caratheodory:oriented-matroid}
                    \ar@{=>}[r]
                    \ar@{=>}[d]
            &   \ref{thm:caratheodory:fixed}
                    \ar@{=>}[d]
            &   \ref{thm:caratheodory:constrained}
                    \ar@{=>}[dl]
            &   \ref{thm:caratheodory:om-cone-matroid}
                    \ar@{=>}[d]
            \\  \ref{thm:caratheodory:pair}
                    \ar@{=>}[r]
            &   \ref{thm:caratheodory:original}
            &   \ref{thm:caratheodory:cone}
                    \ar@{=>}[l]
            &   \ref{thm:caratheodory:om-cone}
                    \ar@{=>}[l]
        }\]
        \caption{Implications between different colorful Carath\'eodory theorems.}
    \end{figure}
    We already noted below the statement of Theorem \ref{thm:caratheodory:cone} that it generalizes Theorem \ref{thm:caratheodory:original}, because to the function $A_{\operatorname{aff}}\colon V\to\R^{d}$ of the latter theorem another function $A\colon {V\sqcup\{v_0\}}\to\R^{d+1}$ can be associated with $0\in\conv(A_{\operatorname{aff}}(U))\iff A(v_0)\in\cone(A(U))$. Before Theorem \ref{thm:caratheodory:om-cone-matroid} we also recalled that $v_0\in\conv(U)$ in the oriented matroid $\mathscr{O}$ associated to $A$ if and only if $A(v_0)\in\cone(A(U))$, and the rank of $\mathscr{O}$ is bounded by the dimension of the codomain of $A$, so Theorem \ref{thm:caratheodory:om-cone} is a generalization of Theorem \ref{thm:caratheodory:cone}. 
    
    At the beginning of Section \ref{sec:intro-and-main-results:new-generalizations} it was highlighted that it suffices to look for a face $\sigma$ of the smallest simplicial complex $\mathscr{K}$ containing $\mathfrak{U}'$ with $\mathcal{F}(\sigma)$ to prove that there is an $U'\in\mathfrak{U}'$ such that $\mathcal{F}(U')$, but note that for a collection of pairwise disjoint $U_1,U_2,\ldots$ the smallest simplicial complex $\mathscr{K}$ containing $\mathfrak{U}':=\{\{u_1,u_2,\ldots\}:u_1\in U_1,u_2\in U_2,\ldots\}$ is the partition matroid $U_1*U_2*\cdots$. In particular, applying Theorem \ref{thm:caratheodory:om-cone-matroid} to the partition matroid $\mathscr{K}:=U_1*\cdots*U_r$ we get that to prove Theorem \ref{thm:caratheodory:om-cone} it suffices to check that $v_0\in\conv(U)$ for all $U\subseteq V$ with $\rho_\mathscr{K}(V-U)<d$, but such $U$ must contain a $U_i$ for some $i$ -- otherwise $\exists u_i\in(V-U)\cap U_i$ for all $i$ and then $\rho_\mathscr{K}(V-U)\geq|\{u_1,\ldots,u_r\}|=r\geq d$ as $\{u_1,\ldots,u_r\}\in U_1*\cdots*U_r$. In particular, $v_0\in\conv(U_i)\subseteq\conv(U)$, showing that Theorem \ref{thm:caratheodory:om-cone} is indeed a consequence of Theorem \ref{thm:caratheodory:om-cone-matroid}.

    Theorems \ref{thm:caratheodory:fixed} and \ref{thm:caratheodory:pair} follow entirely similarly from Theorem \ref{thm:caratheodory:oriented-matroid}. We recalled above Theorem \ref{thm:caratheodory:oriented-matroid} that $U$ contains a positive circuit of the oriented matroid $\mathscr{O}$ (of rank at most $d$) associated to the function $A\colon V\to\R^d$ if and only if $0\in\conv(A(U))$, so to derive Theorems \ref{thm:caratheodory:fixed} and \ref{thm:caratheodory:pair} from Theorem \ref{thm:caratheodory:oriented-matroid} with $\mathscr{M}:=U_1*\cdots*U_r$ it suffices to check that $0\in\conv(A(U))$ for all $U\subseteq V$ with $\rho_\mathscr{M}(V-U)<d$. However, such $U$ must contain the set $U_i$ for at least two distinct $i$, otherwise $\rho_\mathscr{M}(V-U)=|\{1\leq i\leq r:U_i\cap (V-U)\neq\emptyset\}|\geq r-1\geq d$. In particular, in Theorem \ref{thm:caratheodory:fixed} we have $0\in\conv(A(U_i))\subseteq\conv(A(U))$ for an $1\leq i\leq r-1$, and in Theorem \ref{thm:caratheodory:pair} similarly $0\in\conv(A(U_i\cup U_j))\subseteq 0\in\conv(A(U))$, proving that Theorems \ref{thm:caratheodory:fixed} and \ref{thm:caratheodory:pair} indeed follow from Theorem \ref{thm:caratheodory:oriented-matroid}.

    Finally, the original colorful Carath\'eodory theorem (Theorem \ref{thm:caratheodory:original}) of course follows from any of Theorems \ref{thm:caratheodory:fixed}, \ref{thm:caratheodory:pair}, and \ref{thm:caratheodory:constrained}. The set $\{u_1,\ldots,u_{r-1},u\}$ with $0\in\conv(A(\{u_1,\ldots,u_{r-1},u\}))$ guaranteed by Theorem \ref{thm:caratheodory:fixed} for some $u\in U_r$ also satisfies the conclusion of the original theorem, while the assumption of Theorem \ref{thm:caratheodory:pair} clearly follows from that of Theorem \ref{thm:caratheodory:original}, as $0\in\conv(A(U_i))\subseteq\conv(A(U_i\cup U_j))$. Finally, $\mathscr{K}\subseteq U_1*\cdots*U_r$ in Theorem \ref{thm:caratheodory:constrained}, so the conclusion of that theorem also guarantees a selection $u_1\in U_1,\ldots,u_r\in U_r$ with $0\in\conv(A(\{u_1,\ldots,u_r\}))$.
\end{observation}
Just as noted in \cite{Blagojevic2025}, the constrained colorful Carath\'eodory theorem is not a consequence of Theorem \ref{thm:caratheodory:matroid}, or of Theorem \ref{thm:caratheodory:oriented-matroid} for that matter:
\begin{observation}\label{obs:constrained-caratheodory-new}
    Theorem \ref{thm:caratheodory:constrained} is not a corollary of Theorems \ref{thm:caratheodory:matroid} or \ref{thm:caratheodory:oriented-matroid}, particularly if we set $\mathscr{K}_{12}:=\mathscr{K}\subseteq U_1*U_2$, where $U_1$ and $U_2$ are discrete simplicial complexes on vertex sets $\{u_1^1,u_1^2\}$ and $\{u_2^1,u_2^2\}$ respectively, and $\mathscr{K}_{12}$ is illustrated in Figure \ref{fig:cone-counterexamples:K}. Then Theorem \ref{thm:caratheodory:constrained} claims that if $d\geq 0$, $r\geq\max\{d+1,2\}$, and $A\colon {U_1\cup U_2}\to\R^d$ is a function such that $0\in\conv(A(U_1))\cap\conv(A(U_2))$ then there is a $\sigma\in\mathscr{K}_{12}$ with $0\in\conv(A(\sigma))$. $\mathscr{K}_{12}$ is not a matroid (the defining axiom fails for e.g. $\sigma:=\{u_1^1\}$, $\sigma':=\{u_1^2,u_2^1\}$), so the only way that Theorems \ref{thm:caratheodory:matroid} or \ref{thm:caratheodory:oriented-matroid} could guarantee the existence of a $\sigma\in\mathscr{K}_{12}$ with $0\in\conv(A(\sigma))$ is if they are applied to some matroid $\mathscr{M}$ which is a (proper) subcomplex of $\mathscr{K}_{12}$. It is however easily verified that the assumptions of Theorems \ref{thm:caratheodory:matroid} or \ref{thm:caratheodory:oriented-matroid} are not implied for any such choice of $\mathscr{M}$ by the fact that $0\in\conv(A(U_1))\cap\conv(A(U_2))$. Alternatively, one can show examples of $A$ for each such $\mathscr{M}$ where $0\notin\conv(A(\sigma))$ for all $\sigma\in\mathscr{M}$: either the $A$ defined by $A(\{u_1^1,u_2^1\})=1\in\R^1$ and $A(\{u_1^2,u_2^2\})=-1\in\R^1$, or the $A$ defined by $A(\{u_1^1,u_2^2\})=1\in\R^1$ and $A(\{u_1^2,u_2^1\})=-1\in\R^1$ will be good.
\end{observation}

Furthermore, the cone versions of the colorful Carath\'eodory theorems often do not imply the original version, as discussed in the next observation.
\begin{observation}\label{obs:conv-vs-cone-in-oms}
    As noted below the statement of Theorem \ref{thm:caratheodory:cone}, to any function $A_{\operatorname{aff}}\colon V\to\R^{d}$ there is an associated function $A\colon {V\sqcup\{v_0\}}\to\R^{d+1}$ such that $0\in\conv(A_{\operatorname{aff}}(U))\iff A(v_0)\in\cone(A(U))$. 
 This construction means that by replacing in a colorful Carath\'eodory theorem all references to ``$0\in\conv(A(U))$'' with ``$A(v_0)\in\cone(A(U))$'', and allowing $A$ to have codomain $\R^{d+1}$, we obtain a stronger statement than the original.
    
 Performing this transformation on the affine version of Theorem \ref{thm:caratheodory:oriented-matroid}, by which we mean first replacing ``$U$ has a positive circuit of the fixed oriented matroid $\mathscr{O}$ of rank at most $d$'' by ``$0\in\conv(A_{\operatorname{aff}}(U))$ for the fixed function $A_{\operatorname{aff}}\colon V\to\R^d$'', and than this by ``$A(v_0)\in\cone(A(U))$ for the fixed function $A\colon {V\sqcup\{v_0\}}\to\R^{d+1}$'', the third statement is more general than the second. However, this third statement, namely that for a function $A\colon {V\sqcup\{v_0\}}\to\R^{d+1}$ and a matroid $\mathscr{M}$ on $V$ of rank at least $d+1$ we have 
    \[
        \big(\forall U\subseteq V\colon\rho_\mathscr{M}(V-U)<d\implies A(v_0)\in\cone(A(U))\big)\implies\big(\exists\sigma\in\mathscr{M}\colon A(v_0)\in\cone(A(\sigma))\big),
    \]
    is false, as will be discussed in Observation \ref{obs:no-cone-caratheodory}. Therefore Theorem \ref{thm:caratheodory:om-cone-matroid} is not stronger than it: indeed, for this $A$ Theorem \ref{thm:caratheodory:om-cone-matroid} demands that 
    \[
        \forall U\subseteq V\colon \rho_\mathscr{M}(V-U)<d+1\implies A(v_0)\in \cone(A(U)),
    \]
    and so Theorem \ref{thm:caratheodory:om-cone-matroid} does not generalize Theorem \ref{thm:caratheodory:oriented-matroid}.

    Similarly performing the aforementioned transformation on Theorem \ref{thm:caratheodory:pair}, we get a version of Theorem \ref{thm:caratheodory:om-cone-pair} where only $r\geq d$ is assumed instead of $r\geq d+1$, but as we will note in Observation \ref{obs:no-cone-caratheodory}, this statement is also false. Therefore the only true statement, Theorem \ref{thm:caratheodory:om-cone-pair}, is not an immediate generalization of Theorem \ref{thm:caratheodory:pair}. 
\end{observation}

Next, let us note that Theorems \ref{thm:caratheodory:fixed}, \ref{thm:caratheodory:pair}, \ref{thm:caratheodory:oriented-matroid}, and \ref{thm:caratheodory:constrained} do not have satisfying cone versions.
\begin{observation}\label{obs:no-cone-caratheodory}
    Let $V$ be a finite set, $v_0\notin V$, $d\geq 0$, $A\colon{V\sqcup\{v_0\}}\to\R^d$ a function, $\mathscr{O}$ an oriented matroid of rank at most $d$ on set of elements $V\sqcup\{v_0\}$, and let $\mathcal{F}(U)$ stand for either ``$A(v_0)\in\cone(A(U))$'' or ``$v_0\in\conv(U)$''. Then the following analogues of Theorems \ref{thm:caratheodory:fixed}, \ref{thm:caratheodory:pair}, \ref{thm:caratheodory:oriented-matroid}, and \ref{thm:caratheodory:constrained} are all false for $r=d$, \ref{obs:no-cone-caratheodory:fixed} and \ref{obs:no-cone-caratheodory:constrained} are consequences of Theorem \ref{thm:caratheodory:om-cone} for $r\geq d+1$, and \ref{obs:no-cone-caratheodory:oriented-matroid} for $r\geq d+1$ follows from Theorem \ref{thm:caratheodory:om-cone-matroid}:
    \begin{enumerate}[label=\normalfont{(\arabic*)}]
        \item\label{obs:no-cone-caratheodory:fixed} If $U_1,\ldots,U_{r-1}\subseteq V$ are pairwise disjoint and nonempty such that $\mathcal{F}(U)$ for all $1\leq i\leq r-1$ then for any $u\in V-\bigcup_{i=1}^{r-1}(U_i)$ there are $u_1\in U_1,\ldots,u_{r-1}\in U_{r-1}$ with $\mathcal{F}(\{u_1,\ldots,u_{r-1},u\})$.
        \item\label{obs:no-cone-caratheodory:pair} If $U_1,\ldots,U_r\subseteq V$ are pairwise disjoint and nonempty such that $\mathcal{F}(U_i\cup U_j)$ for all $1\leq i<j\leq r$ then there are $u_1\in U_1,\ldots,u_r\in U_r$ with $\mathcal{F}(\{u_1,\ldots,u_r\})$.
        \item\label{obs:no-cone-caratheodory:oriented-matroid} If $\mathscr{M}$ is a matroid of rank at least $d$ on $V$ such that $\mathcal{F}(U)$ for all $U\subseteq V$ with $\rho_\mathscr{M}(V-U)<r-1$, then there is an independent set $U'$ of $\mathscr{M}$ with $\mathcal{F}(U')$. 
        \item\label{obs:no-cone-caratheodory:constrained} If $r\geq 2$, $U_1,\ldots,U_r\subseteq V$ are pairwise disjoint and nonempty such that $\mathcal{F}(U_i)$ for all $1\leq i\leq r$, and $\mathscr{K}$ is a simplicial complex on the ground set $V$ of the form $\mathscr{K}_{12}*U_3*\cdots*U_r$, where $\mathscr{K}_{12}$ is a connected simplicial subcomplex of $U_1*U_2$ with vertex set $U_1\cup U_2$ and the sets $U_i$ are viewed as discrete sets of vertices, then there is a face $\sigma\in\mathscr{K}$ with $\mathcal{F}(\sigma)$. 
    \end{enumerate}
\end{observation}
\begin{proof}
    If $r\geq d+1$, then the collection $(U_i)_{1\leq i\leq r-1}$ of \ref{obs:no-cone-caratheodory:fixed} is easily seen to satisfy the assumptions of Theorem \ref{thm:caratheodory:om-cone}, and hence there are $u_1\in U_1,\ldots,u_{r-1}\in U_{r-1}$ with $\mathcal{F}(\{u_1,\ldots,u_{r-1}\})$, but then of course $\mathcal{F}(\{u_1,\ldots,u_{r-1},u\})$ also holds for any $u\in V$. On the other hand, if $r$ is still at least $d+1$ then Theorem \ref{thm:caratheodory:om-cone} is also applicable to the collection $(U_i)_{2\leq i\leq r}$ extracted from \ref{obs:no-cone-caratheodory:constrained}, and therefore there are $u_2\in U_2,\ldots,u_r\in U_r$ with $\mathcal{F}(\{u_2,\ldots,u_r\})$, but $\{u_2,\ldots,u_r\}\in U_2*\cdots*U_r\subseteq\mathscr{K}$. In conclusion, \ref{obs:no-cone-caratheodory:fixed} and \ref{obs:no-cone-caratheodory:constrained} are indeed corollaries of Theorem \ref{thm:caratheodory:om-cone}. Similarly, if $r\geq d+1$ then we do not even need $\mathscr{M}$ to be of rank at least $d$ to see that the assumptions of \ref{obs:no-cone-caratheodory:oriented-matroid} imply those of \ref{thm:caratheodory:om-cone-matroid}, but the two conclusions are the same, and as such \ref{obs:no-cone-caratheodory:oriented-matroid} is a corollary of Theorem \ref{thm:caratheodory:om-cone-matroid} in this case.
    
    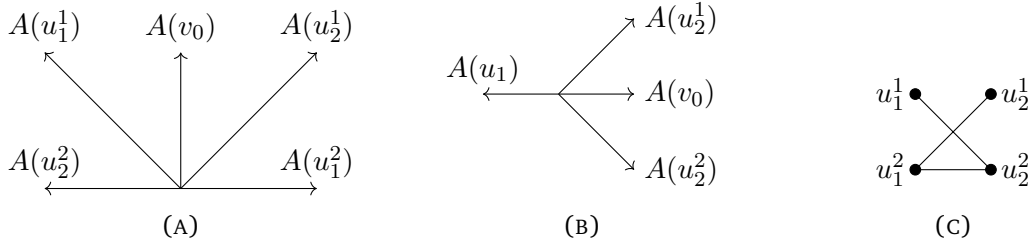
\begin{figure}
        \centering
        \begin{subfigure}{0.35\textwidth}
            \centering
            \begin{tikzpicture}[scale=1.8]
                \draw[->] (0,0) -- (0,1)    node[anchor=south]  {$A(v_0)$};
                \draw[->] (0,0) -- (-1,1)   node[anchor=south]  {$A(u_1^1)$}; 
                \draw[->] (0,0) -- (1,0)    node[anchor=south]  {$A(u_1^2)$};
                \draw[->] (0,0) -- (1,1)    node[anchor=south]  {$A(u_2^1)$};
                \draw[->] (0,0) -- (-1,0)   node[anchor=south]  {$A(u_2^2)$};
            \end{tikzpicture}
            \caption{}
            \label{fig:cone-counterexamples:fixed-and-constrained}
        \end{subfigure}
        \begin{subfigure}{0.3\textwidth}
            \centering
            \begin{tikzpicture}
                \draw[->] (0,0) -- (1,0)    node[anchor=west]  {$A(v_0)$};
                \draw[->] (0,0) -- (-1,0)   node[anchor=south]   {$A(u_1)$};
                \draw[->] (0,0) -- (1,1)    node[anchor=west]  {$A(u_2^1)$};
                \draw[->] (0,0) -- (1,-1)   node[anchor=west]  {$A(u_2^2)$};
            \end{tikzpicture}
            \caption{}
            \label{fig:cone-counterexamples:pair}
        \end{subfigure}
        \begin{subfigure}{0.3\textwidth}
            \centering
            \begin{tikzpicture}
                \draw (0,1) -- (1,0) -- (0,0) -- (1,1);
                \filldraw[black] (0,1) circle (2pt) node[anchor=east] {$u_1^1$};
                \filldraw[black] (0,0) circle (2pt) node[anchor=east] {$u_1^2$};
                \filldraw[black] (1,1) circle (2pt) node[anchor=west] {$u_2^1$};
                \filldraw[black] (1,0) circle (2pt) node[anchor=west] {$u_2^2$};
            \end{tikzpicture}
            \caption{}
            \label{fig:cone-counterexamples:K}
        \end{subfigure}
        \caption{The ingredients used in disproving some cone colorful Carath\'eodory theorems when $r=d$. $U_3,\ldots,U_r$ and their images under $A$ are not illustrated.}
        \label{fig:cone-counterexamples}
    \end{figure}
    The counterexamples for \ref{obs:no-cone-caratheodory:fixed}, \ref{obs:no-cone-caratheodory:pair}, \ref{obs:no-cone-caratheodory:oriented-matroid}, and \ref{obs:no-cone-caratheodory:constrained} when $r=d$ are illustrated in Figure \ref{fig:cone-counterexamples} and explained here. Let $e_1,e_2,\ldots,e_d$ be the standard basis of $\R^d$, and first consider the function $A$ illustrated in Figure \ref{fig:cone-counterexamples:fixed-and-constrained}, namely let $U_i:=\{u_i^1,u_i^2\}$ for $1\leq i\leq r$, $A(v_0):=e_1$, $A(u_1^1):=e_1-e_2$, $A(u_1^2):=e_2$, $A(u_2^1):=e_1+e_2$, $A(u_2^2):=-e_2$, and for $3\leq i\leq d$ put $A(u_i^1):=e_1+e_i$ and $A(u_i^2):=e_1-e_i$. It is easily seen that $A(v_0)\in\conv(A(U_i))$ for all $1\leq i\leq d$, yet if $u:=u_1^2$ then there are no $u_2\in U_2,\ldots,u_d\in U_d$ such that $A(v_0)\in\cone(\{u,u_2,\ldots,u_r\})$, so \ref{obs:no-cone-caratheodory:fixed} does not hold if $r=d$. A related computation shows that if $\sigma\in\mathscr{K}_{12}*U_3*\cdots*U_d$ for the $\mathscr{K}_{12}$ illustrated in Figure \ref{fig:cone-counterexamples:K} then $A(v_0)\notin\cone(A(\sigma))$, disproving \ref{obs:no-cone-caratheodory:constrained} when $r=d$. 

    Moving on to point \ref{obs:no-cone-caratheodory:pair}, take the function $A$ illustrated in Figure \ref{fig:cone-counterexamples:pair}, or in other words let $U_1:=\{u_1\}$, $U_i:=\{u_i^1,u_i^2\}$ for $2\leq i\leq d$, $A(v_0):=e_1$, $A(u_1):=-e_1$, and for $2\leq i\leq d$ put $A(u_i^1):=e_1+e_i$ and $A(u_i^2):=e_1-e_i$. In this case $A(v_0)\in\cone(A(U_i))$ for $2\leq i\leq d$, so $A(v_0)\in\cone(A(U_i\cup U_j))$ for $1\leq i<j\leq d$, yet there are no $u_1\in U_1,\ldots,u_d\in U_d$ with $A(v_0)\in\cone(A(\{u_1,\ldots,u_r\}))$, which means \ref{obs:no-cone-caratheodory:pair} is no longer true if $r=d$.

    Finally, it is left to see that \ref{obs:no-cone-caratheodory:oriented-matroid} is false for $r=d$. This is the case, because point \ref{obs:no-cone-caratheodory:pair} with $r=d$ is a corollary of this claim if we set $\mathscr{M}:=U_1*\cdots*U_d$, where each $U_i$ is viewed as a discrete set of vertices. Indeed, if $d-1>\rho_\mathscr{M}(V-U)=|\{1\leq i\leq d: U_i\cap (V-U)\neq\emptyset\}|$ then $U_i\cap(V-U)=\emptyset$ for at least two distinct values of $i$, but then $U\supseteq U_i\cup U_j$ for some $1\leq i< j\leq d$, which means that $\mathcal{F}(U)$ holds as a consequence of $\mathcal{F}(U_i\cup U_j)$. Therefore the assumptions of \ref{obs:no-cone-caratheodory:oriented-matroid} are satisfied under the conditions of \ref{obs:no-cone-caratheodory:pair} if $r=d$, so if \ref{obs:no-cone-caratheodory:oriented-matroid} was true then we would have an independent set $U'$ of $U_1*\cdots*U_d$ with $\mathcal{F}(U')$, and therefore a maximal independent set satisfying $\mathcal{F}$ as well, but maximal independent sets of $U_1*\cdots*U_d$ are of the form $\{u_1,\ldots,u_r\}$ with $u_1\in U_1,\ldots,u_r\in U_r$. All in all, this means that if $r=d$ then \ref{obs:no-cone-caratheodory:oriented-matroid} indeed implies \ref{obs:no-cone-caratheodory:pair}, which is already seen to be false.
\end{proof}

Nevertheless, as illustrated by Theorems \ref{thm:caratheodory:om-cone-matroid} and \ref{thm:caratheodory:om-cone-pair}, Theorems \ref{thm:caratheodory:oriented-matroid} and \ref{thm:caratheodory:pair} still have non-trivial analogues for cones.

While the constrained colorful Carath\'eodory theorem (Theorem \ref{thm:caratheodory:constrained}) does not admit a meaningful generalization for cones, it is worthwhile to ask if the restriction it gives on the complex $\mathscr{K}=\mathscr{K}_{12}*U_3*\cdots*U_r$ can be strengthened further. Replacing $U_1*U_2$ by a connected subcomplex $\mathscr{K}_{12}$ did not change the homological connectivity of $\mathscr{K}$, which is the key property enabling the method described in Section \ref{sec:the-method} to still be applicable for this $\mathscr{K}$. If $U_3*U_4$ was simultaneously replaced with a connected subcomplex $\mathscr{K}_{34}\subseteq U_3*U_4$, then the homological connectivity of $\mathscr{K}$ would continue to be the same, so it is natural to ask if performing this still yields a true statement. The answer is unfortunately no, which is explained by the observation below.
\begin{observation}\label{obs:no-further-grouping-in-constrained}
    Theorem \ref{thm:caratheodory:constrained} can not be generalized further to complexes of the form $\mathscr{K}=\mathscr{K}_{12}*\mathscr{K}_{34}*U_5*\cdots*U_r$ with $\mathscr{K}_{12}\subseteq U_1*U_2$ and $\mathscr{K}_{34}\subseteq U_3*U_4$ connected containing all vertices of the joins. More precisely, there are $d\geq 0$, $A:V\to\R^d$, $r\geq\max\{d+1,4\}$, $U_1,\ldots,U_r\subseteq V$ pairwise disjoint such that $0\in\conv(A(U_1))\cap\cdots\cap\conv(A(U_r))$, and $\mathscr{K}_{12}\subseteq U_1*U_2$, $\mathscr{K}_{34}\subseteq U_3*U_4$ connected with vertex sets $U_1\cup U_2$ and $U_3\cup U_4$ respectively, such that $\mathscr{K}_{12}*\mathscr{K}_{34}*U_5*\cdots*U_r$ has no face $\sigma$ with $0\in\conv(A(\sigma))$.

    Namely, let $V$ be $\{-1,1\}^3$, $d:=3$, $A:=\id\colon V\hookrightarrow\R^3$, $r:=4$, $U_1,\ldots,U_4\subseteq V$ the four pairs of antipodal vertices, and $\mathscr{K}_{12}\subseteq U_1*U_2$, $\mathscr{K}_{34}\subseteq U_3*U_4$ be obtained by deleting the edge whose endpoints both have first coordinate equal to $-1$. See Figure \ref{fig:cube} for an illustration.
    \begin{figure}[ht]
        \centering
        \begin{tikzpicture}[scale=0.7]
            \draw[lightgray] (0,3) -- (2,4) -- (5,4) -- (5,1) -- (2,1) -- (2,4);
            \draw[black, thick] (2,1) -- (5,1) -- (3,3);
            \draw[white, line width=3pt] (0,3) -- (3,3) -- (3,0);
            \draw[lightgray] (2,1) -- (0,0) -- (0,3) -- (3,3) -- (3,0) -- (0,0);
            \draw[lightgray] (3,3) -- (5,4) -- (5,1) -- (3,0);
            \draw[black, thick] (3,3) -- (0,3);
            \draw[black, thick] (0,0) -- (3,0) -- (5,4) -- (2,4);
            \node[circle, fill, inner sep=2pt] at (0,0) {};
            \node[circle, fill, inner sep=2pt] at (5,4) {};
            \node[rectangle, fill, inner sep=3pt] at (3,0) {};
            \node[rectangle, fill, inner sep=3pt] at (2,4) {};
            \node[star, star points=5, fill, inner sep=2pt] at (0,3) {};
            \node[star, star points=5, fill, inner sep=2pt] at (5,1) {};
            \node[regular polygon, regular polygon sides=3, fill, inner sep=1.7pt] at (3,3) {};
            \node[regular polygon, regular polygon sides=3, fill, inner sep=1.7pt] at (2,1) {};
        \end{tikzpicture}
        \caption{The example of Observation \ref{obs:no-further-grouping-in-constrained}, with elements of $U_1,\ldots,U_4$ drawn using different shapes, and $\mathscr{K}_{12}$, $\mathscr{K}_{34}$ illustrated.}
        \label{fig:cube}
    \end{figure}
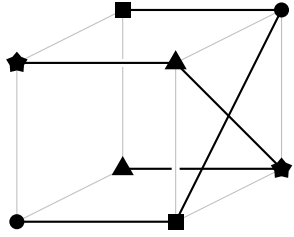
\end{observation}
On the other hand, instead of trying to substitute joins of two color classes $U_i*U_{i+1}$ with connected subcomplexes $\mathscr{K}_{i,i+1}\subseteq U_i*U_{i+1}$, one could attempt to generalize Theorem \ref{thm:caratheodory:constrained} by replacing $U_1*U_2*U_3$ with a suitable homologically $1$-connected subcomplex $\mathscr{K}_{123}\subseteq U_1*U_2*U_3$. The next theorem shows that such an extension should be possible, but we do not attempt to derive a more general formulation of this example in this paper.
\begin{theorem}\label{thm:caratheodory:3-joined}
    Let $V:=U_1\sqcup U_2\sqcup U_3$ with $U_i:=\{u_i^1,u_i^2,u_i^3\}$ for $1\leq i\leq 3$, and $\mathscr{K}:=\mathscr{K}_{123}\subseteq U_1*U_2*U_3$ be the subcomplex obtained by deleting the three edges $\{u_2^1,u_3^1\}$, $\{u_3^2,u_1^2\}$, and $\{u_1^3,u_2^3\}$, as well as all faces containing these. If $d\leq 2$ and $A\colon V\to\R^d$ is an arbitrary function such that $0\in\conv(A(U_1))\cap\conv(A(U_2))\cap\conv(A(U_3))$, then there is a face $\sigma\in\mathscr{K}$ such that $0\in\conv(A(\sigma))$.
\end{theorem}
The proof is postponed to Section \ref{sec:cc-proofs}; here we simply note that this theorem does not follow immediately from Theorems \ref{thm:caratheodory:matroid}, \ref{thm:caratheodory:oriented-matroid}, or \ref{thm:caratheodory:constrained}, so it is indeed a novel kind of generalization, whose analogues are worth exploring in the future. During the proof of this latter claim, we will need to construct a particular $A:=A_{u_1,u_2,u_3}\colon V\to\R^2$ where $0\notin\conv(A(V-\{u_1,u_2,u_3\}))$, $0\notin\conv(A(\{u_1,u_2,u_3\}))$, and $0\in\conv(A(U_1))\cap\conv(A(U_2))\cap\conv(A(U_3))$, but this is doable for any $u_1\in U_1, u_2\in U_2,u_3\in U_3$. Namely, color each $U_i$ using all three of exactly $3$ colors so that $\{u_1,u_2,u_3\}$ becomes a color class, and let $A_{u_1,u_2,u_3}$ send each color class to a different element of $\{v_1,v_2,v_3\}$ where $0\in\interior(\conv(\{v_1,v_2,v_3\}))$.
\begin{remark}
    Let $d:=2$. There is no subcomplex of the $\mathscr{K}$ of Theorem \ref{thm:caratheodory:3-joined} to which either of Theorems \ref{thm:caratheodory:matroid}, \ref{thm:caratheodory:oriented-matroid}, or \ref{thm:caratheodory:constrained} could be applied (after possibly relabelling the sets $U_1,U_2,U_3$) in a way that its assumptions are implied by the assumptions of Theorem \ref{thm:caratheodory:3-joined}. In particular, Theorem \ref{thm:caratheodory:3-joined} is not an immediate consequence of any of these three theorems. For the sake of brevity, we jump over some details while justifying this claim.

    Assume $\mathscr{M}\subseteq\mathscr{K}$ is a subcomplex for which the assumptions of Theorem \ref{thm:caratheodory:matroid} are implied by those of Theorem \ref{thm:caratheodory:3-joined}. First, note that the rank of $\mathscr{M}$ is $3$, or otherwise there would be a maximal face of $\mathscr{M}$ contained in a set of the form $\{u_1,u_2,u_3\}$ with $u_1\in U_1,u_2\in U_2,u_3\in U_3$, but then $A_{u_1,u_2,u_3}$ satisfies the assumptions of Theorem \ref{thm:caratheodory:3-joined} but fails to do the same for Theorem \ref{thm:caratheodory:matroid}, particularly for $U:=\{u_1,u_2,u_3\}$. But even if the rank of $\mathscr{M}$ is $3$, $A_{u_1^k,u_2^k,u_3^k}$ testifies for $1\leq k\leq 3$ that $\rho_\mathscr{M}(V-\{u_1^k,u_2^k,u_3^k\})<\rho_\mathscr{M}(V)=3$, because $\rho_\mathscr{M}(\{u_1^k,u_2^k,u_3^k\})\leq 2=d$ due to the removed edges, and $0\notin\conv(A_{u_1^k,u_2^k,u_3^k}(V-\{u_1^k,u_2^k,u_3^k\}))$. In particular, all 2-dimensional faces of $\mathscr{M}$ must be of the form $\{u_1^i,u_2^j,u_3^k\}$ where $i,j,k$ are pairwise distinct. Using the definition of matroids and the fact that $\mathscr{M}$ has rank $3$, this also implies that any edge $\{u_a^b,u_c^d\}$ must satisfy $b\neq d$. Again by the definition of matroids, we get that for any such edge, $u_a^d$ and $u_c^b$ are not vertices of $\mathscr{M}$. The fact that $\mathscr{M}$ has a two-dimensional face, which is of the form $\{u_1^i,u_2^j,u_3^k\}$ with $i,j,k$ are pairwise distinct, implies that outside this face $\mathscr{M}$ contains no other vertices. But then $A_{u_1^i,u_2^j,u_3^k}$ satisfies the assumptions of Theorem \ref{thm:caratheodory:3-joined} but not those of Theorem \ref{thm:caratheodory:matroid}, particularly for $U:=\{u_1^i,u_2^j,u_3^k\}$.

    Let us move on to Theorem \ref{thm:caratheodory:oriented-matroid}, and make two observations:
    \begin{enumerate}[label=\normalfont{(\arabic*)}]
        \item\label{en:3-joined-not-OM:rk-2} If $u_1\in U_1,u_2\in U_2,u_3\in U_3$ then $\rho_\mathscr{M}(\{u_1,u_2,u_3\})\geq 2$. Otherwise $A_{u_1,u_2,u_3}$ satisfies the assumptions of Theorem \ref{thm:caratheodory:3-joined}, but not those of Theorem \ref{thm:caratheodory:oriented-matroid}, because $0\notin\conv(A_{u_1,u_2,u_3}(V-\{u_1,u_2,u_3\}))$.
        \item\label{en:3-joined-not-OM:forbidden-edges} If $u_i\in U_i,u_j\in U_j$, $\{u_i,u_j\}\notin\mathscr{M}$, $\{u_i\}\in\mathscr{M}$, and $i\neq j$, then $\{u_i',u_j\}\notin\mathscr{M}$ for all $u_i'\in U_i$. For this, use the fact that $\mathscr{M}\subseteq\mathscr{K}\subseteq U_1*U_2*U_3$ and the definition of matroids.
    \end{enumerate}
    Use point \ref{en:3-joined-not-OM:rk-2} to see that at least one element of $\{u_2^1,u_3^1\}$ is a vertex of $\mathscr{M}$. If $u_2^1$ is a vertex of $\mathscr{M}$, then point \ref{en:3-joined-not-OM:forbidden-edges} applied to the non-edge $\{u_2^1,u_3^1\}$ tells us that $\{u_2^3,u_3^1\}\notin\mathscr{M}$, but then point \ref{en:3-joined-not-OM:rk-2} for the selection $\{u_1^3,u_2^3,u_3^1\}$ guarantees that $\{u_1^3,u_3^1\}\in\mathscr{M}$, in particular both $u_2^1$ and $u_3^1$ are vertices of $\mathscr{M}$. The same holds if $u_3^1$ was a vertex of $\mathscr{M}$ instead, as then point \ref{en:3-joined-not-OM:forbidden-edges} applied to the non-edge $\{u_2^1,u_3^1\}$ tells us that $\{u_2^1,u_3^2\}\notin\mathscr{M}$, but then point \ref{en:3-joined-not-OM:rk-2} for the selection $\{u_1^2,u_2^1,u_3^2\}$ guarantees $\{u_1^2,u_2^1\}\in\mathscr{M}$. But if both $u_2^1$ and $u_3^1$ are vertices of $\mathscr{M}$, then point \ref{en:3-joined-not-OM:forbidden-edges} applied to the non-edge $\{u_2^1,u_3^1\}$ tells us that $\mathscr{M}[U_2\cup U_3]$ contains no edges incident to either of the vertices $u_2^1$ or $u_3^1$, meaning that (as an application of the definition of matroids) $\rho_\mathscr{M}(U_2\cup U_3)\leq1$. But then it is easily seen that the rank of $\mathscr{M}$ can not be at least $3$ as assumed by Theorem \ref{thm:caratheodory:oriented-matroid}.

    Lastly, Theorem \ref{thm:caratheodory:constrained} constructs a certain simplicial complex $\mathscr{K}_{12}*U_3$ where $\mathscr{K}_{12}$ has vertex set $U_1\cup U_2$, in particular there is an $1\leq i\leq 3$ such that $\{u_j,u_i\}$ is an edge for all $j\neq i$ and $u_i\in U_i,u_j\in U_j$. A subcomplex of $\mathscr{K}$ satisfying the assumptions of Theorem \ref{thm:caratheodory:constrained} under the conditions of Theorem \ref{thm:caratheodory:3-joined} must have vertex set $U_1\sqcup U_2\sqcup U_3$. However, no such subcomplex of $\mathscr{K}$ satisfies the property mentioned at the beginning of this paragraph because of the three deleted edges $\{u_2^1,u_3^1\}$, $\{u_3^2,u_1^2\}$, and $\{u_1^3,u_2^3\}$, so Theorem \ref{thm:caratheodory:constrained} indeed can not be applied to any subcomplex of $\mathscr{K}$.
\end{remark}

Finally, the constrained colorful Carath\'eodory theorem (Theorem \ref{thm:caratheodory:constrained}) implies the following variation of Tverberg's theorem, taken from \cite[Corollary 1.7]{Blagojevic2025}. Here, let $[n]$ stand for the set $\{1,\ldots,n\}$, sometimes viewed as a discrete collection of $n$ vertices, and if $\mathscr{X}$ is a simplicial complex on ground set $V$ then write $\{a\}\times \mathscr{X}$ for the corresponding simplicial complex on ground set $\{a\}\times V$.
\begin{corollary}\label{crl:new-tverberg}
    Let $r\geq 2$, $d\geq 0$, and $N\geq (d+1)(r-1)$ be integers, $A\colon [N+1]\to\R^d$ a function, and $\mathscr{K}_{12}$ a path-connected subcomplex of $(\{1\}\times[r])*(\{2\}\times[r])$ with vertex set $\{1,2\}\times[r]$. Then there are pairwise disjoint subsets $\sigma_1,\ldots,\sigma_r\subseteq [N+1]$ with $\conv(A(\sigma_1))\cap\cdots\cap\conv(A(\sigma_r))\neq\emptyset$ and $(\sigma_1\times\{1\})\cup\cdots\cup(\sigma_r\times\{r\})\in\mathscr{K}_{12}*(\{3\}\times[r])*\cdots*(\{N+1\}\times [r])$.
\end{corollary}
In other words, this theorem guarantees the existence of pairwise disjoint subsets $\sigma_1,\ldots,\sigma_r\subseteq [N+1]$ with $\conv(A(\sigma_1))\cap\cdots\cap\conv(A(\sigma_r))\neq\emptyset$ satisfying a certain property. Whether this property is indeed true might depend on the ordering of the $r$ subsets $\sigma_1,\ldots,\sigma_r$. If we do not care about the order of the $\sigma_1,\ldots,\sigma_r$, then the corollary above can still be used to impose an additional condition besides $\conv(A(\sigma_1))\cap\cdots\cap\conv(A(\sigma_r))\neq\emptyset$, as formulated below. When $r$ is prime power, this is a special case of the work of Schöneborn \& Ziegler \cite{SchoenebornZiegler2005}, further generalized by Hell \cite{Hell2008}, and finally the Optimal colored Tverberg theorem of Blagojević, Matschke \& Ziegler \cite{Blagojevic2009,Blagojevic2011-2}. In fact, the Optimal colored Tverberg theorem, by virtue of containing the results of Schöneborn \& Ziegler as well as Hell, has so far been the only strengthening of the original result of Tverberg \cite{Tverberg1966}, in which the additional condition on $\sigma_1,\ldots,\sigma_r\subseteq[N+1]$ besides $\conv(A(\sigma_1))\cap\cdots\cap\conv(A(\sigma_r))\neq\emptyset$ is non-trivial even if $N=(r-1)(d+1)$. While the statements above and below do not extend the Optimal colored Tverberg theorem, the latter only works when $r$ is a prime power, so they are not contained in it either, providing novel generalizations imposing non-trivial conditions on the Tverberg partition $\sigma_1,\ldots,\sigma_r\subseteq[N+1]$.
\begin{corollary}\label{crl:new-unordered-tverberg}
    Let $r\geq 3$, $d\geq 0$, and $N\geq (d+1)(r-1)$ be integers, and $A\colon [N+1]\to\R^d$ a function. Then there are pairwise disjoint subsets $\sigma_1,\ldots,\sigma_r\subseteq[N+1]$ with $\conv(A(\sigma_1))\cap\cdots\cap\conv(A(\sigma_r))\neq\emptyset$, such that $\{1,2\}\nsubseteq\sigma_i$ for any $1\leq i\leq r$.
\end{corollary}
We first derive Corollary \ref{crl:new-unordered-tverberg}, then spell out the proof of Corollary \ref{crl:new-tverberg}.
\begin{proof}[Proof of Corollary \ref{crl:new-unordered-tverberg}]
    Let $\mathscr{K}_{12}$ be the subcomplex of $(\{1\}\times[r])*(\{2\}\times [r])$ obtained by deleting all edges between the vertices $(1,i)$ and $(2,i)$ for all $1\leq i\leq r$. When $r\geq 3$, $\mathscr{K}_{12}$ is easily seen to be path-connected, and so Corollary \ref{crl:new-tverberg} guarantees that there are pairwise disjoint subsets $\sigma_1,\ldots,\sigma_r\subseteq[N+1]$ with $\conv(A(\sigma_1))\cap\cdots\cap\conv(A(\sigma_r))\neq\emptyset$ and $(\sigma_1\times\{1\})\cup\cdots\cup(\sigma_r\times\{r\})\in\mathscr{K}_{12}*(\{3\}\times[r])*\cdots*(\{N+1\}\times[r])$. If we had $\{1,2\}\subseteq\sigma_i$ for some $1\leq i\leq r$, then this would mean that the edge connecting $(1,i)$ to $(2,i)$ is a face of $(\sigma_1\times\{1\})\cup\cdots\cup(\sigma_r\times\{r\})$, and therefore of $\mathscr{K}_{12}*(\{3\}\times[r])*\cdots*(\{N+1\}\times[r])$ and also of $\mathscr{K}_{12}$, which is a contradiction.
\end{proof}
\begin{remark}
    Corollary \ref{crl:new-tverberg} can not be used to immediately imply a stronger statement about $\{\sigma_1,\ldots,\sigma_r\}$ other than the one written in Corollary \ref{crl:new-unordered-tverberg}. Indeed, if $\Sigma$ is a collection of $r$ many pairwise disjoint subsets of ${[N+1]}$ such that the convex hulls of their images under $A$ have a common point, and $\mathscr{K}_{12}$ is arbitrary like in Corollary \ref{crl:new-tverberg}, then we can nearly always label the elements of $\Sigma$ with $\sigma_1,\ldots,\sigma_r$ such that $(\sigma_1\times\{1\})\cup\cdots\cup(\sigma_r\times\{r\})$ becomes a face of $\mathscr{K}_{12}*(\{3\}\times[r])*\cdots*(\{N+1\}\times[r])$. The only thing we have to pay attention to is that if $1\in\sigma_{i_1}$ and $2\in\sigma_{i_2}$ then there must be an edge in $\mathscr{K}_{12}$ between $(1,i_1)$ and $(2,i_2)$. Any admissible $\mathscr{K}_{12}$ has an edge the between some $(1,j_1)$ and $(2,j_2)$ with $j_1\neq j_2$ if $r\geq 2$, so if $\{1,2\}\notin\sigma_i$ for any $1\leq i\leq r$ then we can just pick $i_1:=j_1$, $i_2:=j_2$, and choose arbitrarily for $i_3,\ldots,i_r$. On the other hand, if $\{1,2\}\subseteq\sigma_i$ then we can still set $i:=j$ if $\mathscr{K}_{12}$ contains an edge between $(1,j)$ and $(2,j)$ for some $j$. In conclusion, the only non-trivial restriction $\mathscr{K}_{12}$ can pose on $\Sigma$ is to prohibit $\{1,2\}\subset\sigma_i$, which it can do by Corollary \ref{crl:new-unordered-tverberg}.
 \end{remark}
On the other hand, Corollary \ref{crl:new-tverberg} follows from Theorem \ref{thm:caratheodory:constrained}, together with Sarkaria's Lemma, whose formulation we take from \cite[Lemma 13.1]{Barany2021}. 
\begin{lemma}[Sarkaria's Lemma]\label{lem:sarkaria}
    Let $W$ be a finite set, $A\colon W\to\R^d$ a function, and $W=W_1\sqcup\cdots\sqcup W_r$ a partition. Put $j(w):=j$ if $w\in W_j$, take $v_1,\ldots,v_r\in\R^{r-1}$ such that any $r-1$ of them are independent and $v_1+\cdots+v_r=0$, and define $\hat{A}\colon W\to\R^{(r-1)(d+1)}\cong\R^{r-1}\otimes\R^{d+1}$ by
    \[
        \hat{A}(w):=v_{j(w)}\otimes
        (A(w),1).
    \] 
    Then $\emptyset=\conv(A(W_1))\cap\cdots\cap\conv(A(W_r))$ if and only if $0\notin\conv(\hat{A}({W}))$. 
\end{lemma}
\begin{proof}[Proof of Corollary \ref{crl:new-tverberg}]
    Take $v_1,\ldots,v_r$ like in Lemma \ref{lem:sarkaria}, and define the auxiliary function $A'\colon {[N+1]\times[r]}\to\R^{r-1}\otimes\R^{d+1}$ using the formula
    \[
        A'((i,j)):=v_{j}\otimes
        (A(i),1).
    \]
    As $v_1+\cdots+v_r=0$, necessarily 
    $0=\sum_{j=1}^rr^{-1}v_j\otimes(A(i),1)\in\conv(A'(\{i\}\times [r]))$ 
    for every $1\leq i\leq N+1$. This means that as $N+1\geq (r-1)(d+1)+1$, the constrained colorful Carath\'eodory theorem (Theorem \ref{thm:caratheodory:constrained}) can be applied for $A'$ with $U_1:=\{1\}\times[r],\ldots,U_{N+1}:=\{N+1\}\times[r]$ to get a $\sigma\in\mathscr{K}_{12}*(\{3\}\times[r])*\cdots*(\{N+1\}\times[r])$ with $0\in\conv(A'(\sigma))$. Now put $W_j:=\{1\leq i\leq N+1:(i,j)\in\sigma\}$ for $1\leq j\leq r$, and let $W:=W_1\sqcup\cdots\sqcup W_r$, so that $A'((i,j))$ for $(i,j)\in \sigma$ is equal to the $\hat{A}(i)$ defined in Lemma \ref{lem:sarkaria}. In particular, Lemma just mentioned tells us that $\emptyset\neq\conv(A(W_1))\cap\cdots\cap\conv(A(W_r))$ if and only if $0\in\conv(A'(\sigma))$, where the latter claim is true by definition of $\sigma$, so $\sigma_1:=W_1,\ldots,\sigma_r:=W_r$ satisfies the desired properties.
\end{proof}

\bigskip
\section{Outline of the method}\label{sec:the-method}
\bigskip

In this section we reformulate the method -- originally due to Kalai \& Meshulam \cite{KalaiMeshulam2005} and extended by Holmsen \cite{Holmsen2016} -- for more general complexes than just matroids, in order to apply it in Section \ref{sec:cc-proofs} to prove all of our central results. The presentation here intends to highlight some interesting points where more complicated tools could be substituted in in order to extend the computations to less well-behaved settings.

Recall from the previous section that each colorful Carath\'eodory theorem is parametrized by a logical formula $\mathcal{F}(-)$ and two collections $\mathfrak{U}$ and $\mathfrak{U}'$ of subsets of $V$, in terms of which the theorem states that ``if $\mathcal{F}(U)$ for all $U\in \mathfrak{U}$ then $\mathcal{F}(U')$ for some $U'\in\mathfrak{U}'$ too''. The following
\[
    \mathscr{C}_\mathcal{F}:=\{U\subseteq V: \mathcal{F}(U)=\mathrm{false}\}
\]
is an abstract simplicial complex for each of the four $\mathcal{F}(-)$ that we consider:
\begin{enumerate}[label=\normalfont{(\arabic*)}]
    \item\label{en:complex:zero-avoiding} $\mathcal{F}(U)=$``$0\in\conv(A(U))$'' for a fixed function $A\colon V\to\R^d$,
    \item\label{en:complex:vector-avoiding} $\mathcal{F}(U)=$``$A(v_0)\in\cone(A(U))$'' for a fixed function $A\colon {V\sqcup\{v_0\}}\to\R^d$,
    \item\label{en:complex:support} $\mathcal{F}(U)=$``$U$ contains a positive circuit of $\mathscr{O}$'' for a fixed oriented matroid $\mathscr{O}$ of rank at most $d$ on set of elements $V$,
    \item\label{en:complex:elem-avoiding} $\mathcal{F}(U)=$``$v_0\in\conv(U)$'' interpreted in a fixed oriented matroid $\mathscr{O}$ of rank at most $d$ on set of elements $V\sqcup\{v_0\}$.
\end{enumerate}
In \cite{Holmsen2016} the $\mathscr{C}_\mathcal{F}$ of point \ref{en:complex:support} was called the \emph{support complex} of the oriented matroid $\mathscr{O}$. We will refer to the $\mathscr{C}_\mathcal{F}$ of points \ref{en:complex:zero-avoiding}, \ref{en:complex:vector-avoiding}, and \ref{en:complex:elem-avoiding} as the \emph{zero-avoiding}, \emph{vector-avoiding}, and \emph{element-avoiding} complex respectively.

As $\mathscr{C}_\mathcal{F}$ is an abstract simplicial complex, replacing $\mathfrak{U}'$ in a colorful Carath\'eodory theorem with the smallest abstract simplicial complex $\mathscr{K}$ containing it gives the equivalent statement ``if $\mathcal{F}(U)$ for all $U\in\mathfrak{U}$ then $\mathcal{F}(U')$ for some $U'\in\mathscr{K}$'', which can be further rephrased as ``if $\mathfrak{U}\cap\mathscr{C}_\mathcal{F}=\emptyset$ then $\mathscr{K}\nsubseteq\mathscr{C}_\mathcal{F}$''. Statements of this last form have already been established by Kalai \& Meshulam \cite[Theorem 1.6]{KalaiMeshulam2005} and Holmsen \cite[Theorem 1.4]{Holmsen2016}, where $\mathscr{K}$ is demanded to be a matroid and $\mathscr{C}_\mathcal{F}$ satisfies one of the following two properties:
\begin{definition}
    A simplicial complex $\mathscr{C}$ is \emph{near-$(d-1)$-Leray (with coefficients in the field $\mathbb{F}$)} if $\tilde{H}^n(\mathscr{C};\mathbb{F})=0$ for all $n\geq d$ and $\tilde{H}^n(\link_{\mathscr{C}}(\sigma);\mathbb{F})=0$ for each simplex $\emptyset\neq\sigma\in\mathscr{C}$ and $n\geq d-1$. $\mathscr{C}$ is \emph{$(d-1)$-Leray} (with coefficients in the field $\mathbb{F}$) if $\tilde{H}^n(\link_{\mathscr{C}}(\sigma);\mathbb{F})=0$ for all $n\geq d-1$ and $\emptyset\subseteq\sigma\in\mathscr{C}$.
\end{definition}
Here $\link_{\mathscr{C}}(\sigma)$ denotes the \emph{link} of $\sigma$ in the abstract simplicial complex $\mathscr{C}$, which is the subcomplex of $\mathscr{C}$ defined as
\[
    \link_{\mathscr{C}}(\sigma):=\{\tau\in\mathscr{C}:\tau\cap\sigma=\emptyset,\tau\cup\sigma\in\mathscr{C}\}.
\]
With coefficients in a fixed field $\mathbb{F}$, the cohomology $H^n(X)$ of any space $X$ is naturally isomorphic to the dual $H_n(X)^*$ of the homology $H_n(X)$ \cite[Corollary 53.6]{Munkres1984}, so there is an isomorphism $H^n(X)\cong H_n(X)$ whenever the latter is finite dimensional. From this it easily follows that $\tilde{H}^n(X)\cong\tilde{H}_n(X)$ whenever the latter is finite dimensional. In particular, near-$(d-1)$-Lerayness and $(d-1)$-Lerayness could have been defined using reduced homology groups in place of reduced cohomology groups, like it is done in e.g. \cite{Holmsen2016}. We frame the definition differently here, because this is the version which we will always use.

\cite[Proposition 1.3]{Holmsen2016} states that the support complex of an oriented matroid $\mathscr{O}$ is near-$(\operatorname{rank}(\mathscr{O})-1)$-Leray. From this it follows that zero-avoiding complexes are also near-$(d-1)$-Leray, see \cite[p. 4]{Holmsen2016}. We prove that vector- and element-avoiding complexes are $(d-1)$-Leray in Section \ref{sec:avoiding-complexes}. In fact, we will study the (co)homology groups of these complexes more completely than demanded by the $(d-1)$-Leray property, in order to pave the road for future applications of the method described in this section by enabling more detailed computations.

While \cite[Theorem 1.6]{KalaiMeshulam2005} and \cite[Theorem 1.4]{Holmsen2016} ask for $\mathscr{K}$ to be a matroid, and describe $\mathfrak{U}$ in terms of this structure, we are also interested in cases where such an assumption is not fulfilled, hence we present a reformulation of these two theorems below. Its proof can be extracted from either of the two papers mentioned, after observing that the only property of matroids used there is \cite[Lemma 3.2]{Holmsen2016} (a special case of \cite[Proposition 6.2]{Bjorner1995}). Namely, if $\mathscr{K}$ is a matroid and $U\subseteq V$, then $\mathscr{K}[U]$ is homologically $(\rho_\mathscr{K}[U]-2)$-connected, and Theorem \ref{thm:holmsen-repackaged} specializes to \cite[Theorem 1.6]{KalaiMeshulam2005} and \cite[Theorem 1.4]{Holmsen2016} respectively.
\begin{theorem}\label{thm:holmsen-repackaged}
    Fix a field $\mathbb{F}$ to be the coefficient group of all (co)homology groups mentioned. Let $d\geq 0$, $V$ a nonempty finite set, $\mathscr{C}$ a near-$(d-1)$-Leray simplicial complex with ground set $V$, and $\mathscr{K}$ a homologically $(d-1)$-connected simplicial complex with vertex set $V$ such that $\mathscr{K}[V-\sigma]$ is homologically $(d-2)$-connected for all $\sigma\in\mathscr{C}$. Then $\mathscr{K}\nsubseteq\mathscr{C}$.
    The same holds if $\mathscr{K}$ is not assumed to be homologically $(d-1)$-connected but at the same time $\mathscr{C}$ is $(d-1)$-Leray.
\end{theorem}
Recall that a space $X$ is \emph{homologically $c$-connected (with coefficients in $G$)} if $\tilde{H}_n(X;G)=0$ for all $-1\leq n\leq c$. It will be checked in Section \ref{sec:cc-proofs} that $\mathscr{K}$ satisfies the homological connectivity assumption of Theorem \ref{thm:holmsen-repackaged} with $\mathscr{C}:=\mathscr{C}_{\mathcal{F}}$, independently of which $\mathcal{F}$ we choose from the list above and which colorful Carath\'eodory theorem we take $\mathscr{K}$ from. In particular, we can conclude that there is a $U'\in\mathscr{K}$ such that $U'\notin\mathscr{C}_{\mathcal{F}}$, and this latter property is equivalent to saying that $\mathcal{F}(U')=\mathrm{true}$, proving the colorful Carath\'eodory theorems.

The rest of this section is dedicated to proving Theorem \ref{thm:holmsen-repackaged}. After we review Alexander duality, joins, and some auxiliary constructions in Section \ref{sec:the-method:sc-tools}, we give a proof of Theorem \ref{thm:holmsen-repackaged} in Section \ref{sec:the-method:holmsen}. This is essentially a retelling of \cite{Holmsen2016}, particularly of the proof of Theorem 1.2 therein. Using a lemma of Meshulam regarding the existence of colorful simplices (\cite[Proposition 1.6]{Meshulam2003}, see also \cite[Proposition 3.1]{KalaiMeshulam2005}), the discussion there is rather short. In Section \ref{sec:the-method:covering-scheme} we elaborate on the aforementioned lemma of Meshulam, reducing it to his Homological Nerve Theorem, following the exposition in \cite[Proposition 3.1]{KalaiMeshulam2005}. We choose to do this because the Homological Nerve Theorem can be replaced by a spectral sequence, which should enable partial computations in situations where the assumptions of this theorem fail.

\medskip
\subsection{Alexander duality, joins, and colorful faces}\label{sec:the-method:sc-tools}
\medskip

The core idea in the proof of Theorem \ref{thm:holmsen-repackaged}, coming from \cite{KalaiMeshulam2005} and also used in \cite{Holmsen2016}, is to study the simplicial complex
\[
    \mathscr{C}^**\mathscr{K}
\]
together with a particular colouring of its ground set. In this subsection we review -- following the exposition of \cite[Sections 3.1, 3.2]{Holmsen2016} -- the definition and properties of the operations $\mathscr{X}\mapsto\mathscr{X}^*$ and $(\mathscr{X},\mathscr{Y})\mapsto\mathscr{X}*\mathscr{Y}$, as well as the relevance of the colouring of this complex to Theorem \ref{thm:holmsen-repackaged}.

Recall that an \emph{abstract simplicial complex} on the finite ground set $V$ is a subset $\mathscr{X}\subseteq 2^V$ of the power set $2^V$ of $V$, such that if $\sigma\in \mathscr{X}$ and $\sigma'\subseteq\sigma$ then $\sigma'\in \mathscr{X}$. A \emph{face} of $\mathscr{X}$ is an element $\sigma\in \mathscr{X}$, and we say it is of \emph{dimension} $\dim\sigma:=|\sigma|-1$. Note that $\emptyset$ is the unique $-1$-dimensional face of all simplicial complexes except for $\mathscr{X}=\emptyset$. We will refer to $\mathscr{X}=\emptyset$ as the \emph{void complex}, and to $\mathscr{X}=\{\emptyset\}$ as the \emph{empty complex}. A \emph{vertex} of $\mathscr{X}$ is a $0$-dimensional face, and the \emph{set of vertices} of $\mathscr{X}$ is the union of all vertices -- this might be different from $V$. When we say that $\mathscr{X}$ is \emph{on the vertex set $V$} we mean that $V$ is the set of vertices of $\mathscr{X}$. If $\sigma=\{v\}$ is a vertex of $\mathscr{X}$ then sometimes we will also refer to $v\in V$ as a vertex of $\mathscr{X}$, but this shall cause no confusion. If $U\subseteq V$ then we denote by $\mathscr{X}[U]$ the \emph{induced subcomplex} of $\mathscr{X}$ on the ground set $U$, which is just $\{\sigma\in \mathscr{X}:\sigma\subseteq U\}$.

Any simplicial complex $\mathscr{X}\neq\emptyset$ has a geometric realization where to each $\sigma\in \mathscr{X}$ there is associated a $\dim\sigma$-dimensional geometric simplex $\Delta_\sigma$, and if $\sigma'\subseteq\sigma$ then $\Delta_{\sigma'}$ is glued to the appropriate face of $\Delta_\sigma$. To the void complex $\mathscr{X}=\emptyset$ we do not associate a geometric realization, but by convention say that $\tilde{H}_n(\mathscr{X})$ and $\tilde{H}^n(\mathscr{X})$ vanish for all coefficient groups and $n\in\Z$, even $n=-1$. When talking about (abstract) simplicial complexes we always simultaneously think of the abstract simplicial complex and its geometric realization, and we denote the two using the same letter, although this shall cause no confusion either.

If $\mathscr{X}$ is an abstract simplicial complex on the finite ground set $V$, then its \emph{Alexander dual} is the abstract simplicial complex
\[
    \mathscr{X}^*:=\{U\subseteq V: V-U\notin \mathscr{X}\}
\]
defined on the same ground set. Clearly $(\mathscr{X}^*)^*=\mathscr{X}$, and it is easily verified that $\link_\mathscr{X}(\sigma)=(\mathscr{X}^*[V-\sigma])^*$. The (co)homology groups of $\mathscr{X}$ and $\mathscr{X}^*$ are related through the theorem below, referred to as \emph{combinatorial Alexander duality}, a proof of which can be found in for example \cite{BjornerTancer2009}.
\begin{theorem}\label{thm:alexander-duality}
    Let $\mathscr{X}$ be a simplicial complex on the ground set $V$. Then for every commutative unital ring $R$ and $n\in\Z$ there is an isomorphism
    \[
        \tilde{H}_n(\mathscr{X};R)\cong\tilde{H}^{|V|-3-n}(\mathscr{X}^*;R).
    \]
\end{theorem}

Next we recall that the \emph{join} $\mathscr{X}*\mathscr{Y}$ of two abstract simplicial complexes $\mathscr{X}$ and $\mathscr{Y}$ on ground sets $V$ and $V'$ with $V\cap V'=\emptyset$ is the simplicial complex on the ground set $V\sqcup V'$ defined as
\[
    \mathscr{X}*\mathscr{Y}:=\{\sigma\cup\tau:\sigma\in \mathscr{X},\tau\in \mathscr{Y}\}.
\]
Some convenient sources to learn about joins are \cite{Walker1988} and \cite{Whitehead1956}. In particular, the join operation can be extended to any pair of topological spaces as well \cite[p. 99]{Walker1988}, and the resulting operation satisfies $X*Y\simeq X'*Y'$ whenever $X\simeq X'$ and $Y\simeq Y'$, as well as the property that $X*\operatorname{pt}\simeq\operatorname{pt}$ \cite[p. 56]{Whitehead1956}; here $\simeq$ stands for ``is homotopy equivalent to''. The homology of a join is given by the K\"unneth formula \cite[Equation (2.3)]{Whitehead1956}, which states that
\begin{equation}
    \tilde{H}_n(X*Y;\mathbb{F})\cong\bigoplus_{x+y=n-1}\tilde{H}_x(X;\mathbb{F})\otimes_{\mathbb{F}}\tilde{H}_y(Y;\mathbb{F})
\end{equation}
with coefficients in a field $\mathbb{F}$. In particular, if $X$ is homologically $c_X$-connected and $Y$ is homologically $c_Y$ connected with coefficients in $\mathbb{F}$, then $X*Y$ is homologically $(c_X+c_Y+2)$-connected. Both the K\"unneth formula and the homological connectivity result hold true for all simplicial complexes $\mathscr{X}$ and $\mathscr{Y}$ too, even if one or both of $\mathscr{X}$ and $\mathscr{Y}$ is the void- or empty complex.

Although the definitions of the Alexander dual and joins are now reviewed, we have not yet defined the complex $\mathscr{C}^**\mathscr{K}$, as $\mathscr{C}^*$ and $\mathscr{K}$ have the same ground set, but this can be remedied as follows. If $\mathscr{X}^*$ and $\mathscr{Y}$ are any two abstract simplicial complexes on the common ground set $V$, then take two disjoint copies $V^{(1)},V^{(2)}$ of $V$, and view $\mathscr{X}^*$ as a simplicial complex on the ground set $V^{(1)}$ and $\mathscr{Y}$ as a simplicial complex on the ground set $V^{(2)}$. If $V=\{v_0,\ldots,v_m\}$ and $0\leq j\leq m$ then denote by $v_j^{(1)}$ and $v_j^{(2)}$ the copies of $v_j$ inside $V^{(1)}$ and $V^{(2)}$ respectively; in particular, $V^{(1)}=\{v_0^{(1)},\ldots,v_m^{(1)}\}$ and $V^{(2)}=\{v_0^{(2)},\ldots,v_m^{(2)}\}$. This notation gives rise to a colouring (partition) of the ground set $V^{(1)}\sqcup V^{(2)}$ of $\mathscr{X}^**\mathscr{Y}$ as follows:
\[
    V^{(1)}\sqcup V^{(2)}=\bigsqcup_{0\leq j\leq m}C_j=\bigsqcup_{0\leq j\leq m}\{v_j^{(1)},v_j^{(2)}\}.
\]
Of course, $C_j:=\{v_j^{(1)},v_j^{(2)}\}$. Given a simplicial complex whose ground set is coloured by some colours, say that a face is \emph{colorful} if it has exactly one vertex from each colour. In particular, a face $\sigma$ of $\mathscr{Z}:=\mathscr{X}^**\mathscr{Y}$ is colorful if $|\sigma\cap C_j|= 1$ for each $0\leq j\leq m$. The existence of colorful faces is tightly linked to the conclusion of Theorem \ref{thm:holmsen-repackaged} through the lemma below, which appears as Claim 3.4 in \cite{Holmsen2016}.
\begin{lemma}\label{lem:colorful-faces-vs-containment}
    $\mathscr{Z}$ has a colorful face if and only if $\mathscr{Y}\nsubseteq \mathscr{X}$.
\end{lemma}
In particular, $\mathscr{C}^**\mathscr{K}$ has a colorful face if and only if $\mathscr{K}\nsubseteq\mathscr{C}$, i.e., if and only if the conclusion of Theorem \ref{thm:holmsen-repackaged} holds.

\medskip
\subsection{A short proof of Theorem \ref{thm:holmsen-repackaged} using colorful faces}\label{sec:the-method:holmsen}
\medskip

In this subsection we extract a proof of Theorem \ref{thm:holmsen-repackaged} from \cite{Holmsen2016} (itself building on \cite{KalaiMeshulam2005}), in particular as an application of the following Rainbow Simplex Lemma of Meshulam \cite[Proposition 1.6]{Meshulam2003} (see also \cite[Proposition 3.1]{KalaiMeshulam2005}). The proof of this lemma will be reviewed in Section \ref{sec:the-method:covering-scheme}.
\begin{lemma}[Rainbow Simplex Lemma]\label{lem:meshulam-colorful}
    Let $m\geq0$, $\mathscr{Z}$ be a simplicial complex on the ground set $W=\bigsqcup_{0\leq j\leq m}C_j$ which is coloured by $m+1$ colours such that $C_j$ contains a vertex of $\mathscr{Z}$ for each $0\leq j\leq m$. If for any nonempty subset of colors $J\subseteq\{0,\ldots,m\}$ the complex $\mathscr{Z}[\bigsqcup_{j\in J}C_j]$ is homologically $(|J|-2)$-connected with coefficients in a fixed field $\mathbb{F}$ then $\mathscr{Z}$ contains a colorful face.
\end{lemma}
\cite[Proposition 1.6]{Meshulam2003} and \cite[Proposition 3.1]{KalaiMeshulam2005} only state the lemma above in case of $\mathbb{F}=\mathbb{Q}$, but as our review of the proof in Section \ref{sec:the-method:covering-scheme} also shows, it holds with arbitrary field coefficients.

To prove Theorem \ref{thm:holmsen-repackaged} using this lemma, first observe that in the setting of Theorem \ref{thm:holmsen-repackaged}, $\mathscr{C}$ can not be a simplex with vertex set $V$, as then $\mathscr{K}[V-V]=\mathscr{K}[\emptyset]$ should be homologically $(d-2)$-connected while $\link_\mathscr{C}(V)$ is supposed to have vanishing homology in degrees $d-1$ and greater, but both are the the empty complex and thus have non-trivial homology in degree $-1$, and either $-1\leq d-2$ or $d-1\leq -1$.

If $U\subseteq V$ then denote by $U^{(1)}:=\{v_j^{(1)}:v_j\in U\}$ and $U^{(2)}:=\{v_j^{(2)}:v_j\in U\}$  the corresponding subsets of $V^{(1)}$ and $V^{(2)}$ respectively. By assumption any $v_j^{(2)}\in V^{(2)}$ is a vertex of $\mathscr{K}$, and $\emptyset\in\mathscr{C}^*$ because $V\notin\mathscr{C}$, so the $C_j$ defined at the end of the last section contains a vertex of $\mathscr{Z}:=\mathscr{C}^**\mathscr{K}$. Then the lemma above combined with Lemma \ref{lem:colorful-faces-vs-containment} says that to prove Theorem \ref{thm:holmsen-repackaged} it suffices to check that $\mathscr{Z}[U^{(1)}\sqcup U^{(2)}]$ is homologically $(|U|-2)$-connected with coefficients in $\mathbb{F}$ for all nonempty $U\subseteq V$. Note that 
\[
    \mathscr{Z}[U^{(1)}\sqcup U^{(2)}]=\mathscr{C}^*[U^{(1)}]*\mathscr{K}[U^{(2)}],
\]
so the homological connectivity of $\mathscr{Z}[U^{(1)}\sqcup U^{(2)}]$ can be calculated from the homological connectivities of $\mathscr{C}^*[U^{(1)}]$ and $\mathscr{K}[U^{(2)}]$. Using combinatorial Alexander duality (Theorem \ref{thm:alexander-duality}) we have that
\[
    \tilde{H}_n(\mathscr{C}^*[U^{(1)}];\mathbb{F})\cong\tilde{H}^{|U|-3-n}((\mathscr{C}^*[U^{(1)}])^*;\mathbb{F})=\tilde{H}^{|U|-3-n}(\link_\mathscr{C}(V-U);\mathbb{F}).
\]
If $\emptyset=V-U$ then $\link_\mathscr{C}(V-U)=\link_\mathscr{C}(\emptyset)=\mathscr{C}$, so by near-$(d-1)$-Lerayness of $\mathscr{C}$ the cohomology group on the right vanishes whenever $|U|-3-n\geq d$, or equivalently if $n\leq |U|-3-d$, and in particular $\mathscr{C}^*$ is homologically $(|U|-3-d)$-connected. On the other hand, if $\emptyset\neq V-U\in\mathscr{C}$ then the right hand side is $0$ whenever $|U|-3-n\geq d-1$, meaning that $\mathscr{C}^*[U^{(1)}]$ is homologically $(|U|-2-d)$-connected. Finally, if $V-U\notin\mathscr{C}$ then $\link_\mathscr{C}(V-U)$ is the void complex, as otherwise $\emptyset\in\link_\mathscr{C}(V-U)$, and this would imply $V-U=\emptyset\cup(V-U)\in\mathscr{C}$, a contradiction. Therefore in this case all the reduced (co)homology groups above vanish in all degrees, and $\mathscr{C}^*[U^{(1)}]$ is homologically ``$\infty$-connected'' (acyclic).

Now let us consider $\mathscr{K}[U^{(2)}]$. By assumption of the theorem this is homologically $(d-1)$-connected if $\emptyset=V-U$, and homologically $(d-2)$-connected if $\emptyset\neq V-U\in\mathscr{C}$. In the former case we calculated $\mathscr{C}^*[U^{(1)}]$ to be homologically $(|U|-3-d)$-connected, so then $\mathscr{Z}[U^{(1)}\sqcup U^{(2)}]=\mathscr{C}^*[U^{(1)}]*\mathscr{K}[U^{(2)}]$ is homologically $(|U|-3-d)+(d-1)+2=(|U|-2)$-connected, while in the latter case $\mathscr{C}^*[U^{(1)}]$ is homologically $(|U|-2-d)$-connected so $\mathscr{Z}[U^{(1)}\sqcup U^{(2)}]$ is homologically $(|U|-2-d)+(d-2)+2=(|U|-2)$-connected as well. The only case left is when $\emptyset\neq V-U\not\in\mathscr{C}$, but then all reduced homology groups of $\mathscr{C}^*[U^{(1)}]$ vanish, so the same holds for $\mathscr{Z}[U^{(1)}\sqcup U^{(2)}]=\mathscr{C}^*[U^{(1)}]*\mathscr{K}[U^{(2)}]$.

Thus we computed that $\mathscr{Z}[U^{(1)}\sqcup U^{(2)}]$ is homologically $(|U|-2)$-connected for any nonempty $U\subseteq V$, so by the Rainbow Simplex Lemma (Lemma \ref{lem:meshulam-colorful}) $\mathscr{Z}$ must have a colorful face, and by Lemma \ref{lem:colorful-faces-vs-containment} we have $\mathscr{K}\nsubseteq\mathscr{C}$, concluding the proof of Theorem \ref{thm:holmsen-repackaged}.

If we do not assume $\mathscr{K}$ to be homologically $(d-1)$-connected but $\mathscr{C}$ is $(d-1)$-Leray, then the proof above only changes in the case $\emptyset=V-U\in\mathscr{C}$. Then $\tilde{H}^{|U|-3-n}(\link_\mathscr{C}(\emptyset);\mathbb{F})$ vanishes for $|U|-3-n\geq d-1$ by $(d-1)$-Lerayness of $\mathscr{C}$, meaning that $\mathscr{C}^*$ is homologically $(|U|-2-d)$-connected. On the other hand, $\mathscr{K}=\mathscr{K}[V-\emptyset]$ is only $(d-2)$-connected now, but this is still enough to compute that $\mathscr{Z}[U^{(1)}\sqcup U^{(2)}]=\mathscr{Z}=\mathscr{C}^**\mathscr{K}$ is homologically $(|U|-2-d)+(d-2)+2=(|U|-2)$-connected if $U=V$. But the fact that $\mathscr{Z}[U^{(1)}\sqcup U^{(2)}]$ is homologically $(|U|-2)$-connected when $\emptyset\neq V-U\in\mathscr{C}$ or $V-U\notin\mathscr{C}$ remains true with the same proof, so by the Rainbow Simplex Lemma (\ref{lem:meshulam-colorful}) the complex $\mathscr{Z}$ has a colorful face, and by Lemma \ref{lem:colorful-faces-vs-containment} we have $\mathscr{K}\nsubseteq\mathscr{C}$ again.

\medskip
\subsection{The covering scheme}\label{sec:the-method:covering-scheme}
\medskip

Here we recall and comment on the proof of the Rainbow Simplex Lemma (Lemma \ref{lem:meshulam-colorful}) given in \cite[Proposition 3.1]{KalaiMeshulam2005}. First, note the following observation about the connection between colorful simplices and covers:
\begin{lemma}\label{lem:colorful-simplices-vs-covers}
    Let $m\geq 0$, $\mathscr{Z}$ be a simplicial complex on the ground set $W=\bigsqcup_{0\leq j\leq m}C_j$ which is coloured by $m+1$ colours. It has no colorful face if and only if the family of induced subcomplexes
    \[
        (\mathscr{Z}[W-C_j])_{0\leq j\leq m}
    \]
    is a cover of $\mathscr{Z}$, i.e., if $\mathscr{Z}=\bigcup_{0\leq j\leq m}\mathscr{Z}[W-C_j]$.
\end{lemma}
\begin{proof}
    Assume $\mathscr{Z}$ has no colorful face, and let $\sigma\in \mathscr{Z}$ be arbitrary. Then as no face of $\sigma$ is colorful, there is a $0\leq j\leq m$ such that $\sigma\cap C_j=\emptyset$, in particular $\sigma\in \mathscr{Z}[W-C_j]$. This was true for any $\sigma\in \mathscr{Z}$, so $\mathscr{Z}=\bigcup_{0\leq j\leq m}\mathscr{Z}[W-C_j]$.

    For the other direction assume that $\mathscr{Z}=\bigcup_{0\leq j\leq m}\mathscr{Z}[W-C_j]$, and take any $\sigma\in \mathscr{Z}$. Then $\sigma\in \mathscr{Z}[W-C_j]$ for some $0\leq j\leq m$, and in particular $\sigma\cap C_j=\emptyset$, making $\sigma$ not colorful. This was true for any $\sigma\in \mathscr{Z}$, so $\mathscr{Z}$ has no colorful face.
\end{proof}

The Rainbow Simplex Lemma (Lemma \ref{lem:meshulam-colorful}) is proved by contradiction, so we may assume that $\mathscr{Z}$ has no colorful face. We only apply this lemma to $\mathscr{Z}:=\mathscr{X}^**\mathscr{Y}$ with the colouring described at the end of Section \ref{sec:the-method:sc-tools}, so Lemmas \ref{lem:colorful-faces-vs-containment} and \ref{lem:colorful-simplices-vs-covers} can be combined into the following Lemma, which we dub the \emph{covering scheme}.
\begin{lemma}[Covering scheme]
    Assume $m\geq 0$ and $\mathscr{Z}=\mathscr{X}^**\mathscr{Y}$ for some abstract simplicial complexes $\mathscr{X}$ and $\mathscr{Y}$ defined on the same ground set $V=\{v_0,\ldots,v_m\}$. If $\mathscr{Y}\subseteq \mathscr{X}$ then the family of induced subcomplexes
    \[
        (\mathscr{Z}[(V^{(1)}\sqcup V^{(2)})-C_j])_{0\leq j\leq m}
    \]
    is a cover of $\mathscr{Z}$.
\end{lemma}
\begin{proof}
    Lemma \ref{lem:colorful-faces-vs-containment} says that $\mathscr{Y}\subseteq \mathscr{X}$ if and only if $\mathscr{Z}$ has no colorful face, while Lemma \ref{lem:colorful-simplices-vs-covers} guarantees that the given collection of subcomplexes is a cover if $\mathscr{Z}$ has no colorful face.
\end{proof}
As was apparent in Section \ref{sec:the-method:holmsen}, the question of whether $\mathscr{Z}$ has colorful simplices was simply used to connect the containment $\mathscr{K}\stackrel{?}{\subseteq}\mathscr{C}$ to the homological connectivity of the induced subcomplexes $\mathscr{Z}[U^{(1)}\sqcup U^{(2)}]$. Indeed, Lemma \ref{lem:colorful-faces-vs-containment} said that $\mathscr{K}\nsubseteq\mathscr{C}$ if and only if $\mathscr{Z}$ has a colorful face, which by the Rainbow Simplex Lemma (Lemma \ref{lem:meshulam-colorful}) is guaranteed to happen under certain homological conditions on subcomplexes of the form $\mathscr{Z}[U^{(1)}\sqcup U^{(2)}]$.
We present the covering scheme to highlight this fact, and to connect the two ends of the argumentation directly. Another reason for putting the emphasis on the cover $(\mathscr{Z}[(V^{(1)}\sqcup V^{(2)})-C_j])_{0\leq j\leq m}$ instead of the existence of a colorful face $\sigma\in \mathscr{Z}$  is that the main tool in the proof of the Rainbow Simplex Lemma (Lemma \ref{lem:meshulam-colorful}) is the Homological Nerve Theorem of Meshulam (\cite[Theorem 2.1]{Meshulam2001}, see below), which takes a cover as an input. Moreover, this theorem has multiple generalizations, like the Leray spectral sequence it is derived from, that could be substituted in when one attempts to extend the results in this paper, but only once the existence of the cover is established.
\begin{theorem}[The Homological Nerve Theorem]\label{thm:meshulam-nerve}
    Let $\mathscr{Z}$ be a finite simplicial complex and let $\mathfrak{U}:=(\mathscr{Z}_i:i\in I)$ be a finite cover by non-void subcomplexes. Fix a field $\mathbb{F}$ and an integer $k$, and assume that for any $\emptyset\neq J\subseteq I$ the space $\bigcap_{j\in J}\mathscr{Z}_j$ is empty or it is homologically $(k-|J|+1)$-connected with coefficients in $\mathbb{F}$. Then
    \begin{enumerate}[label=\normalfont{(\arabic*)}]
        \item\label{thm:meshulam-nerve:low-degrees} $\tilde{H}_n(\mathscr{Z};\mathbb{F})\cong\tilde{H}_n(\mathscr{N}_\mathfrak{U};\mathbb{F})$ for all $0\leq n\leq k$, and
        \item\label{thm:meshulam-nerve-k+1} if $\tilde{H}_{k+1}(\mathscr{N}_\mathfrak{U};\mathbb{F})\neq 0$ then $\tilde{H}_{k+1}(\mathscr{Z};\mathbb{F})\neq 0$.
    \end{enumerate}
\end{theorem}
Here $\mathscr{N}_\mathfrak{U}$ denotes the \emph{nerve complex} of the cover $\mathfrak{U}$, which is a simplicial complex on the ground set $I$ defined as
\[
    \mathscr{N}_\mathfrak{U}:=\{J\subseteq I: J=\emptyset\text{~or~}\bigcap_{j\in J}\mathscr{Z}_j\text{~is not an empty space}\}.
\]
As promised, the proof of the Rainbow Simplex Lemma (Lemma \ref{lem:meshulam-colorful}) is done by contradiction, so we may assume that $\mathscr{Z}$ contains no colorful face, and so by Lemma \ref{lem:colorful-simplices-vs-covers} that $(\mathscr{Z}[W-C_j])_{0\leq j\leq m}$ is a cover of $\mathscr{Z}$. To apply the Homological Nerve Theorem we need to estimate the homological connectivity of $\bigcap_{j\in J}\mathscr{Z}[W-C_j]=\mathscr{Z}[W-\bigsqcup_{j\in J}C_j]=\mathscr{Z}[\bigsqcup_{i\in I-J}C_i]$, but this space is homologically $(|I|-|J|-2)$-connected by assumption, or empty when $J=I$. (Remember that in Section \ref{sec:the-method:holmsen} proving this assumption amounted to computing the homological connectivity of $\mathscr{Z}[U^{(1)}\sqcup U^{(2)}]=\mathscr{C}^*[U^{(1)}]*\mathscr{K}[U^{(2)}]$). Consequently, the Homological Nerve Theorem can be applied with $k:=|I|-3$. The assumption of the Rainbow Simplex Lemma (Lemma \ref{lem:meshulam-colorful}) for $J:=I$ says that $\mathscr{Z}$ is $(|I|-2)$-connected, so in particular $\tilde{H}_{|I|-2}(\mathscr{Z};\mathbb{F})=0$, and by point \ref{thm:meshulam-nerve-k+1} of Theorem \ref{thm:meshulam-nerve} the homology group $\tilde{H}_{|I|-2}(\mathscr{N}_\mathfrak{U};\mathbb{F})$ also vanishes. The complex $\mathscr{N}_\mathfrak{U}$ is easy to determine for the cover $\mathfrak{U}$ given by Lemma \ref{lem:colorful-simplices-vs-covers}:
\begin{lemma}
    Assume we are in the setup of Lemma \ref{lem:colorful-simplices-vs-covers} and that $C_j$ contains a vertex of $\mathscr{Z}$ for all $0\leq j\leq m$. Then for the cover $\mathfrak{U}$ of the Lemma just mentioned we have that $\mathscr{N}_\mathfrak{U}=\partial\Delta_{\{0,\ldots,m\}}$, i.e., it is the boundary of an $m$-dimensional simplex, and as such homeomorphic to an $(m-1)$-dimensional sphere.
\end{lemma}
\begin{proof}
    $\bigcap_{j\in J}\mathscr{Z}[W-C_j]=\mathscr{Z}[W-\bigsqcup_{j\in J}C_j]=\mathscr{Z}[\bigsqcup_{i\in I-J}C_j]$ is an empty space if $I-J=\emptyset$ and it contains a vertex and is hence nonempty whenever $I-J\neq\emptyset$. In particular, $\mathscr{N}_\mathfrak{U}$ is the abstract simplicial complex on the ground set $I$ where every $J\subseteq I$ is a face except for $J=I$, and this is by definition $\partial\Delta_{\{0,\ldots,m\}}$.
\end{proof}
In our case $I=\{0,\ldots,m\}$, so $|I|=m+1$, but then $0=\tilde{H}_{|I|-2}(\mathscr{N}_\mathfrak{U};\mathbb{F})\cong\tilde{H}_{m-1}(S^{m-1};\mathbb{F})\cong\mathbb{F}$ gives a contradiction, concluding the proof of the Rainbow Simplex Lemma (Lemma \ref{lem:meshulam-colorful}).

\bigskip
\section{Preliminaries regarding oriented matroids}\label{sec:prelim-oms}
\bigskip

In this section we give an introduction to the theory of oriented matroids, presenting the relevant parts of \cite{OrientedMatroids}.

Given a function $A\colon V\to\R^d$, to each linear functional $\varphi\in(\R^d)^*$ associate the function $\mathfrak{s}_A(\varphi)\colon V\to\{-,0,+\}$, $v\mapsto \sgn(\varphi(A(v)))$; think of it as a ``sign vector'' whose coordinates are indexed by $V$. It encodes for each $v\in V$ whether $\varphi$ is on the (in case of $A(v)=0$ possibly degenerate) hyperplane $H_v:=\{\psi\in(\R^d)^*: \psi(A(v))=0\}$, and if not, which side it lays on. It turns out that given only the set $\mathscr{L}:=\mathfrak{s}_A((\R^d)^*)$, but not the original function $A$ itself, it is possible to decide for any $U,U'\subseteq V$, $v\in V$, and $k\in\Z$ which of the following claims hold true:
\begin{center}
    \begin{tabular}{ccc}
        $0\in\conv(A(U))$ & $0\in\relint(\conv(A(U)))$ & $\cone(A(U))\cap\cone(A(U'))=\{0\}$ \\
        $A(v)\in\cone(A(U))$ & $A(v)\in\relint(\cone(A(U)))$ & $\dim(\operatorname{span}(A(U)))=k$
    \end{tabular}
\end{center}
In particular, both the assumptions and the conclusions of the colorful Carath\'eodory theorems \ref{thm:caratheodory:original}, \ref{thm:caratheodory:cone}, \ref{thm:caratheodory:fixed}, \ref{thm:caratheodory:pair}, \ref{thm:caratheodory:matroid}, \ref{thm:caratheodory:constrained}, and \ref{thm:caratheodory:3-joined} can be stated in terms of the set $\mathscr{L}$. Not only that, but $\mathscr{L}$, viewed as an abstract collection of sign vectors, satisfies a few combinatorial properties which are sufficient to prove the conclusions from the assumptions, without referring to affine geometry at all. A set of sign vectors having these combinatorial properties is called \emph{(the set of covectors of) an oriented matroid}, and so both the statements and the proofs of the colorful Carath\'eodory theorems can be phrased in terms of the oriented matroid $\mathscr{L}$ associated to the function $A$. We choose to present the aforementioned proofs in the language of oriented matroids because this way we also obtain generalizations of the theorems, and as it makes our argumentation both clearer and more rigorous.

While the combinatorics of oriented matroids will be helpful, the geometric picture of the arrangement of (possibly degenerate) hyperplanes $(H_v)_{v\in V}$ in $(\R^d)^*$ will also be of importance. Instead of caring about the whole of $(\R^d)^*$ however, it suffices to look at the unit sphere $S((\R^d)^*)$ (with respect to some fixed inner product on $(\R^d)^*$) because of the fact that $\mathscr{L}=\mathfrak{s}_A(S((\R^d)^*))\cup\{0\}$, where by $0$ we mean the sign vector which is $0$ for all $v\in V$. Consequently, to understand $\mathscr{L}$ it suffices to restrict our attention to how $S((\R^d)^*)$ is partitioned into cells by the subspheres $(H_v\cap S((\R^d)^*))_{v\in V}$. Indeed, there is a one-to-one correspondence given by $\mathfrak{s}_A$ between the cells of this regular cell decomposition of $S((\R^d)^*)$ and the elements of $\mathscr{L}-\{0\}$, at least if $\operatorname{span}(A(V))=\R^d$. A similar, but not identical, statement holds if $\operatorname{span}(A(V))\neq\R^d$. By the Topological Representation Theorem of Folkman and Lawrence, reviewed in Theorem \ref{thm:top-repr}, there is a similar geometric interpretation of each oriented matroid where the cells of a decomposition of a sphere are indexed by the covectors of the oriented matroid.

In Section \ref{sec:prelim-oms:combinatorial} we review the definition of oriented matroids as well as some of their combinatorial properties, and explain what parts of the function $A$ they abstract. After that, in Section \ref{sec:prelim-oms:pseudospheres} we recall the Topological Representation Theorem and establish some notation to deal with signed arrangements of pseudospheres -- the geometric pictures associated to oriented matroids.

\medskip
\subsection{The combinatorial approach}\label{sec:prelim-oms:combinatorial}
\medskip

As the definition of oriented matroids presented later in this section concerns itself with collections of sign vectors, we will benefit from fixing some notation both for both sign vectors and the signs involved. On the set $\{-,0,+\}$ we consider the partial order where $0<-$, $0<+$, and $-$ and $+$ are incomparable. If $\mathfrak{s}\in\{-,0,+\}$ write
\[
    -\mathfrak{s}:=\begin{cases}
        +, & \text{~if~} \mathfrak{s}=-,\\
        0, & \text{~if~} \mathfrak{s}=0,\\
        -, & \text{~if~} \mathfrak{s}=+.
    \end{cases}
\]
Furthermore, if $\mathfrak{s},\mathfrak{s}'\in\{-,0,+\}$ write
\[
    \mathfrak{s}\circ\mathfrak{s}':=\begin{cases}
        \mathfrak{s}, & \text{~if~} \mathfrak{s}\neq 0,\\
        \mathfrak{s}', & \text{~otherwise.}
    \end{cases}
\]

Throughout $V$ will be a fixed finite ground set, and by \emph{sign vector} we will mean any function $\Phi\colon V\to\{-,0,+\}$. Extend the above notions regarding signs to sign vectors coordinatewise, i.e., for sign vectors $\Phi$ and $\Psi$ say that $\Phi\leq\Psi$ if $\Phi(v)\leq\Psi(v)$ for all $v\in V$, and define the sign vectors $-\Phi$ and $\Phi\circ\Psi$ with the formulas $(-\Phi)(v):=-\Phi(v)$ and $(\Phi\circ\Psi)(v):=\Phi(v)\circ\Psi(v)$ for all $v\in V$. $\Phi\circ\Psi$ is called the \emph{composition} of $\Phi$ and $\Psi$, and the composition of a collection of sign vectors is said to be \emph{conformal} if it is independent of the order of the sign vectors. Moreover, write $\Phi^+:=\Phi^{-1}(\{+\})$, $\Phi^-:=\Phi^{-1}(\{-\})$, and $\Phi^0:=\Phi^{-1}(\{0\})$, $S(\Phi,\Psi):=(\Phi^+\cap\Psi^-)\cup(\Phi^-\cap\Psi^+)$, and call $\underline{\Phi}:=\Phi^+\cup\Phi^-$ the \emph{support} of $\Phi$. A sign vector $\Phi$ is \emph{(contained) in} $U\subseteq V$ if $\underline{\Phi}\subseteq U$. Let $0$ be the sign vector with support the empty set, and say that a sign vector $\Phi$ is \emph{positive} if $\Phi\neq 0$ and $\Phi^-=\emptyset$. Finally, say that $\Phi$ and $\Psi$ are \emph{orthogonal}, denoted $\Phi\perp\Psi$, if whenever $\Phi(v)=\Psi(v)\neq 0$ for some $v\in V$ then $\Phi(w)=-\Psi(w)\neq 0$ for some $w\in V$, and conversely.

Now we are ready to define oriented matroids. They have many definitions which give rise to equivalent objects, but for the purpose of this paper we restrict our attention to the following one, see \cite[Propositions 3.7.5, 3.7.9]{OrientedMatroids}. The many different definitions motivate denoting an oriented matroid and its set of covectors in different ways, despite the former being defined by the latter.
\begin{definition}\label{def:om}
    An \emph{oriented matroid} $\mathscr{O}$ on the set of elements $V$ is a set $\mathscr{L}(\mathscr{O})$ of sign vectors satisfying the following axioms:
    \begin{enumerate}[label=\normalfont{(V\arabic*)}]
        \setcounter{enumi}{-1}
        \item $0\in\mathscr{L}(\mathscr{O})$,
        \item $-\Phi\in\mathscr{L}(\mathscr{O})$ if and only if $\Phi\in\mathscr{L}(\mathscr{O})$,
        \item for all $\Phi,\Psi\in\mathscr{L}(\mathscr{O})$ we have $\Phi\circ\Psi\in\mathscr{L}(\mathscr{O})$,
    \end{enumerate}
    and finally:
    \begin{enumerate}[resume,label=\normalfont{(V\arabic*')}]
        \item\label{def:om:elim} For all $\Pi,\Xi\in\mathscr{L}(\mathscr{O})$ and $v\in\Pi^+\cap\Xi^-$ there is a $\Theta\in\mathscr{L}(\mathscr{O})$ such that $\Theta(v)=0$ and if $w\in V-S(\Pi,\Xi)$ then $\Theta(w)=\Pi(w)\circ\Xi(w)$.
    \end{enumerate}
    The elements of $\mathscr{L}(\mathscr{O})$ are referred to as the \emph{covectors} of the oriented matroid $\mathscr{O}$.
    \end{definition}
Then indeed the set $\mathscr{L}:=\mathfrak{s}_A((\R^d)^*)$ associated to some function $A\colon V\to\R^d$ is the set of covectors of some oriented matroid $\mathscr{O}$ \cite[Section 2.1]{OrientedMatroidsToday}.

Given an oriented matroid $\mathscr{O}$ on set of elements $V$, say that a sign vector $\Phi$ is a \emph{vector} of $\mathscr{O}$ if $\Phi\perp\Psi$ for all $\Psi\in\mathscr{L}(\mathscr{O})$, and write $\mathscr{V}(\mathscr{O})$ for the set of vectors of $\mathscr{O}$. The minimal non-zero elements of $\mathscr{V}(\mathscr{O})$ are called \emph{circuits}, their set is denoted by $\mathscr{C}(\mathscr{O})$, and any vector of $\mathscr{O}$ is equal to the conformal composition of all circuits smaller than it \cite[Proposition 3.7.2]{OrientedMatroids}. If $\mathscr{L}(\mathscr{O})=\mathfrak{s}_A((\R^d)^*)$ for some function $A\colon V\to\R^d$, then $\Phi\in\mathscr{V}(\mathscr{O})$ if and only if there are real numbers $(\lambda_v)_{v\in V}$ such that $\sum_{v\in V}\lambda_vA(v)=0$ and $\Phi(v)=\sgn\lambda_v$ for all $v\in V$ \cite[Section 2.1]{OrientedMatroidsToday}. $\mathscr{V}(\mathscr{O})$ is the set of covectors of another oriented matroid denoted $\mathscr{O}^*$ \cite[Proposition 3.7.12]{OrientedMatroids} (in light of \cite[Proposition 3.4.1]{OrientedMatroids}), and $(\mathscr{O}^*)^*=\mathscr{O}$ \cite[Proposition 3.4.1]{OrientedMatroids}.

Using the above characterization of vectors and covectors of the $\mathscr{O}$ defined by $\mathscr{L}(\mathscr{O}):=\mathfrak{s}_A((\R^d)^*)$, certain statements about convex hulls and cones can be translated into the language of oriented matroids, as summarized by the lemma below. Here, $\conv(R)$ stands for the convex hull $\{\sum_{r\in R}\lambda_rr:\lambda_r\geq 0\text{~for all~}r\in R,\sum_{r\in R}\lambda_r=1\}$ and $\cone(R)$ is the convex cone $\{\sum_{r\in R}\lambda_rr:\lambda_r\geq 0\text{~for all~}r\in R\}$ of the finite set $R\subseteq\R^d$, while $\conv(U):=U\sqcup\{v\in V-U:\Phi^+\subseteq U\text{~and~}\Phi^-=\{v\}\text{~for some~}\Phi\in\mathscr{V}(\mathscr{O})\}$ for $U\subseteq V$ (see \cite[Exercise 3.9]{OrientedMatroids}). Note that in particular $\cone(\emptyset)=\{0\}$, and $\mathscr{V}(\mathscr{O})$ can be replaced by $\mathscr{C}(\mathscr{O})$ in the definition of $\conv(U)$ because every vector is the conformal composition of all circuits smaller than it.
\begin{lemma}\label{lem:0-in-conv-interpreted-in-oms}
    If $\mathscr{L}(\mathscr{O})=\mathfrak{s}_A((\R^d)^*)$ for some function $A\colon V\to\R^d$ and $U\subseteq V$ then:
    \begin{enumerate}[label=\normalfont{(\arabic*)}]
        \item\label{lem:0-in-conv-interpreted-in-oms:full} $0\in\conv(A(U))$ if and only if $U$ contains a positive vector of $\mathscr{O}$.
        \item\label{lem:0-in-conv-interpreted-in-oms:relint} $0\in\relint(\conv(A(U)))$ if and only if there is a positive vector of $\mathscr{O}$ with support $U$.
        \item\label{lem:0-in-conv-interpreted-in-oms:cone} $A(v)\in\cone(A(U))$ if and only if $v\in\conv(U)$.
        \item\label{lem:0-in-cov-interpreted-in-oms:relintcone} Assuming $v\in V-U$, $A(v)\in\relint(\cone(A(U)))$ if and only if there exists a $\Psi\in\mathscr{V}(\mathscr{O})$ with $\Psi^+=U$ and $\Psi^-=\{v\}$. 
    \end{enumerate}
\end{lemma}
\begin{proof}
    $0\in\conv(A(U))$ if and only if there are non-negative real numbers $(\lambda_u)_{u\in U}$ such that $\lambda_u\neq 0$ for at least one $u\in U$, and $\sum_{u\in U}\lambda_uA(u)=0$. In other words, the condition $\sum_{u\in U}\lambda_u=1$ can be weakened to $\exists u\in U\colon \lambda_u\neq 0$. This equivalent condition can however be rephrased as saying that there is a linear dependency among $(A(u))_{u\in U}$ where all coefficients are non-negative and not all of them are $0$. The sign vectors of linear dependencies are exactly the vectors of $\mathscr{O}$, so such sign vectors correspond exactly to positive vectors contained in $U$. This concludes the proof of point \ref{lem:0-in-conv-interpreted-in-oms:full}.

    $0\in\relint(\conv(A(U)))$ if and only if there is a linear dependency between $(A(u))_{u\in U}$ where all coefficients are positive \cite[p. 11, Theorem 6]{McMullenShephard}, so as sign vectors of linear dependencies correspond exactly to vectors of $\mathscr{O}$, this latter condition holds if and only if there is a positive vector of $\mathscr{O}$ with support $U$, proving point \ref{lem:0-in-conv-interpreted-in-oms:relint}.

    If $A(v)=0$ then $(-1).A(v)=0$ so the sign vector $\Phi$ with $\Phi^+=\emptyset$, $\Phi^-=\{v\}$ is a vector of $\mathscr{O}$, and hence $v\in\conv(U)$ for all $U\subseteq V$, similarly to how $A(v)\in\cone(A(U))$ for any $U\subseteq V$.
    Now assume $A(v)\neq 0$. If $v\in U$ then point \ref{lem:0-in-conv-interpreted-in-oms:cone} of the lemma is trivial, so assume otherwise. Then $A(v)\in\cone(A(U))$ if and only if $\cone(A(v))\cap\cone(A(U))\neq \{0\}$, and $v\in\conv(U)$ if and only if there is a circuit $\Phi$ of $\mathscr{O}$ with $\emptyset\neq\Phi^+\subseteq U$ and $\emptyset\neq\Phi^-\subseteq\{v\}$. Consequently, point \ref{lem:0-in-conv-interpreted-in-oms:cone} follows from Lemma \ref{lem:intersection-of-cones-in-oms} below.

    Let us move on to point \ref{lem:0-in-cov-interpreted-in-oms:relintcone}. If $A(v)\in\relint(\cone(A(U)))\subseteq\cone(A(U))$ then by point \ref{lem:0-in-conv-interpreted-in-oms:cone} there is a vector $\Psi_0$ of $\mathscr{O}$ with $\Psi_0^+\subseteq U$ and $\Psi_0^-=\{v\}$. Because $A(v)\in\relint(\cone(A(U)))$, necessarily $A(v)-\varepsilon\sum_{u\in U-\Psi_0^+}A(u)$ is also in $\cone(A(U))$ for $\varepsilon>0$ sufficiently small, and so it must be of the form $\sum_{u\in U}\lambda_uA(u)$ with $\lambda_u\geq 0$. But then $0=-A(v)+\sum_{u\in U}\lambda_uA(u)+\sum_{u\in U-\Psi_0^+}\varepsilon A(u)$ is a linear dependence to which a vector $\Psi_1$ is associated with $\Psi_1^-=\{v\}$ and $\Psi_1^+\supseteq U-\Psi_0^+$. In particular, then $\Psi:=\Psi_0\circ\Psi_1\in\mathscr{V}(\mathscr{O})$ satisfies the claim of point \ref{lem:0-in-cov-interpreted-in-oms:relintcone}. On the other hand, if there is a $\Psi\in\mathscr{V}(\mathscr{O})$ with $\Psi^+=U$ and $\Psi^-=\{v\}$, then there is a corresponding linear dependence $0=\lambda_vA(v)+\sum_{u\in U}\lambda_uA(u)$ with $\sgn(\lambda_w)=\Psi(w)$ for $w\in U\cup\{v\}$, so $A(v)=\sum_{u\in U}-\lambda_v^{-1}\lambda_uA(u)$. Now let $0<\varepsilon<\min_{u\in U}-\lambda_v^{-1}\lambda_u$, and note that 
    $
        N:=\{A(v)+\sum_{u\in U}\mu_uA(u)=\sum_{u\in U}(\mu_u-\lambda_v^{-1}\lambda_u)A(u):\mu_u\in(-\varepsilon,\varepsilon)\text{~for all~}u\in U\}
    $
    is a relatively open neighborhood of $A(v)$ contained in $\cone(A(U))$, showing that $A(v)\in\relint(\cone(A(U)))$.
\end{proof}
As every vector is a conformal composition of all circuits smaller than it, $U$ contains a positive circuit of $\mathscr{O}$ if and only if it contains a positive vector of $\mathscr{O}$.
\begin{lemma}\label{lem:intersection-of-cones-in-oms}
    If $\mathscr{L}(\mathscr{O})=\mathfrak{s}_A((\R^d)^*)$ for some function $A\colon V\to\R^d$, and $U,U'\subseteq V$ with $U\cap U'=\emptyset$, then $\cone(A(U))\cap\cone(A(U'))\neq\{0\}$ if and only if there is a $\Phi\in\mathscr{C}(\mathscr{O})$ satisfying $\emptyset\neq\Phi^+\subseteq U$ and $\emptyset\neq\Phi^-\subseteq U'$.
\end{lemma}
\begin{proof}
    $\cone(A(U))\cap\cone(A(U'))\neq\{0\}$ if and only if there are non-negative real numbers $(\lambda_u)_{u\in U}$ and $(\mu_{u'})_{u'\in U'}$ such that $x:=\sum_{u\in U}\lambda_uA(u)=\sum_{u'\in U'}\mu_{u'}A(u')\neq 0$. Fix $x$, and choose $(\lambda_u)_{u\in U}$ and $(\mu_{u'})_{u'\in U'}$ so that as few of them are non-zero as possible. Then to the equality defining $x$ there is a corresponding linear dependence $\sum_{u\in U}\lambda_uA(u)-\sum_{u'\in U'}\mu_{u'}A(u')=0$, to which there is an associated vector $\Psi$ of $\mathscr{O}$. Take any circuit $\Phi\leq\Psi$, and note that neither $\Phi$ nor $-\Phi$ can be positive, because then a corresponding linear dependence could be used to falsify the minimality assumption on $(\lambda_u)_{u\in U}$ and $(\mu_{u'})_{u'\in U'}$. In particular, $\Phi$ satisfies the demands of the Lemma.

    On the other hand, given any $\Phi$ like in the lemma, there is a corresponding minimal linear dependence $\sum_{u\in U}\lambda_uA(u)-\sum_{u'\in U'}\mu_{u'}A(u')=0$, where by minimality $x:=\sum_{u\in U}\lambda_uA(u)=\sum_{u'\in U'}\mu_{u'}A(u')\neq 0$ is in $\cone(A(U)) \cap\cone(A(U'))$.
\end{proof}
The conditions on the vectors in Lemma \ref{lem:0-in-conv-interpreted-in-oms} can be replaced by conditions on covectors:
\begin{corollary}\label{crl:0-in-conv-interpreted-in-oms}
    If $\mathscr{L}(\mathscr{O})=\mathfrak{s}_A((\R^d)^*)$ for some function $A\colon V\to\R^d$, $U\subseteq V$, and $v\in V$, then:
    \begin{enumerate}[label=\normalfont{(\arabic*)}]
        \item\label{crl:0-in-conv-interpreted-in-oms:full} $0\in\conv(A(U))$ if and only if $\nexists\Phi\in\mathscr{L}(\mathscr{O})\colon U\subseteq\Phi^+$,
        \item\label{crl:0-in-conv-interpreted-in-oms:relint} $0\in\relint(\conv(A(U)))$ if and only if $\nexists\Phi\in\mathscr{L}(\mathscr{O})\colon\Phi|_U$ is positive,
        \item\label{crl:0-in-conv-interpreted-in-oms:cone} $A(v)\in\cone(A(U))$ if and only if $\nexists\Phi\in\mathscr{L}(\mathscr{O})$ with $v\in\Phi^-$ and $U\cap\Phi^-=\emptyset$, and
        \item\label{crl:0-in-conv-interpreted-in-oms:relintcone} assuming $v\in V-U$, $A(v)\in\relint(\cone(A(U)))$ if and only if $\nexists\Phi\in\mathscr{L}(\mathscr{O})$ with $\Phi^-\cap U=\emptyset=\Phi^+\cap\{v\}$ and $\underline{\Phi}\cap (U\cup\{v\})\neq\emptyset$.
    \end{enumerate}
\end{corollary}
The proof of this corollary relies on an analogue of Farkas' lemma \cite[Proposition 3.4.8]{OrientedMatroids} (in light of \cite[Lemma 4.1.8]{OrientedMatroids}):
\begin{lemma}\label{lem:acyclic-oms}
    If $\mathscr{O}$ is an oriented matroid on the vertex set $V$ and $U\subseteq V$, then $U$ contains no positive vector if and only if there is a covector $\Phi$ of $\mathscr{O}$ with $U\subseteq\Phi^+$.
\end{lemma}
In fact, we need a slight generalization of this lemma, presented below, for the proof of which we use that there is an oriented matroid $\mathscr{O}|_U$ with set of elements $U$, for which $\mathscr{V}(\mathscr{O}|_U)=\{\Psi|_U:\Psi\in\mathscr{V}(\mathscr{O}),\underline{\Psi}\subseteq U\}$ \cite[Proposition 3.7.11]{OrientedMatroids} and $\mathscr{L}(\mathscr{O}|_U)=\{\Phi|_U:\Phi\in\mathscr{L}(\mathscr{O})\}$ \cite[Lemma 4.1.8]{OrientedMatroids}.
\begin{lemma}\label{lem:general-vector-covector-duality}
    Let $\mathscr{O}$ be an oriented matroid on set of elements $V$, $\Xi$ any sign vector, and $U\subseteq\underline{\Xi}$. Then exactly one of the following two statements hold:
    \begin{enumerate}[label=\normalfont{(\roman*)}]
        \item\label{lem:general-vector-covector-duality:vector} $\exists\Psi\in\mathscr{V}(\mathscr{O}):\Psi\leq\Xi$ and $\underline{\Psi}\cap U\neq\emptyset$.
        \item\label{lem:general-vector-covector-duality:covector} $\exists\Phi\in\mathscr{L}(\mathscr{O}):\Phi|_{\underline{\Xi}}\leq\Xi$ and $U\subseteq\underline{\Phi}$.
    \end{enumerate}
\end{lemma}
\begin{proof}[Proof of Lemma \ref{lem:acyclic-oms}]
    Apply Lemma \ref{lem:general-vector-covector-duality} with $\Xi^+:=U$ and $\Xi^-:=\emptyset$.
\end{proof}
\begin{proof}[Proof of Lemma \ref{lem:general-vector-covector-duality}]
    First, observe that if both $\Psi$ and $\Phi$ exist then $\Psi\not\perp\Phi$: $\exists v\in\underline{\Psi}\cap U\neq\emptyset$, and this $v$ satisfies $0\neq\Psi(v)\leq\Xi(v)\geq\Phi(v)\neq 0$, i.e., $\Psi(v)=\Phi(v)\neq0$, yet for any $w\in\underline{\Psi}\subseteq\underline{\Xi}$ we have $0\neq\Psi(w)\leq\Xi(w)\geq\Phi(w)$, and consequently $0\neq\Psi(w)\geq\Phi(w)\neq-\Psi(w)$ for any $w\in\underline{\Psi}$, so there is no $w\in V$ for which $0\neq\Psi(w)=-\Phi(w)$. $\Psi\not\perp\Phi$ however contradicts the fact that any vector is orthogonal to any covector.

    Let $U'$ be the support of the conformal composition $\Psi_{\overline{U}}$ of all vectors in $\{\Psi\in\mathscr{V}(\mathscr{O}):\Psi\leq\Xi,\Psi\cap U=\emptyset\}$, and define the sign vector $\Phi_{\underline{\Xi}}$ by $\Phi_{\underline{\Xi}}|_{\underline{\Xi}-U'}:=\Xi|_{\underline{\Xi}-U'}$ and $\Phi_{\underline{\Xi}}|_{U'}=0$. Then necessarily $U\subseteq\underline{\Xi}-U'$. If $\Phi_{\underline{\Xi}}\in\mathscr{L}(\mathscr{O}|_{\underline{\Xi}})$ then $\exists\Phi\in\mathscr{L}(\mathscr{O})$ with $\Phi|_{\underline{\Xi}}=\Phi_{\underline{\Xi}}$, and this $\Phi$ satisfies point \ref{lem:general-vector-covector-duality:covector}. Otherwise there is a $\Psi_U^{\underline{\Xi}}\in\mathscr{V}(\mathscr{O}|_{\underline{\Xi}})$ with $\Phi_{\underline{\Xi}}\not\perp\Psi_U^{\underline{\Xi}}$ and therefore a corresponding $\Psi_U\in\mathscr{V}(\mathscr{O})$ with $\Psi_U|_{\underline{\Xi}}=\Psi_U^{\underline{\Xi}}$ and $\underline{\Psi_U}\subseteq\underline{\Xi}$. Either $\underline{\Psi_U}\subseteq U'$, or after possibly replacing $\Psi_U$ with $-\Psi_U$ we can assume that $\exists v\in\underline{\Xi}-U'$ for which $0\neq\Psi_U^{\underline{\Xi}}(v)=\Phi_{\underline{\Xi}}(v)$, but then $\Phi_{\underline{\Xi}}\not\perp\Psi_U^{\underline{\Xi}}$ implies that $\Psi_U|_{\underline{\Xi}-U'}=\Psi_U^{\underline{\Xi}}|_{\underline{\Xi}-U'}\leq\Phi_{\underline{\Xi}}|_{\underline{\Xi}-U'}=\Xi|_{\underline{\Xi}-U'}$. We can not have $\underline{\Psi_U}\subseteq U'$ as then from $\underline{\Psi_U}\cap\underline{\Phi_{\underline{\Xi}}}=\emptyset$ it can be derived that $\Phi_{\underline{\Xi}}\perp\Psi_U^{\underline{\Xi}}$. Now define $\Psi:=\Psi_{\overline{U}}\circ\Psi_U\in\mathscr{V}(\mathscr{O})$, and observe that $\Psi\leq\Xi$. It can not hold that $\underline{\Psi}\cap U=\emptyset$, as then the vector $\Psi\leq\Xi$ would have participated in the conformal composition defining $\Psi_{\overline{U}}$, but $\Psi>\Psi_{\overline{U}}$. In other words, $\Psi$ satisfies point \ref{lem:general-vector-covector-duality:vector}.  
\end{proof}
Another relevant specialization of Lemma \ref{lem:general-vector-covector-duality}, needed to prove Corollary \ref{crl:0-in-conv-interpreted-in-oms}, is the following:
\begin{lemma}\label{lem:cone-as-covectors}
    For an oriened matroid $\mathscr{O}$ on set of elements $V$, $U\subseteq V$, and $v\in V$, we have $v\in\conv(U)$ if and only if $\nexists\Phi\in\mathscr{L}(\mathscr{O})$ with $v\in\Phi^-$ and $U\cap\Phi^-=\emptyset$.
\end{lemma}
\begin{proof}
    For $v\in U$ the claim is trivial. For $v\in V-U$ apply Lemma \ref{lem:general-vector-covector-duality} with $\Xi^+:=U$, $\Xi^-:=\{v\}$, and $U$ of that lemma equal to $\{v\}$.
\end{proof}
Now we are ready to present the proof of Corollary \ref{crl:0-in-conv-interpreted-in-oms}.
\begin{proof}[Proof of Corollary \ref{crl:0-in-conv-interpreted-in-oms}]
    Point \ref{crl:0-in-conv-interpreted-in-oms:full} follows from point \ref{lem:0-in-conv-interpreted-in-oms:full} of Lemma \ref{lem:0-in-conv-interpreted-in-oms} combined with Lemma \ref{lem:acyclic-oms}. For point \ref{crl:0-in-conv-interpreted-in-oms:relint}, observe that $0\in\relint(\conv(A(U)))$ if and only if there is a positive vector of $\mathscr{O}$ with support $U$ by point \ref{lem:0-in-conv-interpreted-in-oms:relint} of Lemma \ref{lem:0-in-conv-interpreted-in-oms}, which holds if and only if $\mathscr{O}|_U$ has a positive vector with support $U$, or in other words $(\mathscr{O}|_U)^*$ has a positive covector with support $U$, which is equivalent to $(\mathscr{O}|_U)^*$ having no positive vector by Lemma \ref{lem:acyclic-oms}, but this is the same as stating that $\mathscr{O}|_U$ has no positive covector, which is definitionally the same as the claim $\nexists\Phi\in\mathscr{L}(\mathscr{O})\colon\Phi|_U$ is positive. Moving on, point \ref{crl:0-in-conv-interpreted-in-oms:cone} is just a combination of Lemma \ref{lem:cone-as-covectors} and point \ref{lem:0-in-conv-interpreted-in-oms:cone} of Lemma \ref{lem:0-in-conv-interpreted-in-oms}. Finally, point \ref{crl:0-in-conv-interpreted-in-oms:relintcone} is proved similarly to point \ref{crl:0-in-conv-interpreted-in-oms:relint}: by point \ref{lem:0-in-cov-interpreted-in-oms:relintcone} of Lemma \ref{lem:0-in-conv-interpreted-in-oms} we know that for $v\in V-U$ we have $A(v)\in\relint(\cone(A(U)))$ if and only if $\exists\Psi\in\mathscr{V}(\mathscr{O})$ with $\Psi^+=U$ and $\Psi^-=\{v\}$, which holds if and only if such a sign vector in $\mathscr{V}(\mathscr{O}|_{U\cup\{v\}})=\mathscr{L}((\mathscr{O}|_{U\cup\{v\}})^*)$ exists, but by Lemma \ref{lem:general-vector-covector-duality} this is the case precisely if there is no $\Phi_U\in\mathscr{V}((\mathscr{O}|_{U\cup\{v\}})^*)=\mathscr{L}(\mathscr{O}|_{U\cup\{v\}})$ with $\Phi_U(u)\leq +$ for $u\in U$, $\Phi_U(v)\leq -$, and $\Phi_U\neq 0$, and finally this is equivalent to saying that there is no covector $\Phi$ of $\mathscr{O}$ with $\Phi^-\cap U=\emptyset=\Phi^+\cap\{v\}$ and $\underline{\Phi}\cap (U\cup\{v\})\neq\emptyset$.
\end{proof}

\medskip
\subsection{Pseudosphere arrangements}\label{sec:prelim-oms:pseudospheres}
\medskip

In this section we review the Topological Representation Theorem of Folkman and Lawrence, and establish notation to talk about signed arrangements of pseudospheres. The general reference is \cite[Chapter 5]{OrientedMatroids}.

Fix throughout a sphere $S^{d-1}$ with $d\geq 1$. Later sometimes $d=0$ will be allowed, in which case we mean by $S^{-1}$ the empty set. A \emph{pseudosphere} in $S^{d-1}$ is the image of the equator $S^{d-2}$ under a self-homeomorphism $h\colon S^{d-1}\to S^{d-1}$. The homeomorphism $h$ is not part of the datum. Each pseudosphere cuts the sphere $S^{d-1}$ into two connected components, called its \emph{sides}, which are the images of the Northern and Southern hemispheres -- when $d=1$, the two sides are just the two points of $S^0$. Let a \emph{signed pseudosphere} be a pseudosphere $S$ together with a choice of a positive side and a negative side; in other words it can be encoded by a function $\mathfrak{s}\colon S^{d-1}\to\{-,0,+\}$ where $S=\mathfrak{s}^{-1}(0)$ and there exists a homeomorphism $h:S^{d-1}\to S^{d-1}$ such that $\mathfrak{s}(\varphi)=+$ if and only if $\varphi\in S^{d-1}$ is in the image of the Northern hemisphere. Call $\mathfrak{s}$ the \emph{sign function} of $S$. Write $B^+:=\mathfrak{s}^{-1}(+)$, $B^-:=\mathfrak{s}^{-1}(-)$, $D^+:=\mathfrak{s}^{-1}(\{0,+\})$ and $D^-:=\mathfrak{s}^{-1}(\{-,0\})$. A \emph{signed arrangement $\mathfrak{A}$ of pseudospheres} in $S^{d-1}$ indexed by $V$ is a collection of signed pseudospheres $(S_v)_{v\in V}$ described by sign functions $(\mathfrak{s}_v)_{v\in V}$ satisfying the following axioms:
\begin{enumerate}[label=\normalfont{(A\arabic*)}]
    \item $S(U):=\bigcap_{u\in U}S_u$ is a sphere for all $U\subseteq V$.
    \item If $S(U)\nsubseteq S_v$ for some $U\subseteq V$, $v\in V$, then $S(U)\cap S_v$ is a pseudosphere in $S(U)$ with sign function $\mathfrak{s}_v|_{S(U)}$ and consequently sides $S(U)\cap B^+_v$ and $S(U)\cap B^-_v$. 
    \item The intersection of an arbitrary collection of sets of the form $D^+_v$ and $D^-_v$ is either a sphere or a disk.
\end{enumerate}
For the purposes of this paper we slightly deviate from the definition above by allowing a signed arrangement of pseudospheres to include ``degenerate pseudospheres'' $S_v=S^{d-1}$ for $v\in V'\subseteq V$ ($\mathfrak{s}_v$ is the constant function at $0$ for such $v$) and then only demand that the axioms hold for $(S_v)_{v\in V-V'}$ and $(\mathfrak{s}_v)_{v\in V-V'}$. Then we can also talk about signed arrangements of pseudospheres in $S^{-1}$, where necessarily all pseudospheres are degenerate.

Given a signed arrangement $\mathfrak{A}$ of pseudospheres in $S^{d-1}$ indexed by $V$ and a point $\varphi\in S^{d-1}$, there is an associated sign vector $\mathfrak{s}_\mathfrak{A}(\varphi)\colon V\to\{-,0,+\}$ defined as $\mathfrak{s}_\mathfrak{A}(\varphi)(v):=\mathfrak{s}_v(\varphi)$. It turns out that $\mathfrak{s}_\mathfrak{A}(S^{d-1})\cup\{0\}$ is the set of covectors $\mathscr{L}(\mathscr{O})$ of some oriented matroid $\mathscr{O}$ \cite[Theorem 5.1.4]{OrientedMatroids}. The arrangement $\mathfrak{A}$ can be used to further study the oriented matroid $\mathscr{O}$, for instance the \emph{rank} $\operatorname{rank}(U)$ of a subset $U\subseteq V$ in $\mathscr{O}$ can be defined as $d-1-\dim S(U)$, and the rank of $\mathscr{O}$ is just $\operatorname{rank}(V)$; these values turn out to be functions of $\mathscr{O}$, and so independent of the particular arrangement $\mathfrak{A}$. Even better, the Topological Representation Theorem says that any oriented matroid arises from a signed arrangement of pseudospheres, so the rank is a well-defined concept in each oriented matroid.

The fact that oriented matroids associated to functions $A\colon V\to\R^d$ come from signed arrangements of pseudospheres is clear: the collection of $S_v:=\{\varphi\in S((\R^d)^*):\varphi(A(v))=0\}$ and $\mathfrak{s}_v(\varphi):=\sgn(\varphi(A(v)))$ for $v\in V$ is easily seen to form a signed arrangement $\mathfrak{A}$ of pseudospheres for which $\mathfrak{s}_\mathfrak{A}=\mathfrak{s}_A$. More generally, given a linear subspace $L\supseteq\operatorname{span}(A(V))$, the collection of $S_v:=\{\varphi\in S(L^*):\varphi(A(v))=0\}$ with $\mathfrak{s}_v(\varphi):=\sgn(\varphi(A(v)))$ is also a signed arrangement of pseudospheres in $S(L^*)\cong S^{\dim(L)-1}$ indexed by $V$ for which $\mathfrak{s}_\mathfrak{A}(\varphi)=\mathfrak{s}_A(\varphi')$ for any linear extension $\varphi'$ of $\varphi\in L^*$ to $\R^d$, and the associated oriented matroid is the same independently of the choice of $L$. In either case it can be verified that $\operatorname{rank}(U)=\dim(\operatorname{span}(A(U)))$. For other oriented matroids we do not need the strongest version of the Topological Representation Theorem, but instead the following simplified version of \cite[Theorem 5.2.1]{OrientedMatroids} will be sufficient. Here a \emph{loop-free} oriented matroid is an oriened matroid which has no \emph{loops}, i.e., no element $v\in V$ for which $\Phi(v)=0$ for all covectors $\Phi$.
\begin{theorem}\label{thm:top-repr}
    Let $\mathscr{L}$ be a set of sign vectors of the form $V\to\{-,0,+\}$, and let $k\leq d$. Then the following conditions are equivalent:
    \begin{enumerate}[label=\normalfont{(\roman*)}]
        \item $\mathscr{L}$ is the set of covectors of a loop-free oriented matroid of rank $k$ with set of elements $V$.
        \item $\mathscr{L}=\mathfrak{s}_\mathfrak{A}(S^{d-1})\cup\{0\}$ for some signed arrangement $\mathfrak{A}$ of pseudospheres in $S^{d-1}$ indexed by $V$ such that $\dim(S(V))=d-1-k$.
    \end{enumerate}
\end{theorem}
It is clear that the theorem still holds if we drop the loop-freeness condition and allow ``degenerate pseudospheres'' $S_v=S^{d-1}$ with $\mathfrak{s}_v\equiv 0$. Additionally, a signed arrangement $\mathfrak{A}$ of pseudospheres indexed by $V$ is called \emph{essential} if $S(V)=\emptyset$, and setting $k,d:=\operatorname{rank}(\mathscr{O})$ in the Topological Representation Theorem shows that every oriented matroid $\mathscr{O}$ is associated to some essential $\mathfrak{A}$, so assuming this property usually does not cause any problems. 

We conclude by introducing additional notation for dealing with signed arrangements of pseudospheres and proving a few lemmas about them. Whenever $U\subseteq V$ and $\mathfrak{s}\in\{-,+\}$, write $B^\mathfrak{s}(U):=\bigcap_{u\in U}B^\mathfrak{s}_u$, $B^\mathfrak{s}\langle U\rangle:=\bigcup_{u\in U}B^\mathfrak{s}_u$, $D^\mathfrak{s}(U):=\bigcap_{u\in U}D^\mathfrak{s}_u$, $D^\mathfrak{s}\langle U\rangle:=\bigcup_{u\in U}D^\mathfrak{s}_u$, $\partial D^\mathfrak{s}(U):=D^\mathfrak{s}(U)-B^\mathfrak{s}(U)$. Consider an intersection indexed by $U=\emptyset$ to be $S^{d-1}$ and a union indexed by $U=\emptyset$ to be the empty set. This notation is meant to mirror that found in \cite{Blagojevic2025}, but within the context of signed arrangements of pseudospheres. The following lemma, a paraphrasing of \cite[Lemma 5.1.8]{OrientedMatroids}, can be used to understand the homotopy type of $B^+(U)$, $D^+(U)$, and $\partial D^+(U)$:
\begin{lemma}\label{lem:disks-in-oms}
    Let $\mathfrak{A}$ be a signed arrangement of pseudospheres, $\mathscr{O}$ the associated oriented matroid, and $\Phi\in\mathscr{L}(\mathscr{O})-{0}$. Then  $D^-(\Phi^-)\cap S(\Phi^0)\cap D^+(\Phi^+)=\mathfrak{s}_\mathfrak{A}^{-1}(\{\Psi\in\mathscr{L}(\mathscr{O}):\Psi\leq\Phi\})$ is a $\dim S(\Phi^0)$-dimensional disk. Furthermore, its relative boundary as a subset of the sphere $S(\Phi^0)$ is 
    \[
        D^-(\Phi^-)\cap S(\Phi^0)\cap D^+(\Phi^+)-B^-(\Phi^-)\cap S(\Phi^0)\cap B^+(\Phi^+)=\mathfrak{s}_\mathfrak{A}^{-1}(\{\Psi\in\mathscr{L}(\mathscr{O})-\{0\}:\Psi<\Phi\}).
    \] 
\end{lemma}
\begin{observation}\label{obs:pseudospheres-regular-cw}
    If $\mathfrak{A}$ is a signed arrengement of pseudospheres in $S^{d-1}$ indexed by the finite set $V$ and with associated oriented matroid $\mathscr{O}$, then using Lemma \ref{lem:disks-in-oms} the arrangement $\mathfrak{A}$ gives a regular CW decomposition of $S^{d-1}$ whose open cells are of the form $B^-(\Phi^-)\cap S(\Phi^0)\cap B^+(\Phi^+)$. There is one cell for each $\Phi\in\mathscr{L}(\mathscr{O})-\{0\}$, and if $\mathfrak{A}$ is not essential then there are additional cells used to decompose $S(V)$.
\end{observation}
\begin{observation}\label{obs:restriction-of-pseudosphere-arrangements}
    Let $\mathfrak{A}$ be a signed arrangement of pseudospheres in $S^{d-1}$ indexed by $V$, $\mathscr{O}$ the associated oriented matroid, and $\emptyset\neq U\subseteq V$. $\mathfrak{A}|_U$ consisting of $(S_u)_{u\in U}$ and $(\mathfrak{s}_u)_{u\in U}$ is a signed arrangement of pseudospheres with associated oriented matroid $\mathscr{O}|_U$, and for which the definition of the sets $S_u$, $S(U')$, $B^\mathfrak{s}_u$, $B^\mathfrak{s}(U')$, $B^\mathfrak{s}\langle U'\rangle$, $D^\mathfrak{s}_u$, $D^\mathfrak{s}(U')$, $D^\mathfrak{s}\langle U'\rangle$, and $\partial D^\mathfrak{s}(U')$ for $u\in U$, $U'\subseteq U$, and $\mathfrak{s}\in\{-,+\}$ is independent of which signed arrangement of pseudospheres they are interpreted in: $\mathfrak{A}$ or $\mathfrak{A}|_U$.
\end{observation}
\begin{corollary}\label{crl:disks-in-oms}
    Let $\mathfrak{A}$ be a signed arrangement of pseudospheres in $S^{d-1}$ indexed by $V$, and $\emptyset\neq U\subseteq V$. Then $B^+(U)$ is empty or a $(d-1)$-dimensional open ball, and in the latter case $D^+(U)$ is a $(d-1)$-dimensional disk and its boundary $\partial D^+(U)$ is homeomorphic to a $(d-2)$-dimensional sphere.
\end{corollary}
\begin{proof}
    Applying Observation \ref{obs:restriction-of-pseudosphere-arrangements} we may assume without loss of generality that $U=V$. Suppose $B^+(U)\neq \emptyset$, and pick $\Phi\in\mathfrak{s}_\mathfrak{A}(B^+(U))$. Then $\Phi^+=U=V$, and from Lemma \ref{lem:disks-in-oms} we can derive that $D^-(\Phi^-)\cap S(\Phi^0)\cap D^+(\Phi^+)=D^+(\Phi^+)=D^+(U)$ is a $(d-1)$-dimensional disk, $B^+(U)$, as it is the interior of $D^+(U)$, is an open $(d-1)$-dimensional ball, and the boundary, a $(d-2)$-dimensional sphere, is precisely equal to $D^+(\Phi^+)-B^+(\Phi^+)=\partial D^+(U)$. 
\end{proof}
\begin{corollary}\label{crl:disk-with-partial-boundary-in-oms}
    Let $\mathfrak{A}$ be a signed arrangement of pseudospheres in $S^{d-1}$ indexed by $V$, $\emptyset\neq U_B\subseteq V$, $U_D\subseteq V$. Then $B^+(U_B)\cap D^-(U_D)$ is empty or contractible.
\end{corollary}
\begin{proof}    
    Assume $B^+(U_B)\cap D^-(U_D)\neq\emptyset$ and without loss of generality that $D^-(U_D)\neq S^{d-1}$ (that case is covered by Corollary \ref{crl:disks-in-oms}), and use Observation \ref{obs:restriction-of-pseudosphere-arrangements} to replace $V$ with $U_B\cup U_D$. Let $\Phi$ be the composition of all covectors in $\mathfrak{s}_\mathfrak{A}(B^+(U_B)\cap D^-(U_D))$; by design this is a conformal composition and so $\Phi$ is well-defined and greater than or equal to all covectors in $\mathfrak{s}_\mathfrak{A}(B^+(U_B)\cap D^-(U_D))$. Moreover, $\Phi\neq 0$ as $U_B\neq\emptyset$. In particular, according to Lemma \ref{lem:disks-in-oms} the space $D^-(\Phi^-)\cap S(\Phi^0)\cap D^+(\Phi^+)$ is a $\dim S(\Phi^0)$-dimensional disk with interior $B^-(\Phi^-)\cap S(\Phi^0)\cap B^+(\Phi^+)$, and we also have 
    \begin{multline*}
        B^-(\Phi^-)\cap S(\Phi^0)\cap B^+(\Phi^+)=\mathfrak{s}_\mathfrak{A}^{-1}(\{\Phi\})
        \subseteq\\\subseteq 
        B^+(U_B)\cap D^-(U_D)
        \subseteq\\\subseteq
        \mathfrak{s}_{\mathfrak{A}}^{-1}(\{\Psi\in\mathscr{L}(\mathscr{O}):\Psi\leq\Phi\})=D^-(\Phi^-)\cap S(\Phi^0)\cap D^+(\Phi^+).
    \end{multline*}
    Any subset of a disk which contains the interior is contractible using a straight-line homotopy to the center of the disk, so $B^+(U_B)\cap D^-(U_D)$ is contractible too.
\end{proof}
\begin{remark}\label{rmk:reorientation}
    Given a signed arrangement $\mathfrak{A}$ of pseudospheres in $S^{d-1}$ indexed by $V$ with associated oriented matroid $\mathscr{O}$, and a subset $U\subseteq V$, there is another signed arrangement $\mathfrak{A}_{-U}$ of pseudospheres in $S^{d-1}$ indexed by $V$, with associated oriented matroid denoted by $\mathscr{O}_{-U}$, for which $\mathfrak{s}_{\mathfrak{A}_{-U}}(\varphi)(v)$ is equal to $\mathfrak{s}_\mathfrak{A}(\varphi)(v)$ if $v\in V-U$ and $-\mathfrak{s}_\mathfrak{A}(\varphi)(v)$ if $v\in U$. Then for $v\in V$ and $\mathfrak{s}\in\{-,+\}$ the sets $B^\mathfrak{s}_v$ and $D^\mathfrak{s}_v$ defined in $\mathfrak{A}_{-U}$ are respectively equal to $B^\mathfrak{s}_v$ and $D^\mathfrak{s}_v$ defined in $\mathfrak{A}$ if $v\in V-U$, or to $B^{-\mathfrak{s}}_v$ and $D^{-\mathfrak{s}}_v$ defined in $\mathfrak{A}$ if $v\in U$. This can be combined with the last corollary to show that any intersection of sets of the form $B^{\mathfrak{s}}_v$ and $D^\mathfrak{s}_v$ is empty contractible, given at least one $B^\mathfrak{s}_v$ is used.
\end{remark}
If $\mathscr{O}$ is the oriented matroid associated to $\mathfrak{A}$, define analogously $\mathscr{S}_v:=\{\Phi\in\mathscr{L}(\mathscr{O}):\Phi(v)=0\}$, $\mathscr{B}^\mathfrak{s}_v:=\{\Phi\in\mathscr{L}(\mathscr{O}):\Phi(v)=\mathfrak{s}\}$, and $\mathscr{D}^\mathfrak{s}_v:=\{\Phi\in\mathscr{L}(\mathscr{O}):\Phi(v)\in\{0,\mathfrak{s}\}\}$ for $v\in V$ and $\mathfrak{s}\in\{-,+\}$, and $\mathscr{S}(U):=\bigcap_{u\in U}\mathscr{S}_u$, $\mathscr{B}^\mathfrak{s}(U):=\bigcap_{u\in U}\mathscr{B}^\mathfrak{s}_u$, $\mathscr{B}^\mathfrak{s}\langle U\rangle:=\bigcup_{u\in U}\mathscr{B}^\mathfrak{s}_u$, $\mathscr{D}^\mathfrak{s}(U):=\bigcap_{u\in U}\mathscr{D}^\mathfrak{s}_u$, $\mathscr{D}^\mathfrak{s}\langle U\rangle=\bigcup_{u\in U}\mathscr{D}^\mathfrak{s}_u$, and $\partial\mathscr{D}^\mathfrak{s}(U):=\mathscr{D}^\mathfrak{s}(U)-\mathscr{B}^\mathfrak{s}(U)$ for $U\subseteq V$, $\mathfrak{s}\in\{-,+\}$. Again, consider an intersection indexed by $U=\emptyset$ to be $\mathscr{L}(\mathscr{O})$ and a union indexed by $U=\emptyset$ to be the empty set. The sense in which the subsets of $\mathscr{L}(\mathscr{O})$ defined above are analogues of the subsets of $S^{d-1}$ defined previously is made precise by the observation below.
\begin{observation}\label{obs:curly-vs-print}
    Let $\mathfrak{A}$ be a signed arrangement of pseudospheres in $S^{d-1}$ indexed by $V$, and $\mathscr{O}$ the associated oriented matroid. Let $\mathfrak{F}_\mathfrak{A}$ be the set of all formulas which can be constructed using the symbols $(-)\cap(-)$, $(-)\cup(-)$, as well as $S_v$, $B^\mathfrak{s}_v$, $D^\mathfrak{s}_v$ for $\mathfrak{s}\in\{-,+\}$ and $v\in V$. Similarly, define $\mathfrak{F}_\mathscr{O}$ as the set of formulas which can be constructed using $(-)\cap(-)$, $(-)\cup(-)$, and $\mathscr{S}_v$, $\mathscr{B}^\mathfrak{s}_v$, $\mathscr{D}^\mathfrak{s}_v$ with $\mathfrak{s}\in\{+,-\}$ and $v\in V$. Then there is a bijection between $\mathfrak{F}_\mathfrak{A}$ and $\mathfrak{F}_\mathscr{O}$. Let $\mathcal{F}$ be a formula in $\mathfrak{F}_\mathfrak{A}$, $F$ the subset of $S^{d-1}$ defined by $\mathcal{F}$, and let $\mathscr{F}$ be the subset of $\mathscr{L}(\mathscr{O})$ defined by the corresponding formula in $\mathfrak{F}_\mathscr{O}$. Then $\mathfrak{s}_\mathfrak{A}^{-1}(\mathscr{F})=F$, $\mathscr{F}\supseteq \mathfrak{s}_\mathfrak{A}(F)$, and $\mathscr{F}-\mathfrak{s}_\mathfrak{A}(F)$ is either $\emptyset$ or $\{0\}$. In particular, $F=\mathfrak{s}_\mathfrak{A}^{-1}(\mathscr{F})$ and $\mathfrak{s}_\mathfrak{A}(F)=\mathscr{F}$ hold under the condition that $\mathcal{F}$ is obtained using $(-)\cap(-)$, $(-)\cup(-)$, and formulas of the form $(B^\mathfrak{s}_v)\cap(\mathcal{F}')$, where $\mathfrak{s}\in\{-,+\}$, $v\in V$, and $\mathcal{F}'\in\mathfrak{F}_\mathfrak{A}$.

    To prove all of this, observe that by replacing $(-)\cap(-)$ with ``$(-)$ and $(-)$'', $(-)\cup(-)$ with ``$(-)$ or $(-)$'', $S_v$ with $\Xi(v)=0$, $B^\mathfrak{s}_v$ with $\Xi(v)=\mathfrak{s}$, and $D^\mathfrak{s}_v$ with $\Xi(v)\neq -\mathfrak{s}$, we get a formula $\mathcal{F}'(\Xi)$ with a free variable $\Xi$ such that $F=\{\varphi\in S^{d-1}:\mathcal{F}'(\mathfrak{s}_\mathfrak{A}(\varphi))\}$ and $\mathscr{F}=\{\Phi\in\mathscr{L}(\mathscr{O}):\mathcal{F}'(\Phi)\}$.
\end{observation}  
In particular, the following corollary automatically follows:
\begin{corollary}\label{crl:intersections-in-pseudsphere-arrangements-vs-covectors}
    If $\mathscr{O}$ is the oriented matroid associated to the signed arrangement $\mathfrak{A}$ of pseudospheres indexed by $V$, $U\subseteq V$, and $v\in V$, then:
    \begin{enumerate}[label=\normalfont{(\arabic*)}]
        \item\label{crl:intersections-in-pseudsphere-arrangements-vs-covectors:full} $B^+(U)=\emptyset\impliedby\mathscr{B}^+(U)=\emptyset\iff\nexists\Phi\in\mathscr{L}(\mathscr{O})\colon U\subseteq\Phi^+$,
        \item\label{crl:intersections-in-pseudsphere-arrangements-vs-covectors:relint} $\xymatrix@R=10pt{
                B^+\langle U\rangle\cap D^+(U)=\emptyset
                \ar@{<=>}[r]
                \ar@{<=>}[d]
            &   \mathscr{B}^+\langle U\rangle\cap\mathscr{D}^+(U)=\emptyset
                \ar@{<=>}[r]
                \ar@{<=>}[d]
            &   \nexists\Phi\in\mathscr{L}(\mathscr{O}):\Phi|_U\text{~is positive,}
            \\  D^+(U)=S(U)
            &   \mathscr{D}^+(U)=\mathscr{S}(U)
        }$
        \item\label{crl:intersections-in-pseudsphere-arrangements-vs-covectors:cone} $D^+(U)\cap B^-_v=\emptyset\iff\mathscr{D}^+(U)\cap\mathscr{B}^-_v=\emptyset\iff\nexists\Phi\in\mathscr{L}(\mathscr{O})$ with $v\in\Phi^-$ and $U\cap\Phi^-=\emptyset$.
        \item\label{crl:intersections-in-pseudosphere-arrangements-vs-covectors:relintcone} Assume  $v\in V-U$. Then $D^+(U)\cap D^-_v-S(U\cup\{v\})=\emptyset\iff\mathscr{D}^+(U)\cap\mathscr{D}^-_v-\mathscr{S}(U\cup\{v\})=\emptyset\iff\nexists\Phi\in\mathscr{L}(\mathscr{O})$ with $\Phi^-\cap U=\emptyset=\Phi^+\cap\{v\}$ and $\underline{\Phi}\cap (U\cup \{v\})\neq\emptyset$.
    \end{enumerate}
\end{corollary}
When the pseudosphere arrangement comes from a function $A\colon V\to\R^d$, Corollaries \ref{crl:0-in-conv-interpreted-in-oms} and \ref{crl:intersections-in-pseudsphere-arrangements-vs-covectors} above can be used to characterize when statements such as $0\in\conv(A(U))$, $0\in\relint(\conv(A(U)))$, $A(v)\in\cone(A(U))$, $A(v)\in\relint(\cone(A(U)))$, and  $0\in\interior(\conv(A(U)))$ hold in terms of emptiness of $\mathscr{B}^+(U)$, $\mathscr{B}^+\langle U\rangle\cap \mathscr{D}^+(U)$, $\mathscr{B}^-_v\cap \mathscr{D}^+(U)$, $\mathscr{D}^+(U)\cap\mathscr{D}^-_v-\mathscr{S}(U\cup\{v\})$, and $D^+(U)$ respectively:
\begin{corollary}\label{crl:0-in-conv-final-description}
    If the oriented matroid $\mathscr{O}$ is associated to the function $A\colon V\to\R^d$, $\mathfrak{A}$ is the signed arrangement of pseudospheres in $S((\R^d)^*)$ associated to $A$, $U\subseteq V$, and $v\in V$, then:
    \begin{enumerate}[label=\normalfont{(\arabic*)}]
        \item\label{crl:0-in-conv-final-description:conv} $0\in\conv(A(U))\iff \mathscr{B}^+(U)=\emptyset\implies B^+(U)=\emptyset$,
        \item\label{crl:0-in-conv-final-description:relint} $0\in\relint(\conv(A(U)))\iff \mathscr{D}^+(U)=\mathscr{S}(U)\iff D^+(U)=S(U)$,
        \item\label{crl:0-in-conv-final-description:cone} $A(v)\in\cone(A(U))\iff\mathscr{D}^+(U)\cap\mathscr{B}^-_v=\emptyset\iff D^+(U)\cap B^-_v=\emptyset$, and
        \item\label{crl:0-in-conv-final-description:relintcone} If $v\in V-U$ then $A(v)\in\relint(\cone(A(U)))\iff\mathscr{D}^+(U)\cap\mathscr{D}^-_v-\mathscr{S}(U\cup\{v\})=\emptyset\iff D^+(U)\cap D^-_v-S(U\cup\{v\})=\emptyset$.
        \item\label{crl:0-in-conv-final-description:int} $0\in\interior(\conv(A(U)))\iff D^+(U)=S(U)=\emptyset\implies\mathscr{D}^+(U)=\mathscr{S}(U)=\{0\}$.
    \end{enumerate}
\end{corollary}
\begin{proof}
    Points \ref{crl:0-in-conv-final-description:conv}, \ref{crl:0-in-conv-final-description:relint}, \ref{crl:0-in-conv-final-description:cone}, and \ref{crl:0-in-conv-final-description:relintcone} follow from Corollaries \ref{crl:0-in-conv-interpreted-in-oms} and \ref{crl:intersections-in-pseudsphere-arrangements-vs-covectors}.
    For point \ref{crl:0-in-conv-final-description:int}, observe that $0\in\interior(\conv(A(U)))$ if and only if $0\in\relint(\conv(A(U)))$ and $\operatorname{span}(A(U))=\R^d$. $A(U)$ fails to span $\R^d$ if and only if it is contained in a proper linear subspace, which is equivalent to saying that there is a $\varphi\in(\R^d)^*$ such that $\varphi(A(u))=0$ for all $u\in U$, so $0\in\interior(\conv(A(U)))$ if and only if $0\in\relint(\conv(A(U)))$ and $S(U)=\emptyset$, which happens if and only if $D^+(U)=S(U)=\emptyset$ by point \ref{crl:0-in-conv-final-description:relint}.
\end{proof}

There is one more construction regarding pseudosphere arrangements which will be of use later on. For that, write $T_\Phi:=B^+(\Phi^+)\cap B^-(\Phi^-)$ where $\Phi\neq 0$ is a covector, and introduce the notation $T(\mathscr{F}):=\bigcap_{\Phi\in\mathscr{F}}T_\Phi$ and $T\langle\mathscr{F}\rangle:=\bigcup_{\Phi\in\mathscr{F}}T_\Phi$ where $\mathscr{F}$ is a set of covectors.
\begin{lemma}\label{lem:tubular-neighborhoods-in-oms}
    Let $\mathfrak{A}$ be a signed arrangement of pseudospheres, $\mathscr{O}$ its associated oriented matroid, and let $\mathscr{F}\subseteq\mathscr{L}(\mathscr{O})-\{0\}$ be a set of covectors. Then:
    \begin{enumerate}[label=\normalfont{(\arabic*)}]
        \item\label{lem:tubular-neighborhoods-in-oms:intersection} $T(\mathscr{F})$ is nonempty if and only if the composition $\Phi_\mathscr{F}$ of all covectors in $\mathscr{F}$ is conformal, in which case $T(\mathscr{F})=T_{\Phi_\mathscr{F}}$.
        \item\label{lem:tubular-neighborhoods-in-oms:union} If $\mathfrak{s}_\mathfrak{A}^{-1}(\mathscr{F})$ is open then $T\langle \mathscr{F}\rangle=\mathfrak{s}_\mathfrak{A}^{-1}(\mathscr{F})$.
    \end{enumerate} 
\end{lemma}
\begin{proof}
    First, observe that $\mathfrak{s}_\mathfrak{A}^{-1}(\Phi)\subseteq T_\Phi$ for any covector $\Phi$, and more generally $T_\Phi=\mathfrak{s}_\mathfrak{A}^{-1}(\{\Psi\in\mathscr{L}(\mathscr{O}):\Psi\geq\Phi\})$. The composition of all covectors in $\mathscr{F}$ is not conformal if and only if $\exists v\in V, \Phi_-,\Phi_+\in\mathscr{F}\colon\Phi_-(v)=-,\Phi_+(v)=+$. If this latter condition holds then $T(\mathscr{F})\subseteq T_{\Phi_-}\cap T_{\Phi_+}\subseteq B^-_v\cap B^+_v=\emptyset$, while otherwise $\Phi_\mathscr{F}^+=\bigcup_{\Phi\in\mathscr{F}}\Phi^+$ and $\Phi_\mathscr{F}^-=\bigcup_{\Phi\in\mathscr{F}}\Phi^-$, meaning that 
    \[
        T(\mathscr{F})=\bigcap_{\Phi\in\mathscr{F}}B^-(\Phi^-)\cap B^+(\Phi^+)=B^-(\bigcup_{\Phi\in\mathscr{F}}\Phi^-)\cap B^+(\bigcup_{\Phi\in\mathscr{F}}\Phi^+)=T_{\Phi_{\mathscr{F}}}\supseteq\mathfrak{s}_\mathfrak{A}^{-1}(\Phi_{\mathscr{F}})\neq\emptyset.
    \]
    This concludes the proof of point \ref{lem:tubular-neighborhoods-in-oms:intersection}. $\mathfrak{s}_\mathfrak{A}^{-1}(\mathscr{F})=\bigcup_{\Phi\in\mathscr{F}}\mathfrak{s}_\mathfrak{A}^{-1}(\Phi)\subseteq T\langle\mathscr{F}\rangle$ holds for arbitrary $\mathscr{F}\subseteq\mathscr{L}(\mathscr{O})-\{0\}$, so for point \ref{lem:tubular-neighborhoods-in-oms:union} it suffices to prove that if $\mathfrak{s}_\mathfrak{A}^{-1}(\mathscr{F})$ is open then $T_\Phi\subseteq\mathfrak{s}_\mathfrak{A}^{-1}(\mathscr{F})$ for all $\Phi\in\mathscr{F}$. In light of $T_\Phi=\mathfrak{s}_\mathfrak{A}^{-1}(\{\Psi\in\mathscr{L}(\mathscr{O}):\Psi\geq\Phi\})$ this boils down to verifying that if $\Phi\leq\Psi\in\mathscr{L}(\mathscr{O})$ then the closure of $\mathfrak{s}_\mathfrak{A}^{-1}(\Psi)$ intersects $\mathfrak{s}_\mathfrak{A}^{-1}(\Phi)$. This, however, is a consequence of Lemma \ref{lem:disks-in-oms}.
\end{proof}

\bigskip
\section{Avoiding complexes, and links therein}\label{sec:avoiding-complexes}
\bigskip

Here we study the zero-avoiding, support, vector-avoiding, and element-avoiding complexes introduced in Section \ref{sec:the-method}, building on ideas in \cite{Blagojevic2025}. We are interested in the homotopy type and (reduced) (co)homology of these complexes and of the links of simplices within them. Section \ref{sec:avoiding-complexes:zero} focuses on the zero-avoiding and the support complexes, while Section \ref{sec:avoiding-complexes:vector} is dedicated to the vector-avoiding and element-avoiding complexes. In either section we first recall the definition of these complexes, describe the relation between them, then summarize our claims regarding the homotopy type of the complexes and of the links, and finally translate these results into statements about their (reduced) (co)homology groups. Both sections end with concluding that all of these complexes are near-$(d-1)$-Leray for a relevant choice of $d$.

\medskip
\subsection{Zero-avoiding and support complexes}\label{sec:avoiding-complexes:zero}
\medskip

Recall from Section \ref{sec:the-method} that given a nonempty finite set $V$ and a function $A\colon V\to\R^d$, the \emph{zero-avoiding complex} is the simplicial complex
\[
    \mathscr{C}_A:=\{U\subseteq V:0\notin\conv(A(U))\},
\]
while for an oriented matroid $\mathscr{O}$ on set of elements $V$ the \emph{support complex} is defined to be
\[
    \mathscr{C}_\mathscr{O}:=\{U\subseteq V:U\text{~contains a positive circuit of~}\mathscr{O}\}.
\]
Any vector of an oriented matroid is a conformal composition of all circuits smaller than it, so $U\in\mathscr{C}_\mathscr{O}$ if and only if $U$ contains a positive vector of $\mathscr{O}$. Applying Lemma \ref{lem:acyclic-oms} and Corollary \ref{crl:intersections-in-pseudsphere-arrangements-vs-covectors}, we obtain an equivalent definition of the support complex:
\[
    \mathscr{C}_\mathscr{O}=\{U\subseteq V:\mathscr{B}^+(U)\neq\emptyset\}.
\]
As a consequence of Lemma \ref{lem:0-in-conv-interpreted-in-oms}, if $\mathscr{O}$ is the oriented matroid associated to the function $A$, then $\mathscr{C}_A=\mathscr{C}_\mathscr{O}$. Therefore, in the following we only study the support complex $\mathscr{C}_\mathscr{O}$, but the results about it can be translated into the language of convex geometry and linear functionals using Corollary \ref{crl:0-in-conv-final-description}, see for instance Example \ref{ex:zero-avoiding-lemmas} -- indeed, the results of the example are just lemmas in \cite{Blagojevic2025} about $\mathscr{C}_A$. 

First and foremost, the homotopy types of links $\link_{\mathscr{C}_\mathscr{O}}(V-U)$ in this complex can be characterized as follows:
\begin{proposition}\label{prop:htpy-type-zero-avoiding}
    Let $\mathfrak{A}$ be a signed arrangement of pseudospheres in $S^{d-1}$ indexed by the finite set $V$, and let $\mathscr{O}$ be its associated oriented matroid. The link $\link_{\mathscr{C}_\mathscr{O}}(V-U)=(\mathscr{C}^*_{\mathscr{O}}[U])^*$ for $U\subseteq V$ can be described as follows:
    \begin{enumerate}[label=\normalfont{(\arabic*)}]
        \item\label{prop:htpy-type-zero-avoiding:void} If $\mathscr{B}^+(V-U)=\emptyset$ then $\link_{\mathscr{C}_{\mathscr{O}}}(V-U)$ is the void complex.
        \item\label{prop:htpy-type-zero-avoiding:nerve} If $\mathscr{B}^+(V-U)\neq\emptyset$ then $\link_{\mathscr{C}_{\mathscr{O}}}(V-U)$ is homotopy equivalent to $B^+\langle U\rangle\cap B^+(V-U)$.
    \end{enumerate}
    In particular, if we are in case \ref{prop:htpy-type-zero-avoiding:nerve} then $\link_{\mathscr{C}_\mathscr{O}}(V-U)$ is homotopy equivalent to an open subset of $S^{d-1}$. The assumption $\mathscr{B}^+(V-U)\neq\emptyset$ is automatically satisfied if $U=V$, in which case we have the implications:
    \begin{enumerate}[label=\normalfont{(2.\arabic*)}]
        \item\label{prop:htpy-type-zero-avoiding:sphere} If $\mathscr{D}^+(V)=\{0\}$ then $\mathscr{C}_\mathscr{O}=\link_{\mathscr{C}_\mathscr{O}}(\emptyset)\simeq S^{\operatorname{rank}(\mathscr{O})-1}$.
        \item\label{prop:htpy-type-zero-avoiding:entire-contr} Otherwise $\mathscr{C}_\mathscr{O}=\link_{\mathscr{C}_\mathscr{O}}(\emptyset)$ is contractible.
    \end{enumerate}
    If we are in case \ref{prop:htpy-type-zero-avoiding:nerve} but under the assumption that $U\subsetneq V$ then:
    \begin{enumerate}[resume,label=\normalfont{(2.\arabic*)}]
        \item\label{prop:htpy-type-zero-avoiding:non-Farkas} If $\mathscr{B}^+(V-U)\cap \mathscr{D}^-(U)=\emptyset$ then $\link_{\mathscr{C}_{\mathscr{O}}}(V-U)$ is contractible.
        \item\label{prop:htpy-type-zero-avoiding:boundary} Otherwise $\link_{\mathscr{C}_\mathscr{O}}(V-U)$ is also homotopy equivalent to
        \[
            S^{\operatorname{rank}(U')-1}*\left(S(U')\cap B^+\langle U-U'\rangle\cap B^+(V-U)\right)
        \]
        for any $U'\subseteq U$ with $\mathscr{D}^+(U')=\mathscr{S}(U')$, where by convention $S^{-1}=\emptyset$.
    \end{enumerate}
\end{proposition}
Before the proof we show some examples and make some remarks.
\begin{example}\label{ex:zero-avoiding-lemmas}
    Using Proposition \ref{prop:htpy-type-zero-avoiding} and Corollary \ref{crl:0-in-conv-final-description} the homotopy type of $\link_{\mathscr{C}_A}(V-U)$ can be computed under various assumptions phrased in the language of convex geometry. For example:
    \begin{enumerate}[label=\normalfont{(\arabic*)}]
        \item\label{ex:zero-avoiding-lemmas:full} If $0\in\relint(\conv(A(V)))$ then $\mathscr{C}_A=\link_{\mathscr{C}_A}(\emptyset)$ is homotopy equivalent to a sphere of dimension $\dim(\operatorname{span}(A(V)))-1$, and otherwise it is contractible. Note that $\mathscr{S}(V)=\{0\}$ in any oriented matroid, like the $\mathscr{O}$ associated to $A$, so $\mathscr{D}^+(V)=\{0\}\iff\mathscr{D}^+(V)=\mathscr{S}(V)$, which by Corollary \ref{crl:0-in-conv-final-description} is equivalent to the claim $0\in\relint(\conv(A(U)))$. On the other hand $\operatorname{rank}(\mathscr{O})=\dim(\operatorname{span}(A(V)))$, so points \ref{prop:htpy-type-zero-avoiding:sphere} and \ref{prop:htpy-type-zero-avoiding:entire-contr} of Proposition \ref{prop:htpy-type-zero-avoiding} imply the claim.
    \end{enumerate} 
    With a little effort it can be shown that point \ref{prop:htpy-type-zero-avoiding:non-Farkas} of the Proposition translates to point \ref{ex:zero-avoiding-lemmas:cones} below. We include the example \ref{ex:zero-avoiding-lemmas:int} because its assumptions are more easily digestible, but it is a consequence of \ref{ex:zero-avoiding-lemmas:cones}: under the assumptions of point \ref{ex:zero-avoiding-lemmas:int} for any $v\in V-U$ it is implied that $A(v)\neq 0$ by $0\notin\conv(A(V-U))$, while $0\in\interior(\conv(A(U)))$ means that $\cone(A(U))=\R^d$, so $\cone(A(V-U))\cap\cone(A(U))\supseteq\cone(A(v))\neq\{0\}$, so point \ref{ex:zero-avoiding-lemmas:cones} implies the desired conclusion.
    \begin{enumerate}[resume,label=\normalfont{(\arabic*)}]
        \item\label{ex:zero-avoiding-lemmas:int} If $U\subsetneq V$, $0\notin\conv(A(V-U))$, and $0\in\interior(\conv(A(U)))$ then $\link_{\mathscr{C}_A}(V-U)$ is contractible. The conditions imply $\mathscr{B}^+(V-U)\neq\emptyset$ and $\mathscr{D}^+(U)=\mathscr{S}(U)=\{0\}$ respectively by Corollary \ref{crl:0-in-conv-final-description}, so in particular $\mathscr{D}^-(U)=\mathscr{S}(U)=\{0\}$ and $\mathscr{B}^+(V-U)\cap\mathscr{D}^-(U)=\mathscr{B}^+(V-U)\cap\{0\}=\emptyset$ by $U\neq V$. Thus this situation falls under case \ref{prop:htpy-type-zero-avoiding:non-Farkas} of the Proposition, and the link is contractible.
        \item\label{ex:zero-avoiding-lemmas:cones} If $U\subsetneq V$, $0\notin\conv(A(V-U))$, and $\cone(A(V-U))\cap\cone(A(U))\neq\{0\}$ then $\link_{\mathscr{C}_A}(V-U)$ is contractible. This again follows from case \ref{prop:htpy-type-zero-avoiding:non-Farkas} by Corollary \ref{crl:0-in-conv-final-description} together with Lemma \ref{lem:intersection-of-cones-in-oms}: $\cone(A(V-U))\cap\cone(A(U))\neq\{0\}$ implies according to this latter lemma that there is a vector $\Psi$ of the oriented matroid $\mathscr{O}$ associated to $A$ with $\emptyset\neq\Psi^+\subseteq V-U$ and $\Psi^-\subseteq U$, so there can not be a $\Phi\in \mathscr{B}^+(V-U)\cap \mathscr{D}^-(U)$, as then it would not hold that $\Psi\perp\Phi$, yet $\Psi$ is a vector and $\Phi$ is a covector of $\mathscr{O}$.
    \end{enumerate}
    As \ref{ex:zero-avoiding-lemmas:cones} is a translation of point \ref{prop:htpy-type-zero-avoiding:non-Farkas}, necessarily point \ref{prop:htpy-type-zero-avoiding:boundary} of the Proposition applies if and only if $U\subsetneq V$, $0\notin\conv(A(V-U))$, and $\cone(A(V-U))\cap\cone(A(U))=\{0\}$.
    \begin{enumerate}[resume,label=\normalfont{(\arabic*)}]
        \item If point \ref{prop:htpy-type-zero-avoiding:boundary} of Proposition \ref{prop:htpy-type-zero-avoiding} applies and there is a subset $U'\subseteq U$ with $0$ lying in $\relint(\conv(A(U')))$ and $\dim(\operatorname{span}(A(U')))=d-1$ then $\link_{\mathscr{C}_A}(V-U)$ is homotopy equivalent to a $d-2=(\dim(\operatorname{span}(A(U')))-1)$-sphere. For this we just have to see that the second half of the join in point \ref{prop:htpy-type-zero-avoiding:boundary} is empty, as $\operatorname{rank}(U')=\dim(\operatorname{span}(A(U')))$ and $0\in\relint(\conv(A(U')))\implies\mathscr{D}^+(U')=\mathscr{S}(U')$ by Corollary \ref{crl:0-in-conv-final-description}. 
        $S(U')\cong S^{d-1-\operatorname{rank}(U')}=S^0$, so $\mathscr{S}(U')\subseteq \mathfrak{s}_{\mathfrak{A}}(S(U'))\cup\{0\}$ has at most two non-zero elements. One of these must be a covector $0\neq\Phi\in\mathscr{B}^+(V-U)\cap\mathscr{D}^-(U)\subseteq\mathscr{D}^-(U)\subseteq\mathscr{D}^-(U')=\mathscr{S}(U')$ guaranteed by the assumption of point \ref{prop:htpy-type-zero-avoiding:boundary}, but then $-\Phi$ is the other one; denote the points of $S(U')$ corresponding to these under $\mathfrak{s}_\mathfrak{A}$ by $\varphi$ and $-\varphi$ respectively. $\Phi\in\mathscr{D}^-(U)\subseteq\mathscr{D}^-(U-U')$, so $\varphi\notin B^+\langle U-U'\rangle$, while $-\Phi\in\mathscr{B}^-(V-U)$ so $-\varphi\notin B^+(V-U)$ (using $U\subsetneq V$ and Observation \ref{obs:curly-vs-print}). In conclusion, $S(U')\cap B^+\langle U-U'\rangle\cap B^+(V-U)=\{\varphi,-\varphi\}\cap B^+\langle U-U'\rangle\cap B^+(V-U)=\emptyset$, which is just what we wanted.
        \item If point \ref{prop:htpy-type-zero-avoiding:boundary} of Proposition \ref{prop:htpy-type-zero-avoiding} applies and $0\in\relint(\conv(A(U)))$ then $\link_{\mathscr{C}_A}(V-U)$ is homotopy equivalent to a $(\operatorname{rank}(U)-1)$-dimensional sphere. In light of Corollary \ref{crl:0-in-conv-final-description}, for this we just have to show that for $U':=U$ the second half of the join in point \ref{prop:htpy-type-zero-avoiding:boundary} is empty. However, $U'=U$ implies that $B^+\langle U-U'\rangle$ is already empty.
    \end{enumerate}
\end{example}
The proof of Proposition \ref{prop:htpy-type-zero-avoiding} depends on two things. On one hand, we need the Nerve Theorem below (see e.g. \cite[Corollary 4G.3]{Hatcher2002}), which we will repeatedly apply to so-called good open coverings of subsets of the sphere $S^{d-1}$ in order to replace them with simplicial complexes, which we in turn replace by different coverings of different subsets, until we get the desired statements of points \ref{prop:htpy-type-zero-avoiding:nerve} and \ref{prop:htpy-type-zero-avoiding:boundary}. We note that every subset of the sphere $S^{d-1}$ is metrizable, and that metrizable spaces are paracompact \cite{StoneMetricIsParacompact} (see also \cite{RudinParacompact}), so we will not need to worry about paracompactness when we apply the following theorem.
\begin{theorem}[Nerve Theorem]\label{thm:nerve}
    Let $X$ be a paracompact topological space, and $\mathfrak{U}=(U_i)_{i\in I}$ an open cover of it with nerve complex $\mathscr{N}_\mathfrak{U}$. If $\bigcap_{j\in J}U_j$ is empty or contractible for every finite $\emptyset\neq J\subseteq I$ then $X\simeq \mathscr{N}_\mathfrak{U}$.
\end{theorem}
The other key observation is that in case \ref{prop:htpy-type-zero-avoiding:boundary} we have that \[B^+\langle U\rangle\cap B^+(V-U)=B^+(V-U)-B^+(V-U)\cap D^-(U),\] but $B^+(V-U)\cap D^-(U)\subseteq B^+(V-U)\cap S(U')$, which is of codimension $\operatorname{rank}(U')$ in $B^+(V-U)$, so the description of $\link_{\mathscr{C}_\mathscr{O}}(V-U)$ given in point \ref{prop:htpy-type-zero-avoiding:nerve} is homeomorphic to $\R^{d-1}$ with some subset of a codimension $\operatorname{rank}(U')$ linear subspace deleted, which can be recognized as having a description as a join like stated in the proposition. This argumentation is made precise below.
\begin{proof}[Proof of Proposition \ref{prop:htpy-type-zero-avoiding}]
    Consider $\mathscr{C}_\mathscr{O}$ as the simplicial complex given by $\{U\subseteq V: \mathscr{B}^+(U)\neq\emptyset\}$. In particular, if $\mathscr{B}^+(V-U)=\emptyset$ then $V-U\notin\mathscr{C}_\mathscr{O}$, so $\link_{\mathscr{C}_\mathscr{O}}(V-U)$ is the void complex:
    \[
        \link_{\mathscr{C}_\mathscr{O}}(V-U)=\{U'\in\mathscr{C}_\mathscr{O}: (V-U)\cap U'=\emptyset,(V-U)\cup U'\in\mathscr{C}_\mathscr{O}\},
    \]
    but $(V-U)\cup U'\in\mathscr{C}_\mathscr{O}\implies V-U\in\mathscr{C}_\mathscr{O}$, so the latter set is empty, proving point \ref{prop:htpy-type-zero-avoiding:void}.

    For point \ref{prop:htpy-type-zero-avoiding:nerve}, consider the collection
    \[
        \mathfrak{U}:=(B_u^+\cap B^+(V-U))_{u\in U}
    \]
    of subsets of $B^+(V-U)$. They are by construction open, and they cover the space $\bigcup_{u\in U}(B^+_u\cap B^+(V-U))=B^+\langle U\rangle\cap B^+(V-U)$. If $\emptyset\neq U'\subseteq U$ then $\bigcap_{u'\in U'}(B^+_{u'}\cap B^+(V-U))=B^+(V-(U-U'))$ is empty or contractible (Corollary \ref{crl:disks-in-oms}), meaning that $\mathfrak{U}$ is a good open cover. Note that this intersection is nonempty if and only if $\mathscr{B}^+(V-(U-U'))\neq\emptyset$ (Observation \ref{obs:curly-vs-print}), or equivalently if $V-(U-U')\in\mathscr{C}_\mathscr{O}$ by the description of $\mathscr{C}_\mathscr{O}$, but this holds if and only if $U-U'\notin\mathscr{C}_\mathscr{O}^*$, which is the same as saying $U'\in(\mathscr{C}_\mathscr{O}^*[U])^*$. In particular, to see that the nerve $\mathscr{N}_\mathfrak{U}$ of the cover $\mathfrak{U}$ is just $(\mathscr{C}_\mathscr{O}^*[U])^*=\link_{\mathscr{C}_\mathscr{O}}(V-U)$ it suffices to check that $\emptyset\in\mathscr{N}_{\mathfrak{U}}\iff\emptyset\in\link_{\mathscr{C}_\mathscr{O}}(V-U)$, but nerves are by definition never void and $\mathscr{B}^+(V-U)\neq\emptyset\implies V-U\in\mathscr{C}_\mathscr{O}\implies\emptyset\in\link_{\mathscr{C}_\mathscr{O}}(V-U)$. Knowing all this, the Nerve Theorem (Theorem \ref{thm:nerve}) gives point \ref{prop:htpy-type-zero-avoiding:nerve} of the Proposition.

    Let us continue with case \ref{prop:htpy-type-zero-avoiding:sphere}, where it is assumed that $U=V$ and $\mathscr{D}^+(V)=\{0\}$. $\mathscr{B}^+(V-U)=\mathscr{B}^+(\emptyset)=\mathscr{L}(\mathscr{O})\neq\emptyset$, so we fall under case \ref{prop:htpy-type-zero-avoiding:nerve}: applying it to an essential signed arrangement of pseudospheres $\mathfrak{A}'$ to which $\mathscr{O}$ is associated, which exists by the Topological Representation Theorem (Theorem \ref{thm:top-repr}), we get that $\link_{\mathscr{C}_\mathscr{O}}(V-U)\simeq B^+\langle U\rangle\cap B^+(V-U)$, which using $U=V$ is further equal to $B^+\langle V\rangle=S^{\operatorname{rank}(\mathscr{O})-1}-D^-(V)$. As $\mathfrak{A}'$ is essential, we can use Observation \ref{obs:curly-vs-print} to see that $\mathscr{D}^+(V)=\{0\}\iff\mathscr{D}^-(V)=\{0\}\implies D^-(V)=\emptyset$, meaning that $\mathscr{C}_\mathscr{O}=\link_{\mathscr{C}_\mathscr{O}}(\emptyset)\simeq S^{\operatorname{rank}(\mathscr{O})-1}$.

    In case \ref{prop:htpy-type-zero-avoiding:entire-contr}, where $U=V$ and $\mathscr{D}^+(V)\neq\{0\}$ are assumed, the description $\link_{\mathscr{C}_\mathscr{O}}(\emptyset)=\mathscr{C}_\mathscr{O}=\{U\subseteq V:\mathscr{B}^+(U)\neq\emptyset\}$ is still useful. By construction the composition of all covectors in $\mathscr{D}^+(V)$ is conformal, so denote this well-defined composition by $\Phi\neq0$. We wish to show that $\mathscr{C}_\mathscr{O}=\mathscr{C}_\mathscr{O}[\underline{\Phi}]*\mathscr{C}_\mathscr{O}[V-\underline{\Phi}]$, of which $\mathscr{C}_\mathscr{O}[\underline{\Phi}]$ is a simplex, whose contractibility implies that the entire join, that is $\mathscr{C}_\mathscr{O}$, is contractible as well. The desired statement can be rephrased as 
    \[
    \mathscr{B}^+(U)\neq\emptyset\iff\mathscr{B}^+(U\cap\underline{\Phi})\neq\emptyset\neq\mathscr{B}^+(U-\underline{\Phi}),
    \] 
    of which only the ``$\Leftarrow$'' direction is non-automatic. Nevertheless, if $\Psi\in\mathscr{B}^+(U-\underline{\Phi})$, then $\Phi\circ\Psi\in\mathscr{B}^+(U)$, concluding the proof of point \ref{prop:htpy-type-zero-avoiding:entire-contr}.
    
    Now assume we are in the situation of point \ref{prop:htpy-type-zero-avoiding:non-Farkas}, i.e., $U\subsetneq V$ such that $\mathscr{B}^+(V-U)\neq\emptyset$ and $\mathscr{B}^+(V-U)\cap\mathscr{D}^-(U)=\emptyset$. Then point \ref{prop:htpy-type-zero-avoiding:nerve} says that 
    \[
    \link_{\mathscr{C}_\mathscr{O}}(V-U)\simeq B^+\langle U\rangle\cap B^+(V-U)=B^+(V-U)-B^+(V-U)\cap D^-(U)=B^+(V-U), 
    \]
    where Observation \ref{obs:curly-vs-print} justifies the last equality and guarantees $B^+(V-U)$ to be nonempty, and therefore by Corollary \ref{crl:disks-in-oms} to be a $(d-1)$-dimensional open ball, and as such contractible, proving point \ref{prop:htpy-type-zero-avoiding:non-Farkas} of the Proposition.
    
    It is left to prove point \ref{prop:htpy-type-zero-avoiding:boundary}, namely that if $\mathscr{B}^+(V-U)\cap D^-(U)\neq\emptyset$ for $U\subsetneq V$ then 
    \[
    \link_{\mathscr{C}_\mathscr{O}}(V-U)\simeq S^{\operatorname{rank}(U')-1}*(S(U')\cap B^+\langle U-U'\rangle\cap B^+(V-U))
    \] 
    for any $U'\subseteq U$ with $\mathscr{D}^+(U')=\mathscr{S}(U')$.
    Using point \ref{prop:htpy-type-zero-avoiding:nerve} we know that $\link_{\mathscr{C}_\mathscr{O}}(V-U)\simeq B^+\langle U\rangle\cap B^+(V-U)$, while our assumption guarantees the existence of a covector $\Psi\in \mathscr{B}^+(V-U)\cap \mathscr{D}^-(U)$.

    Point \ref{prop:htpy-type-zero-avoiding:boundary} is clear for $U'=\emptyset$, so assume otherwise. Set $R:=S(U')\cap B^+\langle U-U'\rangle\cap B^+(V-U)$, $\mathscr{R}:=\mathfrak{s}_\mathfrak{A}(R)$, and note that $\mathfrak{U}:=(B^+_{u'}\cap B^+(V-U))_{u'\in U'}\cup(T_\Phi)_{\Phi\in\mathscr{R}}$ is an open cover of $B^+\langle U\rangle\cap B^+(V-U)\simeq\link_{\mathscr{C}_\mathscr{O}}(V-U)$; let us check that it is good. Let $U''\subseteq U'$ and $\mathscr{R}'\subseteq \mathscr{R}$ such that $U''\cup\mathscr{R}'\neq\emptyset$, and consider
    \[
        \bigcap_{u''\in U''}(B^+_{u''}\cap B^+(V-U))\cap\bigcap_{\Phi\in\mathscr{R}'}T_\Phi
        =B^+(U'')\cap B^+(V-U)\cap\bigcap_{\Phi\in \mathscr{R}'}T_\Phi.
    \]
    In light of Remark \ref{rmk:reorientation} this space is empty or contractible, making the cover into a good open cover. Now consider the covers
    \begin{align*}
        \mathfrak{U}_{U'}:=(B^+_{u'})_{u'\in U'}&\text{~of~}B^+\langle U'\rangle\text{~and}\\
        \mathfrak{U}_{\mathscr{R}}:=(S(U')\cap T_\Phi)_{\Phi\in\mathscr{R}}&\text{~of~}S(U')\cap T\langle\mathscr{R}\rangle=R.
    \end{align*}
    The nerve of $\mathfrak{U}_{U'}$ is equal to $\mathscr{C}_{\mathscr{O}|_{U'}}$, which, using point \ref{prop:htpy-type-zero-avoiding:sphere}, the description of $\mathscr{L}(\mathscr{O}|_{U'})$, and the fact that $\mathscr{D}^+(U')=\mathscr{S}(U')$, is homotopy equivalent to $S^{\operatorname{rank}(\mathscr{O}|_{U'})-1}=S^{\operatorname{rank}(U')-1}$. On the other hand, $\mathfrak{U}_\mathscr{R}$ is a good open cover by Lemma \ref{lem:tubular-neighborhoods-in-oms} and Remark \ref{rmk:reorientation}, so its nerve is homotopy equivalent to $R$ by the Nerve Theorem. Thus to show that $\link_{\mathscr{C}_\mathscr{O}}(V-U)\simeq S^{\operatorname{rank}(U')-1}*R$, it suffices to see that the nerve of $\mathfrak{U}$ is equal to the join of the nerves of $\mathfrak{U}_{U'}$ and $\mathfrak{U}_\mathscr{R}$. In other words, we want to show that for $U''\subseteq U'$ and $\mathscr{R}'\subseteq \mathscr{R}$ it is true that $B^+(U'')\cap B^+(V-U)\cap\bigcap_{\Phi\in \mathscr{R}'}T_\Phi$ is nonempty if and only if both $B^+(U'')$ and $S(U')\cap\bigcap_{\Phi\in \mathscr{R}'}T_\Phi$ are nonempty, where each of these three spaces is deemed nonempty if $U''\cup\mathscr{R}'=\emptyset$, $U''=\emptyset$, or $\mathscr{R}'=\emptyset$ respectively.
    
    The ``only if'' direction is easy: the nonemptyness of $B^+(U'')\cap B^+(V-U)\cap\bigcap_{\Phi\in\mathscr{R}'}T_\Phi$ implies the non-emptiness of $B^+(U'')$ and $\bigcap_{\Phi\in\mathscr{R}'}T_\Phi$ (where the latter is interpreted as nonempty for $\mathscr{R}'=\emptyset$), so either $\mathscr{R}'=\emptyset$ or by Lemma \ref{lem:tubular-neighborhoods-in-oms} the composition $\Phi_{\mathscr{R}'}$ of all elements of $\mathscr{R}'$ is conformal and $\bigcap_{\Phi\in \mathscr{R}'}T_\Phi=T_{\Phi_{\mathscr{R}'}}$, which in particular has a nonempty intersection with $S(U')$ by virtue of $\Phi_{\mathscr{R}'}\in \mathscr{S}(U')$. For the ``if'' direction, take covectors $\Phi_{U''}\in\mathfrak{s}_\mathfrak{A}(B^+(U''))$ and $\Phi_{\mathscr{R}'}\in\mathfrak{s}_{\mathfrak{A}}(S(U')\cap\bigcap_{\Phi\in\mathscr{R}'}T_\Phi)$, and form the composition $\Theta:=\Phi_{\mathscr{R}'}\circ\Psi\circ\Phi_{U''}$, with $\Phi_{\mathscr{R}'}$ and $\Phi_{U''}$ omitted if $\mathscr{R}'=\emptyset$ or $U''=\emptyset$ respectively. 
    $\mathscr{D}^+(U')=\mathscr{S}(U')$ implies $\mathscr{D}^-(U)=\mathscr{S}(U')$, and in particular 
    \[
    \Psi\in\mathscr{B}^+(V-U)\cap\mathscr{D}^-(U)=\mathscr{B}^+(V-U)\cap\mathscr{D}^-(U-U')\cap\mathscr{S}(U'). 
    \]
    It is then easy to verify using Observation \ref{obs:curly-vs-print} that $\mathfrak{s}_\mathfrak{A}^{-1}(\{\Theta\})\subseteq B^+(V-U)$, if $\mathscr{R}'\neq\emptyset$ then $\mathfrak{s}_\mathfrak{A}^{-1}(\{\Theta\})\subseteq\bigcap_{\Phi\in\mathscr{R}'}T_\Phi$, if $U''\neq\emptyset$ then $\mathfrak{s}_\mathfrak{A}^{-1}(\{\Theta\})\subseteq B^+(U'')$, and by virtue of $\Theta\in\mathscr{L}(\mathscr{O})-\{0\}$ (we already had $\Psi\neq 0$) it also holds that $\mathfrak{s}_\mathfrak{A}^{-1}(\{\Theta\})\neq\emptyset$, which together mean that $\emptyset\neq\mathfrak{s}_\mathfrak{A}^{-1}(\{\Theta\})\subseteq B^+(U'')\cap B^+(V-U)\cap\bigcap_{\Phi\in\mathscr{R}'}T_\Phi$ is nonempty as well. 
\end{proof}

If we are interested in the homology groups $\tilde{H}_n(\link_{\mathscr{C}_\mathscr{O}}(V-U))$ rather than the homotopy type of the space $\link_{\mathscr{C}_\mathscr{O}}(V-U)$ itself, then we can say more than Proposition \ref{prop:htpy-type-zero-avoiding}:
\begin{corollary}\label{crl:homology-of-zero-avoiding}
    Let $\mathscr{O}$ be an oriented matroid of rank $k$ on set of elements $V$, $U\subsetneq V$, and consider all reduced homology groups to have coefficients in a fixed field $\mathbb{F}$. Then:
    \begin{enumerate}[label=\normalfont{(\arabic*)}]
        \item\label{crl:homology-of-zero-avoiding:complex-sphere} If $\mathscr{D}^+(V)=\{0\}$ then $\tilde{H}_n(\mathscr{C}_\mathscr{O})=0$ for $n\neq k-1$ and $\tilde{H}_{k-1}(\mathscr{C}_\mathscr{O})\cong\mathbb{F}$.
        \item\label{crl:homology-of-zero-avoiding:complex-pt} If $\mathscr{D}^+(V)\neq\{0\}$ then $\tilde{H}_n(\mathscr{C}_\mathscr{O})=0$ for all $n$. 
        \item\label{crl:homology-of-zero-avoiding:link-contr} If $\mathscr{B}^+(V-U)\cap\mathscr{D}^-(U)=\emptyset$ then $\tilde{H}_n(\link_{\mathscr{C}_\mathscr{O}}(V-U))=0$ for all $n$.
        \item\label{crl:homology-of-zero-avoiding:link-nontrivial} If $\mathscr{B}^+(V-U)\cap\mathscr{D}^-(U)\neq\emptyset$ then $\tilde{H}_n(\link_{\mathscr{C}_\mathscr{O}}(V-U))=0$ in any of the following cases:
        \begin{enumerate}[label=\normalfont{(4\alph*)}]
            \item\label{crl:homology-of-zero-avoiding:link-nontrivial:upper-bound} $n\leq \operatorname{rank}(U')-2$ for any $U'\subseteq U$ with $\mathscr{D}^+(U')=\mathscr{S}(U')$.
            \item\label{crl:homology-of-zero-avoiding:link-nontrivial:lower-bound} $n\geq k-1$,
            \item\label{crl:homology-of-zero-avoiding:link-nontrivial:pass-to-boundary} $n=k-2$ and $\mathscr{D}^-(U)\nsubseteq\mathscr{B}^+(V-U)\cup\{0\}$
            \item\label{crl:homology-of-zero-avoiding:link-nontrivial:sphere} $n\neq k-2$ and $\mathscr{D}^-(U)\subseteq\mathscr{B}^+(V-U)\cup\{0\}$. In this last case $\tilde{H}_{k-2}(\link_{\mathscr{C}_\mathscr{O}}(V-U))\cong \mathbb{F}$.
        \end{enumerate}
    \end{enumerate}
\end{corollary}
\begin{proof}
    We argue separately about each point, always using Proposition \ref{prop:htpy-type-zero-avoiding}.
    \begin{itemize}
        \item[\ref{crl:homology-of-zero-avoiding:complex-sphere}] Point \ref{prop:htpy-type-zero-avoiding:sphere} of Proposition \ref{prop:htpy-type-zero-avoiding} applies, meaning that $\mathscr{C}_\mathscr{O}\simeq S^{k-1}$, so $\tilde{H}_n(\mathscr{C}_\mathscr{O})\cong\tilde{H}_n(S^{k-1})$, which is $0$ for $n\neq k-1$ and $\mathbb{F}$ for $n=k-1$.
        \item[\ref{crl:homology-of-zero-avoiding:complex-pt}] Point \ref{prop:htpy-type-zero-avoiding:entire-contr} of Proposition \ref{prop:htpy-type-zero-avoiding} applies, meaning that $\mathscr{C}_\mathscr{O}$ is contractible, and in particular all of its reduced homology groups vanish.
        \item[\ref{crl:homology-of-zero-avoiding:link-contr}] If $\mathscr{B}^+(V-U)=\emptyset$ too then point \ref{prop:htpy-type-zero-avoiding:void} of Proposition \ref{prop:htpy-type-zero-avoiding} applies, meaning that $\link_{\mathscr{C}_\mathscr{O}}(V-U)$ is the void complex, so by convention $\tilde{H}_n(\link_{\mathscr{C}_\mathscr{O}}(V-U);\mathbb{F})=0$ for all $n$ and $\mathbb{F}$. If on the other hand $\mathscr{B}^+(V-U)\neq\emptyset$, then point \ref{prop:htpy-type-zero-avoiding:non-Farkas} of Proposition \ref{prop:htpy-type-zero-avoiding} applies, saying that $\link_{\mathscr{C}_\mathscr{O}}(V-U)$ is contractible, so in particular all of its reduced homology groups vanish.
    \end{itemize}    
    Let $\mathfrak{A}$ be an essential signed arrangement of pseudospheres in $S^{k-1}$ whose associated oriented matroid is $\mathscr{O}$ (such $\mathfrak{A}$ exists by the Topological Representation Theorem). In case \ref{crl:homology-of-zero-avoiding:link-nontrivial}, points \ref{prop:htpy-type-zero-avoiding:nerve} and \ref{prop:htpy-type-zero-avoiding:boundary} of Proposition \ref{prop:htpy-type-zero-avoiding} apply, so $\link_{\mathscr{C}_\mathscr{O}}(V-U)$ is homotopy equivalent to both of 
    \[
        B^+\langle U\rangle\cap B^+(V-U)\;\;\;\text{~and~}\;\;\;
        S^{\operatorname{rank}(U')-1}*\left(S(U')\cap B^+\langle U-U'\rangle\cap B^+(V-U)\right)
    \]
    for any $U'\subseteq U$ with $\mathscr{D}^+(U')=\mathscr{S}(U')$, where by convention $S^{-1}=\emptyset$.
    \begin{itemize}
        \item[\ref{crl:homology-of-zero-avoiding:link-nontrivial:upper-bound}] $S(U')\cap B^+\langle U-U'\rangle\cap B^+(V-U)$, like any topological space, is at leat homologically $(-2)$-connected, while $S^{\operatorname{rank}(U')-1}$ is homologically $(\operatorname{rank}(U')-2)$-connected, so the K\"unneth formula tells us that $S^{\operatorname{rank}(U')-1}*\left(S(U')\cap B^+\langle U-U'\rangle\cap B^+(V-U)\right)$ is homologically $(\operatorname{rank}(U')-2)$-connected, and thus so is $\link_{\mathscr{C}_\mathscr{O}}(V-U)$.
        \item[\ref{crl:homology-of-zero-avoiding:link-nontrivial:lower-bound}] By $V-U\neq\emptyset$ we have  that $B^+\langle U\rangle\cap B^+(V-U)\subsetneq S^{k-1}$ is open and hence a non-compact (or empty) $(k-1)$-manifold, which by Poincar\'e duality has no homology in degrees $k-1$ and higher. Note that $k\neq 0$ as we are in case \ref{crl:homology-of-zero-avoiding:link-nontrivial}. Consequently, $\link_{\mathscr{C}_\mathscr{O}}(V-U)$ also can not have non-trivial homology in degrees $k-1$ and higher.
    \end{itemize}
    By assumption $\mathscr{B}^+(V-U)\neq\emptyset$ in the last two cases, so by Corollary \ref{crl:disks-in-oms} the space $D^+(V-U)$ is homeomorphic to a $(k-1)$-dimensional disk with interior $B^+(V-U)$ and boundary $\partial D^+(V-U)$, so in particular the quotient space $D^+(V-U)/\partial D^+(V-U)$ is homeomorphic to a $(k-1)$-dimensional sphere. According to Observation \ref{obs:pseudospheres-regular-cw}, $\mathfrak{A}$ induces the structure of a finite regular CW complex on $D^+(V-U)$, and $D^+(V-U)\cap D^-(U)$ and $(D^+(V-U)\cap D^-(U))\cup\partial D^+(V-U)$ are subcomplexes of this. The quotient of a CW complex by a subcomplex can be viewed as a CW complex itself, so 
    \[
        N:=((D^+(V-U)\cap D^-(U))\cup\partial D^+(V-U))/\partial D^+(V-U)
    \] 
    is a CW complex too, and so its (reduced) Čech and singular cohomology theories agree \cite[Chapter VI.8]{Bredon2010}. Apply the following version of Alexander Duality, taken from \cite[Corollary VI.8.7]{Bredon2010}:
    \begin{theorem}
        If $\emptyset\neq N\subseteq S^d$ is closed, then $\tilde{H}_n(S^d-N;G)\cong\tilde{\check{H}}^{d-n-1}(N;G)$ for any abelian group $G$, where $\tilde{\check{H}}(-;G)$ denotes the reduced Čech cohomology group with coefficients in $G$.
    \end{theorem}
    Our $N$ is a compact subspace of the Hausdorff space $D^+(V-U)/\partial D^+(V-U)\cong S^{k-1}$ and hence is closed, while $\mathscr{B}^+(V-U)\cap\mathscr{D}^-(U)\neq\emptyset$ guarantees $N$ to be nonempty. Hence we can apply the theorem above to get the following isomorphism of reduced (co)homology groups with coefficients in $\mathbb{F}$:
    \begin{multline*}
        \tilde{H}_n(\link_{\mathscr{C}_\mathscr{O}}(V-U))
        \cong\tilde{H}_n(B^+(V-U)-B^+(V-U)\cap D^-(U))
        \cong\\\cong\tilde{H}_n(D^+(V-U)/\partial D^+(V-U)-N)
        \cong\tilde{\check{H}}^{k-2-n}(N)
        \cong\tilde{H}^{k-2-n}(N)
    \end{multline*}
    Knowing that $B^+(V-U)\cap D^-(U)$ is nonempty and contractible by Corollary \ref{crl:disk-with-partial-boundary-in-oms}, we can control the reduced cohomology groups of $N$. We will use that non-emptiness and contractibility of $B^+(V-U)\cap D^-(U)$ implies that its closure, $D^+(V-U)\cap D^-(U)$, is nonempty and path-connected.
    \begin{itemize}
        \item[\ref{crl:homology-of-zero-avoiding:link-nontrivial:pass-to-boundary}] If $\mathscr{D}^-(U)\nsubseteq\mathscr{B}^+(V-U)\cup\{0\}$ then $D^-(U)\nsubseteq B^+(V-U)$, so the nonempty and path-connected spaces $\partial D^+(V-U)$ and $D^+(V-U)\cap D^-(U)$ have a nonempty intersection, meaning that their union, and hence also $N$, is nonempty and path-connected. Thus $\tilde{H}^m(N)=0$ for $m\leq 0$. Setting $m:=k-2-n$, $\tilde{H}_n(\link_{\mathscr{C}_\mathscr{O}}(V-U))\cong\tilde{H}^{k-2-n}(N)=0$ for $k-2-n\leq 0$, or equivalently $n\geq k-2$.
        \item[\ref{crl:homology-of-zero-avoiding:link-nontrivial:sphere}] If $\mathscr{D}^-(U)\subseteq\mathscr{B}^+(V-U)\cup\{0\}$ then $D^-(U)\subseteq B^+(V-U)$ in an essential arrangement, meaning that $N\cong D^-(U)\sqcup\operatorname{pt}$, but $D^-(U)=B^+(V-U)\cap D^-(U)$ is contractible, so $N$ is homotopy equivalent to $S^0$. In particular, $\tilde{H}^m(N)$ is $0$ for $m\neq 0$ and $\mathbb{F}$ for $m=0$. Then $\tilde{H}_n(\link_{\mathscr{C}_\mathscr{O}}(V-U))\cong\tilde{H}^{k-2-n}(N)=0$ for $n\neq k-2$ and $\mathbb{F}$ for $n=k-2$.
    \end{itemize}
\end{proof}
We already recalled in Section \ref{sec:the-method} that with coefficients in a fixed field $\mathbb{F}$ we have $\tilde{H}^n(X)\cong\tilde{H}_n(X)$ whenever the latter is finite dimensional. Therefore:
\begin{corollary}\label{crl:cohomology-of-zero-avoiding}
    Corollary \ref{crl:homology-of-zero-avoiding} remains true if all reduced homology groups are replaced by reduced cohomology groups with coefficients in the same field.
\end{corollary}
It is easy to read off the following from points \ref{crl:homology-of-zero-avoiding:complex-sphere}, \ref{crl:homology-of-zero-avoiding:complex-pt}, \ref{crl:homology-of-zero-avoiding:link-contr}, and \ref{crl:homology-of-zero-avoiding:link-nontrivial:lower-bound} of Corollaries \ref{crl:homology-of-zero-avoiding} and \ref{crl:cohomology-of-zero-avoiding} above:
\begin{corollary}\label{crl:support-near-leray}
    $\mathscr{C}_\mathscr{O}$ for $\operatorname{rank}(\mathscr{O})\leq d$ is a near-$(d-1)$-Leray simplicial complex.
\end{corollary}
Proving just this requires considerably less effort than the proofs of Proposition \ref{prop:htpy-type-zero-avoiding} and Corollary \ref{crl:homology-of-zero-avoiding} did, as in these points \ref{prop:htpy-type-zero-avoiding:boundary} as well as \ref{crl:homology-of-zero-avoiding:link-nontrivial:upper-bound}, \ref{crl:homology-of-zero-avoiding:link-nontrivial:pass-to-boundary}, and \ref{crl:homology-of-zero-avoiding:link-nontrivial:sphere} caused the most trouble respectively. Nevertheless, we emphasize once more that the more complete understanding of the complex $\mathscr{C}_\mathscr{O}$ provided by the pair of this Proposition and Corollary paves the road for future generalizations and the applications of this method to technically more challenging special cases.

\medskip
\subsection{Vector- and element-avoiding complexes}\label{sec:avoiding-complexes:vector}
\medskip

Recall from Section \ref{sec:the-method} that given a nonempty finite set $V$, an element $v_0\notin V$, and a function $A\colon {V\cup\{v_0\}}\to\R^d$, the \emph{vector-avoiding complex} is the simplicial complex
\[
    \mathscr{C}_{v_0,A}:=\{U\subseteq V:A(v_0)\notin\cone(A(U))\},
\]
while for an oriented matroid $\mathscr{O}$ on set of elements $V\sqcup\{v_0\}$ the \emph{element-avoiding complex} is defined to be
\[
    \mathscr{C}_{v_0,\mathscr{O}}:=\{U\subseteq V: v_0\notin\conv(U)\}.
\]
An equivalent definition of the element-avoiding complex can be obtained by applying Lemma \ref{lem:cone-as-covectors} and Corollary \ref{crl:intersections-in-pseudsphere-arrangements-vs-covectors}:
\[
    \mathscr{C}_{v_0,\mathscr{O}}:=\{U\subseteq V: \mathscr{D}^+(U)\cap\mathscr{B}^-_{v_0}\neq\emptyset\}.
\]
As a consequence of Lemma \ref{lem:0-in-conv-interpreted-in-oms}, if $\mathscr{O}$ is the oriented matroid associated to the function $A$, then $\mathscr{C}_{v_0,A}=\mathscr{C}_{v_0,\mathscr{O}}$. Therefore, in the following we only study the element-avoiding complex $\mathscr{C}_{v_0,\mathscr{O}}$, but the results about it can be translated into the language of convex geometry using Corollary \ref{crl:0-in-conv-final-description}. 

First and foremost, the homotopy types of links $\link_{\mathscr{C}_{v_0,\mathscr{O}}}(V-U)$ in this complex can be characterized using the proposition below. Here, $\mathscr{O}/v_0:=(\mathscr{O}^*|_{V})^*$ is defined in terms of the oriented matroid $\mathscr{O}$ on set of elements $V\sqcup\{v_0\}$. It is easily verified that $\mathscr{L}(\mathscr{O}/v_0)=\{\Phi|_V:\Phi\in\mathscr{L}(\mathscr{O}),\Phi(v_0)=0\}$, $\mathscr{V}(\mathscr{O}/v_0)=\{\Psi|_V:\Psi\in\mathscr{V}(\mathscr{O})\}$. Moreover, $\operatorname{rank}(\mathscr{O}/v_0)=\operatorname{rank}(\mathscr{O})-1$ if $v_0$ is not a loop in $\mathscr{O}$, and otherwise $\mathscr{O}/v_0=\mathscr{O}|_V$ both have the same rank as $\mathscr{O}$.
\begin{proposition}\label{prop:htpy-type-elem-avoiding}
    Let $\mathfrak{A}$ be a signed arrangement of pseudospheres in $S^{d-1}$ indexed by the finite set $V\sqcup\{v_0\}$, and let $\mathscr{O}$ be its associated oriented matroid. The link $\link_{\mathscr{C}_{v_0,\mathscr{O}}}(V-U)=(\mathscr{C}_{v_0,\mathscr{O}}^*[U])^*$ can be desribed as follows:
    \begin{enumerate}[label=\normalfont{(\arabic*)}]
        \item\label{prop:htpy-type-elem-avoiding:void} If $\mathscr{D}^+(V-U)\cap\mathscr{B}^-_{v_0}=\emptyset$ then $\link_{\mathscr{C}_{v_0,\mathscr{O}}}(V-U)$ is the void complex.
        \item\label{prop:htpy-type-elem-avoiding:nerve} If $\mathscr{D}^+(V-U)\cap\mathscr{B}^-_{v_0}\neq\emptyset$ then $\link_{\mathscr{C}_{v_0,\mathscr{O}}}(V-U)\simeq D^+\langle U\rangle\cap D^+(V-U)\cap B^-_{v_0}$.
    \end{enumerate}
    In particular, if we are in case \ref{prop:htpy-type-elem-avoiding:nerve} then $\link_{\mathscr{C}_{v_0,\mathscr{O}}}(V-U)$ is homotopy equivalent to a subset of the open hemisphere $B^-_{v_0}$ of $S^{d-1}$. Moreover, if $v_0$ is not a loop in $\mathscr{O}$ then we have the implications:
    \begin{enumerate}[label=\normalfont{(2.\arabic*)}]
        \item\label{prop:htpy-type-elem-avoiding:nonpointed} If $\mathscr{B}^+(V\cup\{v_0\})=\emptyset$ then $\mathscr{C}_{v_0,\mathscr{O}}=\link_{\mathscr{C}_{v_0,\mathscr{O}}}(\emptyset)$ is contractible.
        \item\label{prop:htpy-type-elem-avoiding:pointed} Otherwise $\mathscr{C}_{v_0,\mathscr{O}}=\mathscr{C}_{\mathscr{O}/v_0}$.
    \end{enumerate}
    Therefore in point \ref{prop:htpy-type-elem-avoiding:pointed} the links in $\mathscr{C}_{v_0,\mathscr{O}}$ are described by Proposition \ref{prop:htpy-type-zero-avoiding}.
\end{proposition}

We again present a few examples to contextualize the proposition above before discussing its proof. First, a few special cases:
\begin{example}\label{ex:htpy-type-elem-avoiding-specials}
    Let $\mathfrak{A}$ be a signed arrangement of pseudospheres in $S^{d-1}$ indexed by $V\sqcup\{v_0\}$ where $V$ is finite and nonempty, let $\mathscr{O}$ be the associated oriented matroid, and $U\subseteq V$. Then:
    \begin{enumerate}[label=\normalfont{(\arabic*)}]
        \item\label{ex:htpy-type-elem-avoiding-specials:full} If $\mathscr{B}^+(V\cup\{v_0\})\neq\emptyset$ and $\mathscr{D}^+(V)\cap\mathscr{S}_{v_0}=\{0\}$, or equivalently $\mathscr{B}^+(V)\neq\emptyset$ and $\mathscr{D}^+(V)\cap\mathscr{D}^-_{v_0}=\{0\}$, then $\mathscr{C}_{v_0,\mathscr{O}}=\link_{\mathscr{C}_{v_0,\mathscr{O}}}(\emptyset)\simeq S^{\operatorname{rank}(\mathscr{O})-2}$, and otherwise $\mathscr{C}_{v_0,\mathscr{O}}=\link_{\mathscr{C}_{v_0,\mathscr{O}}}(\emptyset)$ is contractible under the extra assumption that $v_0$ is not a loop in $\mathscr{O}$. 
        The first assumption indeed implies the second, as if $\Pi\in\mathscr{B}^+(V\cup\{v_0\})$ and $0\neq\Xi\in\mathscr{D}^+(V)\cap\mathscr{D}^-_{v_0}$ then either $0\neq\Xi\in\mathscr{D}^+(V)\cap\mathscr{S}_{v_0}$ already (a contradiction) or $\Theta\in\mathscr{B}^+(V)\cap\mathscr{S}_{v_0}\subseteq\mathscr{D}^+(V)\cap\mathscr{S}_{v_0}$ is non-zero ($\emptyset\neq V\subseteq\Theta^+$) for the $\Theta$ given by axiom \ref{def:om:elim} of oriented matroids applied with $v:=v_0$. 
        On the other hand, the second assumption implies the first, as a $\Phi\in\mathscr{B}^+(V)$ can not have $\Phi(v_0)\in\{-,0\}$ because of $\mathscr{D}^+(V)\cap\mathscr{D}^-_{v_0}=\{0\}$ and $\Phi\neq 0$ ($\emptyset\neq V\subseteq\Phi^+$), so $\Phi\in\mathscr{B}^+(V\cup\{v_0\})$. 
        Finally, the implication holds by point \ref{prop:htpy-type-elem-avoiding:pointed} of Proposition \ref{prop:htpy-type-elem-avoiding} and point \ref{prop:htpy-type-zero-avoiding:sphere} of Proposition \ref{prop:htpy-type-zero-avoiding} applied to $\mathscr{O}/v_0$. 
        The ``otherwise'' clause is covered by either point \ref{prop:htpy-type-elem-avoiding:nonpointed} of Proposition \ref{prop:htpy-type-elem-avoiding} or by point \ref{prop:htpy-type-zero-avoiding:entire-contr} of Proposition \ref{prop:htpy-type-zero-avoiding}.
        \item\label{ex:htpy-type-elem-avoiding-specials:link} If $\mathscr{D}^+(V-U)\cap\mathscr{B}^-_{v_0}\neq\emptyset$ but $\mathscr{D}^+(V-U)\cap\mathscr{B}^-(U\cup\{v_0\})=\emptyset$ then $\link_{\mathscr{C}_{v_0,\mathscr{O}}}(V-U)$ is contractible. This follows from point \ref{prop:htpy-type-elem-avoiding:nerve} of Proposition \ref{prop:htpy-type-elem-avoiding}, as that claims that $\link_{\mathscr{C}_{v_0,\mathscr{O}}}(V-U)\simeq D^+\langle U\rangle\cap D^+(V-U)\cap B^-_{v_0}=(S^{d-1}-B^-(U))\cap D^+(V-U)\cap B^-_{v_0}=D^+(V-U)\cap B^-_{v_0}-D^+(V-U)\cap B^-(U\cup\{v_0\})=D^+(V-U)\cap B^-_{v_0}$ using Observation \ref{obs:curly-vs-print}, but this space is empty or contractible using the end of Remark \ref{rmk:reorientation}, and it is nonempty by Observation \ref{obs:curly-vs-print}.
    \end{enumerate}
\end{example}
\begin{example}
    Let $V$ be a nonempty finite set, $v_0\notin V$, $A\colon {V\sqcup\{v_0\}}\to\R^d$ with $A(v_0)\neq 0$, and $U\subseteq V$. Then:
    \begin{enumerate}[label=\normalfont{(\arabic*)}]
        \item If $0\notin\conv(A(V))$ and $A(v_0)\in\relint(\cone(A(V)))$ then $\mathscr{C}_{v_0,A}=\link_{\mathscr{C}_{v_0,A}}(\emptyset)$ is homotopy equivalent to $S^{\dim(\operatorname{span}(A(V)))-2}$, and otherwise it is contractible. This follows from point \ref{ex:htpy-type-elem-avoiding-specials:full} of Example \ref{ex:htpy-type-elem-avoiding-specials} using Corollary \ref{crl:0-in-conv-final-description} and the observation that $\operatorname{span}(A(V))=\operatorname{span}(A(V\cup\{v_0\}))$ because $A(v_0)\in\cone(A(V))$.
        \item If $A(v_0)\notin\cone(A(V-U))$ and $0\in\conv(A(U\cup\{v_0\}))$ then $\link_{\mathscr{C}_{v_0,A}}(V-U)$ is contractible. This follows from point \ref{ex:htpy-type-elem-avoiding-specials:link} of Example \ref{ex:htpy-type-elem-avoiding-specials} using Corollary \ref{crl:0-in-conv-final-description}.
    \end{enumerate}
\end{example}

The proof of Proposition \ref{prop:htpy-type-elem-avoiding} uses analogous techniques to those employed in the proof of Proposition \ref{prop:htpy-type-zero-avoiding}, but the following, different version of the Nerve Theorem is required now, whose proof is presented in Appendix \ref{sec:nerve-thm-extension}.
\begin{theorem}\label{thm:nerve-closed-cellular}
    Let $X$ be a regular CW complex, $A\subseteq X$ an intersection of a subcomplex and an open union of some open cells of $X$, and $(Z_i)_{i\in I}$ a finite collection of subcomplexes of $X$ with $A\subseteq\bigcup_{i\in I}Z_i$. If for all $\emptyset\neq J\subseteq I$ the space $A\cap\bigcap_{j\in J}Z_j$ is empty or contractible, then $A$ is homotopy equivalent to the nerve $\mathscr{N}_{\mathfrak{Z}}$ of the cover $\mathfrak{Z}:=(A\cap Z_i)_{i\in I}$ of $A$.
\end{theorem}
\begin{proof}[Proof of Proposition \ref{prop:htpy-type-elem-avoiding}]
    Consider $\mathscr{C}_{v_0,\mathscr{O}}$ as the simplicial complex given by $\{U\subseteq V:\mathscr{D}^+(U)\cap\mathscr{B}^-_{v_0}\neq\emptyset\}$. In particular, under the assumption of point \ref{prop:htpy-type-elem-avoiding:void} the set $V-U$ is not a face of $\mathscr{C}_{v_0,\mathscr{O}}$, so its link is the void complex:
    \[
        \link_{\mathscr{C}_{v_0,\mathscr{O}}}(V-U)=\{U'\in\mathscr{C}_{v_0,\mathscr{O}}:(V-U)\cap U'=\emptyset,(V-U)\cup U'\in\mathscr{C}_{v_0,\mathscr{O}}\},
    \]
    but $(V-U)\cup U'\in\mathscr{C}_{v_0,\mathscr{O}}\implies V-U\in\mathscr{C}_{v_0,\mathscr{O}}$, so the latter set is empty, proving point \ref{prop:htpy-type-elem-avoiding:void}.

    For point \ref{prop:htpy-type-elem-avoiding:nerve}, consider the collection
    \[
        \mathfrak{U}:=(D^+_u\cap D^+(V-U)\cap B^-_{v_0})_{u\in U}
    \]
    of subsets of $D^+(V-U)\cap B^-_{v_0}$, which form a cover of $D^+\langle U\rangle\cap D^+(V-U)\cap B^-_{v_0}$. If $\emptyset\neq U'\subseteq U$, then $\bigcap_{u'\in U'}(D^+_{u'}\cap D^+(V-U)\cap B^-_{v_0})=D^+(V-(U-U'))\cap B^-_{v_0}$ is empty or contractible by the end of Remark \ref{rmk:reorientation}. Note that this intersection is nonempty if and only if $\mathscr{D}^+(V-(U-U'))\cap\mathscr{B}^-_{v_0}\neq\emptyset$ (Observation \ref{obs:curly-vs-print}), or equivalently if $V-(U-U')\in\mathscr{C}_{v_0,\mathscr{O}}$ by the description of $\mathscr{C}_{v_0,\mathscr{O}}$, but this holds if and only if $U-U'\notin\mathscr{C}_{v_0,\mathscr{O}}^*$, which is the same as saying $U'\in(\mathscr{C}_{v_0,\mathscr{O}}^*[U])^*$. In particular, to see that the nerve $\mathscr{N}_{\mathfrak{U}}$ of the cover $\mathfrak{U}$ is just $(\mathscr{C}_{v_0,\mathscr{O}}^*[U])^*=\link_{\mathscr{C}_{v_0,\mathscr{O}}}(V-U)$ it suffices to check that $\emptyset\in\mathscr{N}_\mathfrak{U}\iff\emptyset\in\link_{\mathscr{C}_{v_0,\mathscr{O}}}(V-U)$, but nerves are by definition not void and $\mathscr{D}^+(V-U)\cap\mathscr{B}^-_{v_0}\neq\emptyset\implies V-U\in\mathscr{C}_{v_0,\mathscr{O}}\implies\emptyset\in\link_{\mathscr{C}_{v_0,\mathscr{O}}}(V-U)$. Therefore, the only thing needed to prove point \ref{prop:htpy-type-elem-avoiding:nerve} is to verify that the assumptions of the Nerve Theorem (Theorem \ref{thm:nerve-closed-cellular}) hold. 
    Let $X:=S^{d-1}$ with the regular cell decomposition given by $\mathfrak{A}$ (see Observation \ref{obs:pseudospheres-regular-cw}), $A:=D^+\langle U\rangle\cap D^+(V-U)\cap B^-_{v_0}$, and $Z_u:=D^+_u$ for $u\in U$. Then $A\subseteq\bigcup_{u\in U}Z_u$, $Z_u$ is easily seen to be a subcomplex of $X$, $A$ is indeed a union of open cells, and the nerve of the cover $(A\cap Z_u)_{u\in U}$ of $A$ is equal to $\mathscr{N}_\mathfrak{U}$, so Theorem \ref{thm:nerve-closed-cellular} is applicable and gives that $\mathscr{N}_\mathfrak{U}=\mathscr{N}_{(A\cap Z_u)_{u\in U}}\simeq A=D^+\langle U\rangle\cap D^+(V-U)\cap B^-_{v_0}$.

    For point \ref{prop:htpy-type-elem-avoiding:nonpointed} note that when $U:=V$ the condition $\mathscr{D}^+(V-U)\cap\mathscr{B}^-_{v_0}\neq\emptyset$ is automatically satisfied as $v_0$ is not a loop, so 
    \begin{multline*}
    \mathscr{C}_{v_0,\mathscr{O}}=\link_{\mathscr{C}_{v_0,\mathscr{O}}}(\emptyset)\simeq D^+\langle V\rangle\cap D^+(\emptyset)\cap B^-_{v_0}=D^+\langle V\rangle\cap B^-_{v_0}\\ =(S^{d-1}-B^-(V))\cap B^-_{v_0}=B^-_{v_0}-B^-(V\cup\{v_0\}).	
    \end{multline*}
    However, $B^-(V\cup\{v_0\})$ is empty because $\mathscr{B}^-(V\cup\{v_0\})$ is empty (Observation \ref{obs:curly-vs-print}), which holds as $\mathscr{B}^+(V\cup\{v_0\})$ is empty, so in this case $\mathscr{C}_{v_0,\mathscr{O}}\simeq B^-_{v_0}-B^-(V\cup\{v_0\})=B^-_{v_0}$, but this space is homeomorphic to an open hemisphere of $S^{d-1}$ because $v_0$ is not a loop, and hence is contractible.

    It is left to prove point \ref{prop:htpy-type-elem-avoiding:pointed}, namely that $\mathscr{C}_{v_0,\mathscr{O}}=\mathscr{C}_{\mathscr{O}/v_0}$ if $\mathscr{B}^+(V\cup\{v_0\})\neq\emptyset$.
    After expanding the definitions of both complexes, we just have to show that for any $\emptyset\neq U\subseteq V$ we have that $\mathscr{D}^+(U)\cap\mathscr{B}^-_{v_0}\neq\emptyset$ in $\mathscr{O}$ if and only if $\mathscr{B}^+(U)\neq\emptyset$ in $\mathscr{O}/v_0$, but this latter condition is equivalent to saying that $\mathscr{B}^+(U)\cap\mathscr{S}_{v_0}\neq\emptyset$ in $\mathscr{O}$. Hence from this point on we work in $\mathscr{O}$ and aim to show $\mathscr{D}^+(U)\cap\mathscr{B}^-_{v_0}\neq\emptyset\iff\mathscr{B}^+(U)\cap\mathscr{S}_{v_0}\neq\emptyset$.
    By assumption we have a $\Phi\in\mathscr{B}^+(V\cup\{v_0\})$. If $\Psi\in\mathscr{B}^+(U)\cap\mathscr{S}_{v_0}$ then $\Psi\circ(-\Phi)\in\mathscr{B}^+(U)\cap\mathscr{B}^-_{v_0}\subseteq\mathscr{D}^+(U)\cap\mathscr{B}^-_{v_0}$. For the other direction, first take a $\Psi_0\in\mathscr{D}^+(U)\cap\mathscr{B}^-_{v_0}$, then pick a $\Theta$ given by axiom \ref{def:om:elim} for $\Pi:=\Phi$, $\Xi:=\Psi_0\circ\Phi$, and $v:=v_0$. For $u\in U$ both $\Pi(u)=+$ and $\Xi(u)=\Psi_0(u)\circ+=+$, so $\Theta(u)=+$ too, but also $\Theta(v_0)=0$, so $\Theta\in\mathscr{B}^+(U)\cap\mathscr{S}_{v_0}$.
\end{proof}

We can summarize the consequences of Proposition \ref{prop:htpy-type-elem-avoiding} regarding the reduced homology groups of links $\link_{\mathscr{C}_{v_0,\mathscr{O}}}(V-U)$ as follows:
\begin{corollary}\label{crl:homology-of-elem-avoiding}
    Let $\mathscr{O}$ be an oriented matroid of rank $k$ on set of elements $V\sqcup\{v_0\}$ where $V$ is finite and nonempty, $U\subseteq V$, and consider all reduced homology groups to have coefficients in a fixed field $\mathbb{F}$. Then:
    \begin{enumerate}[label=\normalfont{(\arabic*)}]
        \item\label{crl:homology-of-elem-avoiding:sphere} If $\mathscr{B}^+(V)\neq\emptyset$ and $\mathscr{D}^+(V)\cap\mathscr{D}^-_{v_0}=\{0\}$ then $\tilde{H}_n(\mathscr{C}_{v_0,\mathscr{O}})=0$ for $n\neq k-2$ and $\tilde{H}_{k-2}(\mathscr{C}_{v_0,\mathscr{O}})\cong\mathbb{F}$.
        \item\label{crl:homology-of-elem-avoiding:entire-contr} If $\mathscr{B}^+(V)=\emptyset$ or $\mathscr{D}^+(V)\cap\mathscr{D}^-_{v_0}\neq\{0\}$ then $\tilde{H}_n(\mathscr{C}_{v_0,\mathscr{O}})=0$ for all $n$.
        \item\label{crl:homology-of-elem-avoiding:void} If $\mathscr{D}^+(V-U)\cap\mathscr{B}^-_{v_0}=\emptyset$ then $\tilde{H}_n(\link_{\mathscr{C}_{v_0,\mathscr{O}}}(V-U))=0$ for all $n$.
        \item\label{crl:homology-of-elem-avoiding:rest} If $\mathscr{D}^+(V-U)\cap\mathscr{B}^-_{v_0}\neq\emptyset$ then $\tilde{H}_n(\link_{\mathscr{C}_{v_0,\mathscr{O}}}(V-U))=0$ for $n\geq k-1$.
    \end{enumerate}
\end{corollary}
\begin{proof}
    We argue separately about each point, always using Proposition \ref{prop:htpy-type-elem-avoiding} or Example \ref{ex:htpy-type-elem-avoiding-specials}.
    \begin{itemize}
        \item[\ref{crl:homology-of-elem-avoiding:sphere}] Point \ref{ex:htpy-type-elem-avoiding-specials:full} of Example \ref{ex:htpy-type-elem-avoiding-specials} applies, so $\mathscr{C}_{v_0,\mathscr{O}}$ is homotopy equivalent to $S^{k-2}$, and thus $\tilde{H}_n(\mathscr{C}_{v_0,\mathscr{O}})\cong\tilde{H}_n(S^{k-2})$, which is $0$ for $n\neq k-2$ and isomorphic to $\mathbb{F}$ for $n=k-2$.
        \item[\ref{crl:homology-of-elem-avoiding:entire-contr}] If $v_0$ is a loop then point \ref{prop:htpy-type-elem-avoiding:void} of Proposition \ref{prop:htpy-type-elem-avoiding} applies with $U:=V$, and therefore all reduced homology groups of the void complex $\mathscr{C}_{v_0,\mathscr{O}}=\link_{\mathscr{C}_{v_0,\mathscr{O}}}(V-U)$ vanish. If $v_0$ is not a loop, the ``otherwise'' clause of point \ref{ex:htpy-type-elem-avoiding-specials:full} of Example \ref{ex:htpy-type-elem-avoiding-specials} applies, so $\mathscr{C}_{v_0,\mathscr{O}}$ is contractible and hence its reduced homology groups all still vanish.
        \item[\ref{crl:homology-of-elem-avoiding:void}] Point \ref{prop:htpy-type-elem-avoiding:void} of Proposition \ref{prop:htpy-type-elem-avoiding} applies and so $\link_{\mathscr{C}_{v_0,\mathscr{O}}}(V-U)$ is the void complex, whose reduced homology groups all vanish by convention.
        \item[\ref{crl:homology-of-elem-avoiding:rest}] Point \ref{prop:htpy-type-elem-avoiding:nerve} of Proposition \ref{prop:htpy-type-elem-avoiding} applies, meaning that $\link_{\mathscr{C}_{v_0,\mathscr{O}}}(V-U)$ is homotopy equivalent to a subset of the space $B^-_{v_0}$, interpreted in an essential signed arrangement $\mathfrak{A}$ of pseudospheres indexed by $V\sqcup\{v_0\}$ to which $\mathscr{O}$ is associated. But an essential $\mathfrak{A}$ gives a regular CW decomposition of the sphere $S^{k-1}$ (see Observation \ref{obs:pseudospheres-regular-cw}), in which the subset homotopy equivalent to $\link_{\mathscr{C}_{v_0,\mathscr{O}}}(V-U)$ is cellular in the sense of Appendix \ref{sec:nerve-thm-extension}. In particular, we can apply Lemma \ref{lem:operator-T} to see that this subset is homotopy equivalent to a proper open subset of $S^{k-1}$, i.e., a non-compact (or empty) $(k-1)$-manifold, which by Poincar\'e duality can not have homology in degrees $k-1$ or higher.   
    \end{itemize}
\end{proof}
As we recalled in Section \ref{sec:the-method}, with coefficients in a fixed field $\mathbb{F}$ and for any space $X$ we have $\tilde{H}^n(X)\cong\tilde{H}_n(X)$ whenever the latter is finite dimensional. Thus we get:
\begin{corollary}\label{crl:cohomology-of-elem-avoiding}
    Corollary \ref{crl:homology-of-elem-avoiding} remains true if all reduced homology groups are replaced by reduced cohomology groups with coefficients in the same field.
\end{corollary}
It is easy to read off the following from Corollaries \ref{crl:homology-of-elem-avoiding} and \ref{crl:cohomology-of-elem-avoiding}, and to extend it to the case when $V=\emptyset$:
\begin{corollary}\label{crl:element-avoiding-leray}
    $\mathscr{C}_{v_0,\mathscr{O}}$ for $\operatorname{rank}(\mathscr{O})\leq d$ is a $(d-1)$-Leray simplicial complex.
\end{corollary}

\bigskip
\section{Proofs of the colorful Carath\'eodory theorems}\label{sec:cc-proofs}
\bigskip

Here we prove the colorful Carath\'eodory theorems \ref{thm:caratheodory:oriented-matroid}, \ref{thm:caratheodory:constrained}, \ref{thm:caratheodory:om-cone-matroid}, and \ref{thm:caratheodory:om-cone-pair}. As noted in Observation \ref{obs:implications-between-caratheodorys}, Theorems \ref{thm:caratheodory:original} through \ref{thm:caratheodory:om-cone-pair} all follow from these by simple substitutions, except for Theorem \ref{thm:caratheodory:matroid}. In each proof we pick the appropriate complex out of the zero-avoiding, support, vector-avoiding, and element-avoiding complexes to be $\mathscr{C}$, and pick the simplicial complex $\mathscr{K}$ so that the conclusion of the colorful Carath\'eodory theorem becomes equivalent to the statement $\mathscr{K}\nsubseteq\mathscr{C}$. To verify this latter statement, we wish to use Theorem \ref{thm:holmsen-repackaged}. Thankfully, the assumptions of all colorful Carath\'eodory theorems are chosen so that $\mathscr{C}$ is indeed near-$(d-1)$-Leray or $(d-1)$-Leray simplicial complex by Corollary \ref{crl:support-near-leray} or \ref{crl:element-avoiding-leray}: the oriented matroid associated to a function $A$ with codomain $\R^d$ is necessarily of rank at most $d$, so
\begin{itemize}
    \item the zero-avoiding complex of a function $A$ with codomain $\R^d$ is near-$(d-1)$-Leray by Corollary \ref{crl:support-near-leray},
    \item the support complex of an oriented matroid of rank at most $d$ is near-$(d-1)$-Leray by Corollary \ref{crl:support-near-leray}.
    \item the vector-avoiding complex of a function $A$ with codomain $\R^d$ is $(d-1)$-Leray by Corollary \ref{crl:element-avoiding-leray}, and
    \item the element-avoiding complex of an oriented matroid of rank at most $d$ is $(d-1)$-Leray by Corollary \ref{crl:element-avoiding-leray}.
\end{itemize}
However, the vertex set of $\mathscr{K}$ is sometimes not $V$, like demanded by the Theorem, but rather some subset $V'\neq\emptyset$ of it. This can in all cases be remedied by applying Theorem \ref{thm:holmsen-repackaged} to $\mathscr{C}[V']$ instead, as this is just the zero-avoiding, support, vector-avoiding, or element-avoiding complex associated to $A|_{V'}\colon{V'}\to\R^d$, $\mathscr{O}|_{V'}$, $A|_{V'\sqcup\{v_0\}}\colon{V'\sqcup\{v_0\}}\to\R^d$, or $\mathscr{O}|_{V'\sqcup\{v_0\}}$ respectively, so it is also near-$(d-1)$-Leray or $(d-1)$-Leray.

It will be a useful observation throughout that the join of $k$ many nonempty spaces is homologically $(k-2)$-connected: each space is homologically $(-1)$-connected, so the K\"unneth formula computes that the join is homologically $k(-1)+(k-1)2=(k-2)$-connected.

First, we paraphrase the proof of Theorem \ref{thm:caratheodory:oriented-matroid} presented in \cite{Holmsen2016}.
\begin{proof}[Proof of Theorem \ref{thm:caratheodory:oriented-matroid}]
    $\mathscr{M}$ is of rank at least $d+1\geq 1$ so its set $V'$ of vertices is necessarily nonempty. Let $\mathscr{C}:=\mathscr{C}_{\mathscr{O}|_{V'}}$ be the support complex of $\mathscr{O}|_{V'}$, and let $\mathscr{K}:=\mathscr{M}[V']$. Then it suffices to prove that $\mathscr{K}\nsubseteq\mathscr{C}$.

    $\mathscr{C}$ is near-$(d-1)$-Leray, so to apply Theorem \ref{thm:holmsen-repackaged} it suffices to check that $\mathscr{K}$ is homologically $(d-1)$-connected, and that $\mathscr{K}[V'-\sigma]$ is homologically $(d-2)$-connected for all $\sigma\in\mathscr{C}$. First, note that $\mathscr{K}[U]=\mathscr{M}[U]$ is homologically $(\rho_\mathscr{M}(U)-2)$-connected for all $U\subseteq V'$ by \cite[Lemma 3.2]{Holmsen2016}, and so $\mathscr{K}$ is homologically $\operatorname{rank}(\mathscr{M}[V'])-2=\operatorname{rank}(\mathscr{M})-2\geq (d-1)$-connected. On the other hand, $\sigma\in\mathscr{C}$ must have $\rho_{\mathscr{M}[V']}(V'-\sigma)\geq d$, as otherwise we would also have $\rho_\mathscr{M}(V-\sigma)<d$ so by assumption $\sigma$ would contain a positive circuit of $\mathscr{O}$ and hence would not be a face of $\mathscr{C}$. But then $\mathscr{K}[V'-\sigma]=\mathscr{M}[V'-\sigma]$ is homologically $\rho_{\mathscr{M}[V']}(V'-\sigma)-2\geq (d-2)$ connected.
\end{proof}
Next we prove the version of Theorem \ref{thm:caratheodory:constrained} where instead of a function $A\colon V\to\R^d$ we consider an oriented matroid $\mathscr{O}$ on set of elements $V$ or rank at most $d$, and replace ``$0\in\conv(A(U))$'' with ``$U$ contains a positive circuit of $\mathscr{O}$'' everywhere. This is a generalization of the aforementioned theorem (see Corollary \ref{crl:0-in-conv-final-description}), so it suffices to consider only this version. The proof completely mirrors the original one found in \cite{Blagojevic2025}, just written in the language of this paper.
\begin{proof}[Proof of Theorem \ref{thm:caratheodory:constrained}]
    By possibly replacing $\mathscr{O}$ with $\mathscr{O}|_{U_1\sqcup\cdots\sqcup U_r}$ we may assume that $V=U_1\sqcup\cdots\sqcup U_r$. Let $\mathscr{C}:=\mathscr{C}_{\mathscr{O}}$ be the support complex of $\mathscr{O}$, and note that it suffices to prove $\mathscr{K}\nsubseteq\mathscr{C}$. $U_1,\ldots,U_r$ are nonempty because they all contain positive circuits of $\mathscr{O}$, so by $r\geq 2$ the complex $\mathscr{K}$ has the nonempty vertex set $U_1\sqcup U_2\sqcup U_3\sqcup\cdots\sqcup U_r=V$.

    $\mathscr{C}$ is near-$(d-1)$-Leray, so to apply Theorem \ref{thm:holmsen-repackaged} it suffices to check that $\mathscr{K}$ is homologically $(d-1)$-connected, and that $\mathscr{K}[V-\sigma]$ is homologically $(d-2)$-connected for all $\sigma\in\mathscr{C}$. First, observe that $\mathscr{K}$ is a join of the nonempty and path-connected (and therefore homologically $0$-connected) space $\mathscr{K}_{12}$ and of the $r-2$ many nonempty (and therefore homologically $(-1)$-connected) spaces $U_3,\ldots,U_r$. Therefore, the K\"unneth formula says that $\mathscr{K}$ is homologically $0+(r-2)(-1)+(r-2)2=r-2\geq(d-1)$-connected. On the other hand, $U_1,\ldots,U_r\notin\mathscr{C}$ because they all contain positive circuits of $\mathscr{O}$, so no $\sigma\in\mathscr{C}$ may contain them, and consequently $(V-\sigma)\cap U_i\neq\emptyset$ for $1\leq i\leq r$. Therefore 
    \[
    \mathscr{K}[V-\sigma]=(\mathscr{K}_{12}*U_3*\cdots*U_r)[V-\sigma]=\mathscr{K}_{12}[(U_1\cup U_2)\cap (V-\sigma)]*U_3[U_3\cap(V-\sigma)]*\cdots*U_r[U_r\cap(V-\sigma)]
    \] 
    is the join of $r-1$ many nonempty spaces and therefore homologically $(r-1)-2\geq (d-2)$-connected.
\end{proof}
This theorem (Theorem \ref{thm:caratheodory:constrained}) has no cone generalization in the sense of point \ref{obs:no-cone-caratheodory:constrained} of Observation \ref{obs:no-cone-caratheodory} with $r\geq d$.
The proof above can not be modified to yield a proof of a cone version of Theorem \ref{thm:caratheodory:constrained}, because the argument there would still only compute $\mathscr{K}[V-\sigma]$ to be homologically $(r-3)$-connected, but now $r-3\geq d-3\ngeq d-2$.

Next, we present a proof of Theorem \ref{thm:caratheodory:om-cone-matroid}, following the same general structure as previously. However, we would also like to note that this presentation is a somewhat more verbose way of saying that one should simply apply \cite[Theorem 1.6]{KalaiMeshulam2005} to the $(d-1)$-Leray simplicial complex $\mathscr{C}_{v_0,\mathscr{O}}$ (see Corollary \ref{crl:element-avoiding-leray}).
\begin{proof}[Proof of Theorem \ref{thm:caratheodory:om-cone-matroid}]
    If $\mathscr{M}$ is of rank $0$ then $v_0$ must be a loop in $\mathscr{O}$ because otherwise we would need $v_0\in\conv(\emptyset)$ (which is false for non-loops) as $\rho_\mathscr{M}(V-\emptyset)=0<\operatorname{rank}(\mathscr{O})\leq d$. But then $v_0\in\conv(\emptyset)$ and $\emptyset$ is an independent set of $\mathscr{M}$, yielding the conclusion of the theorem, so assume that $\mathscr{M}$ is of positive rank, and therefore has a nonempty vertex set $V'$.
    Let $\mathscr{C}:=\mathscr{C}_{v_0,\mathscr{O}|_{V'\cup\{v_0\}}}$ be the element-avoiding complex of $\mathscr{O}|_{V'\cup\{v_0\}}$, and let $\mathscr{K}:=\mathscr{M}[V']$. Then it suffices to prove that $\mathscr{K}\nsubseteq\mathscr{C}$.

    $\mathscr{C}$ is $(d-1)$-Leray, so to apply Theorem \ref{thm:holmsen-repackaged} it suffices to check that $\mathscr{K}[V'-\sigma]$ is homologically $(d-2)$-connected for all $\emptyset\subseteq\sigma\in\mathscr{C}$. $\mathscr{K}[U]=\mathscr{M}[U]$ is homologically $(\rho_{\mathscr{M}[V']}(U)-2)$-connected for all $U\subseteq V'$ according to \cite[Lemma 3.2]{Holmsen2016}, and $\rho_{\mathscr{M}[V']}(V'-\sigma)=\rho_\mathscr{M}(V-\sigma)\geq d$ for any $\sigma\in\mathscr{C}$ by assumption, so $\mathscr{K}[V'-\sigma]=\mathscr{M}[V'-\sigma]$ is homologically $\rho_{\mathscr{M}[V']}(V'-\sigma)-2\geq (d-2)$-connected.
\end{proof}

\begin{proof}[Proof of Theorem \ref{thm:caratheodory:om-cone-pair}]
    By possibly replacing $\mathscr{O}$ with $\mathscr{O}|_{U_1\sqcup\cdots\sqcup U_r}$ we may assume that $V=U_1\sqcup\cdots\sqcup U_r$. Let $\mathscr{C}:=\mathscr{C}_{v_0,\mathscr{O}}$ be the element-avoiding complex of $\mathscr{O}$, view $U_1,\ldots,U_r$ as discrete collections of vertices, and form $\mathscr{K}:=U_1*\cdots*U_r$. $U_1,\ldots,U_r$ are nonempty by assumption, so as $r\geq d+1\geq 1$ the complex $\mathscr{K}$ has the nonempty vertex set $U_1\sqcup\cdots\sqcup U_r=V$, any face $\sigma\in\mathscr{K}$ is contained in one of the form $\{u_1,\ldots,u_r\}\in\mathscr{K}$ with $u_1\in U_1,\ldots,u_r\in U_r$, and it suffices to prove $\mathscr{K}\nsubseteq\mathscr{C}$.

    $\mathscr{C}$ is $(d-1)$-Leray, so to apply Theorem \ref{thm:holmsen-repackaged} it suffices to check that $\mathscr{K}[V-\sigma]$ is homologically $(d-2)$-connected for all $\emptyset\subseteq\sigma\in\mathscr{C}$. $v_0\in\bigcap_{1\leq i<j\leq r}\conv(U_i\cup U_j)$ implies that $U_i\cup U_j\notin\mathscr{C}$ whenever $1\leq i<j\leq r$, so no $\sigma\in\mathscr{C}$ can contain any of them, but then $(V-\sigma)\cap(U_i\cup U_j)\neq\emptyset$ for $1\leq i<j\leq r$. Therefore 
    \[
    \mathscr{K}[V-\sigma]=(U_1*\cdots*U_r)[V-\sigma]=U_1[U_1\cap(V-\sigma)]*\cdots*U_r[U_r\cap(V-\sigma)]
    \] is the join of at least $r-1$ many nonempty spaces (if two of these were empty then we would have $(V-\sigma)\cap(U_i\cup U_j)$ for those) and is hence homologically $(r-1)-2\geq (d-2)$-connected.
\end{proof}
And finally, we just have to see that Theorem \ref{thm:caratheodory:3-joined} is indeed true. Again, we phrase the proof in the language of oriented matroids, thereby slightly generalizing the original statement.
\begin{proof}[Proof of Theorem \ref{thm:caratheodory:3-joined}]
    Let $\mathscr{C}:=\mathscr{C}_\mathscr{O}$ be the support complex of the oriented matroid $\mathscr{O}$ of rank $d$ on nonempty set of elements $V=U_1\sqcup U_2\sqcup U_3$. This is a $(d-1)$-Leray simplicial complex, so to apply Theorem \ref{thm:holmsen-repackaged} it suffices to check that $\mathscr{K}$ has vertex set $V$ (this is obvious), that $\mathscr{K}$ is homologically $(d-1)$-connected, and that $\mathscr{K}[V-\sigma]$ is homologically $(d-2)$-connected for all $\sigma\in\mathscr{C}$. 
    
    As before, $U_1,U_2,U_3\notin\mathscr{C}$ because they all contain positive circuits of $\mathscr{O}$ by assumption, so no $\sigma\in\mathscr{C}$ may contain them, and consequently $(V-\sigma)\cap U_i\neq\emptyset$ for $1\leq i\leq 3$. Therefore if $u_i\in U_i$ and $u_j\in U_j$ with $i\neq j$ are vertices of $\mathscr{K}[V-\sigma]$, then there is a $i\neq k\neq j$ and a $u_k\in (V-\sigma)\cap U_k$, and knowing the structure of $\mathscr{K}$, either $u_i$ and $u_j$ are directly connected by an edge of $\mathscr{K}$, or both are connected to $u_k$ via edges. In particular, $u_i$ and $u_j$ are in the same connected component of $\mathscr{K}[V-\sigma]$, from which it easily follows that $\mathscr{K}[V-\sigma]$ is path-connected and therefore homologically $d-2\leq 0$-connected.

    $\mathscr{K}$ is easily seen to be homologically $d-1\leq 1$-connected, either by letting a computer perform a direct simplicial homology computation, or using a Mayer-Vietoris argument. In particular, if one considers the open cover of $\mathscr{K}$ by either deleting $U_3$ or $\mathscr{K}[U_1\cup U_2]$, then the corresponding Mayer-Vietoris sequence lets one compute the first homology group of $\mathscr{K}$ by sandwiching it between a surjection and an injection:
    \[\xymatrix@C=5pt@R=8pt{
            H_1(\mathscr{K}[U_1\cup U_2-\{u_2^1\}])
            \oplus H_1(\mathscr{K}[U_1\cup U_2-\{u_1^2\}])
            \oplus H_1(\mathscr{K}[U_1\cup U_2])
                \ar[d]
        \\  H_1(\{u_3^1\})
            \oplus H_1(\{u_3^2\})
            \oplus H_1(\{u_3^3\})
                \ar@{}[d]|{\oplus}
        \\  H_1(\mathscr{K}[U_1\cup U_2])
                \ar[d]
        \\  H_1(\mathscr{K})
                \ar[d]
        \\  H_0(\mathscr{K}[U_1\cup U_2-\{u_2^1\}])
            \oplus H_0(\mathscr{K}[U_1\cup U_2-\{u_1^2\}])
            \oplus H_0(\mathscr{K}[U_1\cup U_2])
                \ar[d]
        \\  H_0(\{u_3^1\})
            \oplus H_0(\{u_3^2\})
            \oplus H_0(\{u_3^3\})
                \ar@{}[d]|{\oplus}
        \\  H_0(\mathscr{K}[U_1\cup U_2])
    }\]
\end{proof}

\appendix
\section{Nerve Theorems for ugly covers}\label{sec:nerve-thm-extension}

Given a regular CW complex $X$, call a subset $A\subseteq X$ a \emph{cellular subset} if it is a union of some open cells of $X$, such that if $\sigma$ and $\sigma''$ are open cells in $A$, and $\sigma'$ is an open cell of $X$ with $\sigma\subseteq\operatorname{Cl}(\sigma')$ and $\sigma'\subseteq\operatorname{Cl}(\sigma'')$, then $\sigma'\subseteq A$. Here $\operatorname{Cl}(\tau)$ stands for the closure of $\tau$ in $X$. Equivalently, a cellular subset is an intersection of a subcomplex and an open union of open cells of $X$. In particular, if a union of open cells of a cellular subset is closed or open within this subset, then it is itself a cellular subset, in which case we call it a closed (respectively open) cellular subset. Say that a collection $(A_i)_{i\in I}$ of cellular subsets of $X$ is \emph{ugly} if whenever $i\neq j\in I$, $\sigma_i$ an open cell of $A_i$, $\sigma_j$ an open cell of $A_j$, and $\sigma_i$ intersects the closure of $\sigma_j$ we have $\sigma_i\subseteq A_j$ or $\sigma_j\subseteq A_i$. Collections of open cellular subsets and collections of closed cellular subsets of a fixed cellular subset $X'$ are always ugly. If $A$ is a cellular subset and $Z$ is a subcomplex of $X$, then $A\cap Z$ is closed in $A$, so Theorem \ref{thm:nerve-closed-cellular} is a consequence of the more general Theorem \ref{thm:nerve-ugly-cellular} below:
\begin{theorem}\label{thm:nerve-ugly-cellular}
    Let $X$ be a regular CW complex and $(A_i)_{i\in I}$ an ugly collection of cellular subspaces of it. If $\bigcap_{j\in J}A_j$ is empty or contractible for all $\emptyset\neq J\subseteq I$ then $\bigcup_{i\in I}A_i$ is homotopy equivalent to the nerve of the cover $(A_i)_{i\in I}$.
\end{theorem}
We present the proof in relatively broad strokes.
\begin{proof}
    If $Y$ is a refinement of $X$, then cellular subsets of $X$ are also cellular subsets of $Y$, and if a collection of cellular subsets of $X$ is ugly, then the same collection of cellular subsets of $Y$ is still ugly. The barycentric subdivision of any regular CW complex is a simplicial complex, which is a refinement of the CW complex, so without loss of generality we may assume that $X$ is a simplicial complex.

    The theorem hinges on the following lemma, introducing the operator $T\langle -\rangle$. As $A\subseteq T\langle A\rangle$ is a strong deformation retract, the two spaces are homotopy equivalent, and so the Nerve Theorem (Theorem \ref{thm:nerve}) can be applied to the cover $(T\langle A_i\rangle)_{i\in I}$ of $T\langle\bigcup_{i\in I}A_i\rangle$ to prove our theorem.
\end{proof}
\begin{lemma}\label{lem:operator-T}
    There is an operator $T\langle-\rangle$ which assigns to every cellular subset $A$ of $X$ an open subset $T\langle A\rangle$ of $X$, such that $A\subseteq T\langle A\rangle$ is a strong deformation retract, $\bigcup_{i\in I}T\langle A_i\rangle=T\langle\bigcup_{i\in I}A_i\rangle$, and for a finite ugly collection $(A_i)_{i\in I}$ we have $\bigcap_{i\in I}T\langle A_i\rangle=T\langle\bigcap_{i\in I}A_i\rangle$.
\end{lemma}
\begin{proof}
    Just like before, we may assume that $X$ is actually a simplicial complex. 
    Afterwards, replace $X$ with its barycentric subdivision $\operatorname{sd}(X)$, and let $T\langle A\rangle$ consist of those open cells of $\operatorname{sd}(X)$ whose closures intersect $A$.

    Then clearly $A\subseteq T\langle A\rangle$. To see that $A$ is a deformation retract of $T\langle A\rangle$, first take any simplex $\Delta$ of $\operatorname{sd}(X)$ whose interior $\sigma$ is an open cell of $T\langle A\rangle$. $\sigma$ being an open cell of $T\langle A\rangle$ means that $\Delta$ intersects the cellular subset $A$ of $X$, in particular  $\Delta\cap \operatorname{Cl}(A)$ is a non-empty subcomplex of $\Delta$. In fact, it is always a face of $\Delta$, because the vertices of $\Delta$, like those of any simplex of $\operatorname{sd}(X)$, are totally ordered, and as $\operatorname{Cl}(A)$ is a subcomplex of $X$, if a vertex of $\Delta$ is in $\operatorname{Cl}(A)$ then the simplex spanned by it and all smaller vertices is also in $\operatorname{Cl}(A)$. This implies that any point $x$ of $\sigma$ can be uniquely written as $x=sy+(1-s)z$ for some $s\in(0,1)$, $y\in\Delta\cap\operatorname{Cl}(A)$, and $z$ in the face of $\Delta$ spanned by the vertices not in $\Delta\cap\operatorname{Cl}(A)$. Using the definition of cellular subsets and the fact that $\Delta\cap A$ is non-empty, it can be shown that the relative interior of $\Delta\cap\operatorname{Cl}(A)$ is contained in $A$, so in particular $y\in A$. Define the deformation retraction of $\sigma\cup A$ to $A$ to take the point $x=sy+(1-s)z$ to $(s+t(1-s))y+(1-s)(1-t)z$ at time $t\in[0,1]$. It is straightforward to verify that these deformation retractions glue together to a strong deformation retraction of $T\langle A\rangle$ to $A$. 

    Now assume $(A_i)_{i\in I}$ is any finite ugly collection of cellular subsets of $X$, and let us show that $\bigcap_{i\in I}T\langle A_i\rangle=T\langle\bigcap_{i\in I}A_i\rangle$. The containment $\supseteq$ is clear, and for the other containment it suffices to show that if $\sigma$ is an open cell of $\operatorname{sd}(X)$ whose closure $\Delta$ intersects $\sigma_i\subseteq A_i$ for all $i\in I$, where $\sigma_i$ is an open cell of $\operatorname{sd}(X)$, then there is an open cell $\tau\subseteq\bigcap_{i\in I}A_i$ of $\operatorname{sd}(X)$ which is intersected by $\Delta$. As $\operatorname{sd}(X)$ is the barycentric subdivision of $X$, we can choose the $\sigma_i$ such that if $\sigma_i\subseteq A_i$ and $\sigma_j\subseteq A_j$ then either the closure of $\sigma_j$ intersects $\sigma_i$ or the other way around; namely take $\sigma_i:=\relint(\Delta\cap\operatorname{Cl}(A_i))$. Thus for $|I|=2$ the definition of ugliness guarantees the existence of $\tau$. Induction on $|I|$: we may assume that there is an open cell $\tau'$ of $\bigcap_{i\in I-\{i_0\}}A_i$ intersected by $\Delta$, and that $\tau'=\sigma_i$ for some $i\in I-\{i_0\}$, but then either $\tau'\subseteq A_{i_0}$ or by ugliness $\sigma_{i_0}\subseteq A_i$ for all $i\in I-\{i_0\}$, and consequently $\tau:=\tau'$ or $\sigma_{i_0}$ works.

    Finally, $\bigcup_{i\in I}T\langle A_i\rangle=T\langle\bigcup_{i\in I}A_i\rangle$ clearly holds by definition of $T\langle -\rangle$.
\end{proof}

\begin{small}

\end{small}
\end{document}